\documentclass[reqno,final,a4paper]{amsart} 
\usepackage[utf8]{inputenc}
\usepackage[notref,notcite]{showkeys} 
\usepackage[shortlabels]{enumitem}
\setitemize{itemsep=0.3em}
\setenumerate{itemsep=0.3em}
\usepackage[english]{babel}
\usepackage{xcolor}
\usepackage[T1]{fontenc}
\usepackage[normalem]{ulem}
\usepackage{caption}

\makeatletter
\def\subsubsection{\@startsection{subsubsection}{3}%
  \z@{3pt}{-\fontdimen2\font}%
  {\normalfont\bfseries}}
\def\paragraph{\@startsection{paragraph}{4}%
  \z@{3pt}{-\fontdimen2\font}%
  {\normalfont\bfseries}}
\makeatother

\usepackage{amssymb}
\usepackage{amsfonts}
\usepackage[tikz]{bclogo}
\usetikzlibrary{arrows, arrows.meta}
\usepackage{fullpage}
\usepackage{lmodern}
\usepackage{textcomp}
\usepackage{mathtools, bm}
\usepackage{amssymb, bm}
\usepackage{amsmath}
\usepackage[unq]{unique}
\usepackage{comment}
\usepackage[breaklinks]{hyperref}
\usepackage{bbm}

\usepackage[numbers]{natbib}
\usepackage{graphicx} 

\usepackage{amsthm}
\usepackage{tikz}
\usepackage{caption}
\usepackage{subcaption}
\usepackage{cleveref}
\usepackage[notref,notcite]{showkeys} 

\newtheorem{theorem}{Theorem}[section]
\newtheorem{lemma}[theorem]{Lemma}
\newtheorem{proposition}[theorem]{Proposition}

\newtheorem{claim}[theorem]{Claim}

\theoremstyle{definition}
\newtheorem{remark}[theorem]{Remark}

\newtheorem{algorithm}[theorem]{Algorithm}

\newcommand{\bbE}{\mathbb{E}}

\newcommand{\bbP}{\mathbb{P}}

\newcommand{\bbR}{\mathbb{R}}

\newcommand{\bbT}{\mathbb{T}}

\newcommand{\cA}{\mathcal{A}}
\newcommand{\cB}{\mathcal{B}}
\newcommand{\cC}{\mathcal{C}}

\newcommand{\cG}{\mathcal{G}}

\newcommand{\cI}{\mathcal{I}}
\newcommand{\cJ}{\mathcal{J}}

\newcommand{\cM}{\mathcal{M}}
\newcommand{\cN}{\mathcal{N}}

\newcommand{\cP}{\mathcal{P}}
\newcommand{\cQ}{\mathcal{Q}}
\newcommand{\cR}{\mathcal{R}}
\newcommand{\cS}{\mathcal{S}}
\newcommand{\cT}{\mathcal{T}}

\newcommand{\fN}{\mathfrak{N}}
\newcommand{\fR}{\mathfrak{R}}

\newcommand{\bX}{\mathbf{X}}

\newcommand{\bY}{\mathbf{Y}}

\newcommand{\taukcone}[1]{\tau_{1,#1}}
\newcommand{\taukctwo}[1]{\tau_{2,#1}}

\newcommand{\eps}{\varepsilon}
\newcommand{\e}{\mathrm{e}}

\newcommand{\vol}[1]{\mathrm{Vol}_{#1}}
\newcommand{\dist}{\mathrm{dist}}
\newcommand{\diam}{\mathrm{diam}}
\newcommand{\sat}{\mathrm{sat}}

\definecolor{sea}{RGB}{47, 180, 241}
\definecolor{close}{RGB}{255, 255, 255}
\definecolor{far}{RGB}{255, 0, 0}
\definecolor{geometry_green}{RGB}{16, 207, 129}

\title{Sharp thresholds, hitting times and the power of choice for random geometric graphs}

\author[Ignasiak]{Dawid Ignasiak}

\author[Lichev]{Lyuben Lichev}

\address{TU Wien, Wiedner Hauptstra\ss e 8-10, A-1040 Vienna, Austria}
\email{[dawid.ignasiak|lyuben.lichev]@tuwien.ac.at}

\begin{document}

\begin{abstract}
We consider a random geometric graph process where random points $(X_i)_{i\ge 1}$ are embedded consecutively in the $d$-dimensional unit torus $\mathbb T^d$, and every two points at distance at most $r$ form an edge.
As $r\to 0$, we confirm that well-known hitting time results for $k$-connectivity (with $k\ge 1$ fixed) and Hamiltonicity in the Erd\H{o}s-R\'enyi graph process also hold for the considered geometric analogue.
Moreover, we exhibit a sort of probabilistic monotonicity for each of these properties. 

We also study a geometric analogue of the power of choice where, at each step, an agent is given two random points sampled independently and uniformly from $\mathbb{T}^d$ and must add exactly one of them to the already constructed point set.
When the agent is allowed to make their choice with the knowledge of the entire sequence of random points (offline 2-choice), we show that they can construct a connected graph at the first time $t$ when none of the first $t$ pairs of proposed points contains two isolated vertices in the graph induced by $(X_i)_{i=1}^{2t}$, and maintain connectivity thereafter by following a simple algorithm.
We also derive analogous results for $k$-connectivity and Hamiltonicity. 
This shows that each of the said properties can be attained two times faster (time-wise) and with four times fewer points in the offline 2-choice process compared to the 1-choice process.

In the online version where the agent only knows the process until the current time step, we show that $k$-connectivity and Hamiltonicity cannot be significantly accelerated (time-wise) but may be realised on two times fewer points compared to the 1-choice analogue.
\end{abstract}

\maketitle



\section{Introduction}

More than 60 years ago, Gilbert introduced two important models of random graphs~\cite{Gil59,Gil61}.
The first of them, the \emph{binomial random graph} $G(n,p)$, is nowadays most often associated with Erd\H{o}s and R\'enyi~\cite{ER59,ER60} who introduced and analysed its analogue $G(n,m)$ independently of Gilbert.
The second model, widely known as \emph{Gilbert's model} or the \emph{random geometric graph}, is obtained by embedding $n$ random points in a compact metric space and constructing an edge between every pair at distance at most a given threshold length $r$.
Originally introduced as a model for sensor networks, when $r$ is small, random geometric graphs exhibit features closer to percolation on lattices compared to the `mean field' behaviour of $G(n,p)$.
For a detailed treatment of the model, see the book of Penrose~\cite{Pen04}.

One of the main research directions in the theory of random graphs is dedicated to the analysis of their dynamic analogues.
In the world of binomial random graphs where randomness is carried by the set of edges, the corresponding random process is defined naturally via ordering the edges uniformly at random and sequentially including them in the graph.
A classic geometric analogue of the above process consists in sampling independently $n$ random points in a compact metric space (usually the unit hypercube $[0,1]^d$ or its boundary-free analogue, the unit torus $\mathbb{T}^d$) and gradually increasing the radius $r$, thus connecting points at smaller distances earlier in the process.
In a pioneering work, Penrose~\cite{Pen97} showed that the dynamic random geometric graph on $[0,1]^d$ for $d\in \{1,2\}$ or $\mathbb{T}^d$ for $d\ge 1$ becomes connected when the last isolated vertex disappears. 
Later, he extended and generalised this result to $k$-connectivity of the dynamic random geometric graph on $[0,1]^d$ for arbitrary $k,d\ge 1$~\cite{Pen04}.
Afterwards, D\'iaz, Mitsche and P\'{e}rez~\cite{DMP07} proved a sharp threshold result for Hamiltonicity in the random geometric graph on $[0,1]^2$. 
Their theorem was later strengthened to a hitting time result in every dimension $d\ge 1$ by Balogh, Bollob\'as, Krivelevich, M\"uller and Walters~\cite{BBKMW11} and M\"uller, P\'erez-Gim\'enez and Wormald~\cite{MPW11}.

Motivated by growing data sets and adaptive algorithms for their processing, in this work, we take a different perspective and define the dynamic random geometric graph via a sequence of random points $(X_i)_{i\ge 1}$ embedded consecutively in the torus $\mathbb T^d$.\footnote{The choice of the torus as an ambient space is done mostly for convenience. We believe that analogues of our results also hold for the unit cube $[0,1]^d$ with some modifications.}
Contrary to~\cite{BBKMW11,DMP07,MPW11,Pen04,Pen97}, our asymptotic parameter is the radius of connectivity $r\to 0$.
This alternative point of view allows to see the model as a discrete-time stochastic process and raises related questions.
Some of the latter may be seen as `dual' to the questions considered in the dynamic model with increasing radius: 
for example, it is natural to expect that analogues of the hitting time results for $k$-connectivity and Hamiltonicity from~\cite{BBKMW11,MPW11,Pen04,Pen97} still hold for every dimension $d\ge 1$. 
To confirm these results, our analysis overcomes an important additional obstruction: unlike in the model with increasing radius, in our model, connectivity and Hamiltonicity are \emph{not} monotone properties.
For example, the point $X_t$ could land at distance more than $r$ from each of $(X_i)_{i=1}^{t-1}$, and thus become isolated, even if the geometric graph on $(X_i)_{i=1}^{t-1}$ is already Hamiltonian.
We show that the said scenario is atypical by exhibiting a sort of `probabilistic monotonicity' for each of the discussed properties.

While the dynamic random geometric graph with growing number of vertices exhibits phenomena already established in the dynamic model of Penrose~\cite{Pen97},
some natural questions seem to have no clear analogues in this original model.
For example, our discrete-time point of view on the model allows to analyse it through the lens of the classic \emph{power of choice} paradigm introduced 
by Vvedenskaya, Dobrushin and Karpelevich~\cite{VDK96} and Mitzenmacher~\cite{Mit96} in the context of load balancing, and further developed by Azar, Broder, Karlin and Upfal~\cite{ABKU99} and by Mitzenmacher~\cite{Mit01}.
They showed that, in a process with $n$ balls and $n$ bins where two random bins are proposed at every step and an \emph{agent} places a ball in the less full of them, the maximum number of balls in a bin is considerably smaller compared to the classical purely random setting. 
Many variations of this basic model have been analysed~\cite{BCSV06, KP06, Mit01, Red06}. The power of choice in the context of random graphs was introduced by Achlioptas. 
He was interested in the question of accelerating or delaying certain monotone graph properties with respect to the original Erd\H{o}s-R\'enyi process if two edges are proposed at each step (instead of one) and the agent may choose the one that better suits their goal.
This resulted in a rich line of research with focus on the problem of delaying the appearance of a giant component~\cite{ADS09,BF01,RW12,SW07} but delaying or accelerating connectivity~\cite{ELMPS16}, avoidance of fixed subgraphs~\cite{KLS09}, containment of a perfect matching and Hamiltonicity~\cite{KLS10} have also been analysed.
Outside of the graph setup, delaying the satisfiability of a random $k$-SAT formula has also received some attention~\cite{Per15,SV13}.

We analyse a geometric version of the power of choice introduced by M\"{u}ller and Sp\"{o}hel~\cite{MS15}.
In this setting, at each round, the agent is offered two uniformly chosen points and is required to include exactly one of them in the geometric graph either based on complete information of the sequence $(X_i)_{i\ge 1}$ (offline 2-choice) or only based on the history of the process (online 2-choice).
M\"{u}ller and Sp\"{o}hel showed that the 2-dimensional random geometric graph exhibits a threshold behaviour with respect to the appearance of a linear-sized component. 
In the online 2-choice process with $n$ rounds, they confirmed that a linear-sized component typically exists when $r^3n\log\log n = \Omega(1)$ irrespectively of the strategy of the agent, and such components can be avoided by the agent when $r^3n\log\log n = o(1)$. 
The same phenomenon holds for the offline version around $r^3n=\Theta(1)$.
In our work, we analyse the geometric power of choice from a point of view coming closer to~\cite{ELMPS16,KLS10}, namely, we establish how fast the agent can attain a $k$-connected or a Hamiltonian geometric graph compared to the original (1-choice) process.
More precisely, by providing new hitting time results (for offline 2-choice) and sharp threshold results (for online 2-choice), we give precise estimates on how much the said properties can be accelerated.
In each of the two settings, constructing and maintaining a $k$-connected or a Hamiltonian graph after the corresponding hitting time/sharp threshold is done by analysing a linear-time probabilistic algorithm.

\subsection{Notation}\label{sec:notation} 
We use mostly standard notation. The more experienced reader may wish to consult our model-specific notation and skip to \Cref{sec:main_res}.

\vspace{0.4em}

\noindent
\textbf{General notation.} For an integer $n\ge 1$, we define $[n] = \{1,\ldots,n\}$. For real numbers $a,b$ and $c>0$, we write $a=b\pm c$ to denote the fact that $a\in [b-c,b+c]$. Rounding is often omitted where irrelevant.

\vspace{0.4em}

\noindent
\textbf{Geometric notation.} 
We work with the unit torus $\bbT^d$ equipped with the metric inherited from the Euclidean distance on $\mathbb R^d$.
For a point $x\in \bbT^d$ and $r\in (0,1/2)$, we denote by $B(x,r)$ the ball with centre $x$ and radius $r$.
Moreover, for $x\in \bbT^d$ and $0<r_1<r_2<1/2$, we denote by $A(x,r_1,r_2) = B(x,r_2)\setminus B(x,r_1)$ the \emph{annulus with centre $x$ and radii $r_1,r_2$}.
Also, we denote by $\theta_d$ the volume (i.e.\ Lebesgue measure) of the unit $d$-dimensional ball.
The volume of a measurable subset $U$ of $\bbT^d$ is denoted $\vol{d}(U)$.
The distance between two sets $U,W\subseteq \bbT^d$ is denoted $\dist(U,W)$. When some of the sets is reduced to a single point, curly braces are omitted: for example, we write $\dist(u,w)$ instead of $\dist(\{u\},\{w\})$. 
The diameter of a set $U\subseteq \mathbb T^d$ is denoted $\diam(U)$ and is equal to the supremum over the distances $\dist(u,v)$ with $u,v\in U$.
For sets $A, B \subseteq \mathbb R^d$, $A \oplus B$ denotes the \emph{Minkowski sum} of sets $A$ and $B$ defined as $\{ a + b: a \in A, b \in B \}$.

\vspace{0.4em}

\noindent
\textbf{Graph-theoretic notation.}
For a graph $G$, we write $V(G)$ and $E(G)$ to denote its sets of vertices and edges, respectively.
For a vertex subset $U$, we denote by $N_G(U)$ the set of vertices outside $U$ but adjacent to a vertex in $U$, and $N_G[U] = N_G(U)\cup U$.
In the same setting, we denote by $G[U]$ the subgraph of $G$ induced by $U$.
The degree of a vertex $v$ in $G$ is denoted $d_G(v)$, and maximum and minimum degrees in $G$ are denoted respectively $d_{\min}(G)$ and $d_{\max}(G)$.
For better readability, the index $G$ is omitted in $N_G$ and $d_G$ when the graph is clear from the context.

\vspace{0.4em}

\noindent
\textbf{Model-specific notation.}
Fix $d\ge 1$, a sequence $(X_i)_{i=1}^{\infty}$ of independent random points distributed uniformly in the $d$-dimensional unit torus $\mathbb{T}^d = \mathbb R^d/\mathbb Z^d$, and a \emph{radius of connectivity} $r > 0$.
For every $i\ge 1$, the points $X_{2i-1}$ and $X_{2i}$ are called \emph{partners} or a \emph{partner pair}.
For every point $Y$ among $(X_i)_{i = 1}^\infty$, we write $\Bar{Y}$ for the partner of $Y$.
For a finite set of points $S$, we denote by $G(S,r)$ (or simply $G(S)$ when the radius is clear from the context) the geometric graph on vertex set $S$ where every pair of points at distance at most $r$ forms an edge.
Moreover, for every $t\ge 1$, we write $G_t = G((X_i)_{i=1}^t,r)$ for short.
Our asymptotic parameter is $r$: in particular, our asymptotic notation is used exclusively with respect to $r\to 0$. 
The implicit constants in the usual $O,\Theta,\Omega$-notation are allowed to depend on the desired connectivity $k\ge 1$ and the dimension of the ambient space $d\ge 1$.

Conditionally on the sequence $(X_i)_{i=1}^{\infty}$, a \emph{choice set} is a subset $CS$ of $\{X_i: i\ge 1\}$ containing exactly one point from each partner pair.
Fix $\mathbb{T}^* = \bigcup_{k=0}^{\infty} (\mathbb{T}^d)^k$.
We consider two variations of the online 2-choice process. 
In the \emph{simultaneous} variation, the points in every pair are revealed at once. Then, the agent is allowed to add one of the points in a partner pair $x,y$ to the already constructed set $W$ according to a \emph{simultaneous choice function}, that is, a measurable function $f: \mathbb{T}^*\times (\mathbb{T}^d)^2\to \{1,2\}$ where $f(W,x,y)=1$ indicates adding $x$ to $W$ and $f(W,x,y)=2$ indicates adding $y$.
In the \emph{consecutive} online 2-choice process, the agent needs to take a decision based on the already constructed set $W$ and the first point in the current pair $x,y$ according to a \emph{consecutive choice function}, that is, a measurable function $g: \mathbb{T}^*\times \mathbb{T}^d\to \{1,2\}$ where $g(W,x)=1$ indicates adding $x$ to $W$ and $g(W,x)=2$ indicates adding $y$.
For every $t\ge 1$, a choice set $CS$ and a (simultaneous or consecutive) choice function $f$, we denote $\mathbf{X}_t(CS) = \{X_i: i\in [2t]\}\cap CS$ and $\mathbf{X}_t(f)$ the subset of $\{X_i: i\in [2t]\}$ produced by the choice function~$f$.

\vspace{0.4em}

\noindent
\textbf{Probabilistic notation.}
For a family of events $(A_r)_{r > 0}$ indexed by the radius of connectivity $r$ and measurable in terms of $(X_i)_{i=1}^{\infty}$,
we say that $(A_r)_{r > 0}$ (or, more often, $A_r$ itself) holds \emph{with high probability} or \emph{whp} if $\mathbb P(A_r)\to 1$ as $r\to 0$.
The complement of an event $A$ in the $\sigma$-algebra generated by $(X_i)_{i=1}^{\infty}$ is denoted by $A^c$.
We denote by $\mathrm{Bin}(n,p)$ the binomial distribution with parameters $n$ and $p$, and by $\mathrm{Po}(\lambda)$ the Poisson distribution with parameter $\lambda > 0$.

\subsection{Main results}\label{sec:main_res}

The first main theorem establishes hitting time results for $k$-connectivity and Hamiltonicity in our dynamic random geometric graph, and ensures that these properties are monotone in a probabilistic sense.
Define
\begin{alignat*}{3}
\taukcone{k} &= \taukcone{k}(r) &&:= \min\{t\ge 1: G_t \text{ has no vertex of degree at most }k-1\}.
\end{alignat*}

\begin{theorem}\label{thm:1}
For every $k\ge 1$, with high probability, for every $t\ge \taukcone{k}$, the graph $G_t$ is $k$-connected. 
Also, with high probability, for every $t\ge \taukcone{2}$, the graph $G_t$ is Hamiltonian.
\end{theorem}

Our second result provides an analogue of \Cref{thm:1} for the offline 2-choice process. 
For every $k\ge 1$, define $\taukctwo{k} = \taukctwo{k}(r)$ to be the smallest $t\ge 1$ such that, for every $i\in [t]$, $X_{2i-1}$ and $X_{2i}$ do not simultaneously have degree in the interval $[0,k-1]$ in $G_{2t}$.
This minimum degree constraint ensures that the graph $G(\mathbf{X}_t(CS))$ cannot be $k$-connected for any choice set $CS$ when $t < \taukctwo{k}$, and cannot be Hamiltonian when $t < \taukctwo{2}$. We show that these necessary conditions for $k$-connectivity and Hamiltonicity are typically also sufficient.

\begin{theorem}\label{thm:2}
For every $k\ge 1$, whp there is a choice set $CS$ such that the graph $G(\mathbf{X}_t(CS))$ is $k$-connected for every $t\ge \taukctwo{k}$.
Furthermore, whp there is a choice set $CS$ such that the graph $G(\mathbf{X}_t(CS))$ is Hamiltonian for every $t\ge \taukctwo{2}$.
\end{theorem}

\begin{remark}\label{rem:main}
Deriving precise asymptotic expressions for the typical values of the hitting times $\taukcone{k},\taukctwo{k}$ in terms of $r$ is not complicated (see \Cref{lem:props,lem:hitting-time-offline-lb} and their proofs) and allows to quantitatively compare the 1-choice model and the offline 2-choice model.
Namely, for every $k\ge 1$, whp $\taukcone{k} = (2+o(1))\taukctwo{k}$.
In particular, there exists a choice set which allows the agent to construct well-connected and Hamiltonian graphs in the offline 2-choice process on roughly four times fewer points compared to the 1-choice process.
\end{remark}

We turn to the online 2-choice process. 
Our last main theorem ensures that the performance of each of the two variations of online 2-choice introduced in \Cref{sec:notation} is `midway' between the 1-choice process and the offline 2-choice process. 
More precisely, we show that obtaining $k$-connectivity or Hamiltonicity typically requires processing roughly the same number of points as in the 1-choice process but this still ensures that the number of vertices in the underlying graph remains roughly twice as small.

\begin{theorem}\label{thm:3}
Whp, for every simultaneous choice function $f$ and $\eps > 0$, $G(\mathbf{X}_t(f))$ remains disconnected for every $t\ge 0$ such that $2t\le \mathbb E[\taukcone{1}]- (\theta_d^{-1}+\eps) r^{-d}\log\log(1/r)$.
Moreover, there is a constant $C = C(d)$ and a consecutive choice function $g$ such that whp each of the following holds: 
\begin{enumerate}[label=(\roman*)]
    \item\label{item:A1} for every $k\ge 1$, $G(\mathbf{X}_t(g))$ is $k$-connected for every $t$ such that $2t\ge \mathbb E[\taukcone{k}] + kC r^{-d}\log\log(1/r)$,
    \item\label{item:A2} $G(\mathbf{X}_t(g))$ is Hamiltonian for every $t$ such that $2t\ge \mathbb E[\taukcone{2}] + 2C r^{-d}\log\log(1/r)$.
\end{enumerate}
\end{theorem}
We observe that the additional number of steps above $\mathbb E[\taukcone{k}]$ required in \ref{item:A1} and \ref{item:A2} is comparable to the second-order term in $\mathbb E[\taukcone{k}]$, thus ensuring some control on the size of the critical window for each of the relevant properties.

\noindent

\paragraph{Outline of the proofs.} 
First of all, we outline the ideas behind \Cref{thm:1} starting with the $k$-connectivity part.
We consider the 1-choice process at time $t_* = t_*(k) = \mathbb E[\taukcone{k}] - h(r)/r^d$ for a function $h(r)$ tending to infinity suitably slowly.
A second moment computation shows that whp, for every $t\le t_*$, the graph $G_t$ has a vertex of degree at most $k-1$, and $G_{t_*}$ still has $\omega(1)$ of them (\Cref{lem:conc}).
Moreover, by using results of Penrose~\cite{Pen04}(\Cref{thm:Pen04}) and our choice of $h$, we derive that, at time $t_*$, whp the random geometric graph consists of a giant $k$-connected component and few vertices of degree $k-1$ located away from each other (\Cref{lem:props}(d) and (e)).

It remains to ensure that whp each of the vertices of degree $k-1$ at time $t_*$ is incorporated in the giant $k$-connected component by time $\taukcone{k}$, and that no new vertices of degree $k-1$ ever appear.
The first task is rather standard but the second is less routine: it requires to track the region where, if a new random point lands, it will connect to no more than $k-1$ of the present vertices (later called the \emph{remote} region). 
While this region becomes empty around (though possibly shortly after) time $\taukcone{k}$, volume considerations are not sufficient to show this: indeed, the remote region could be simultaneously small and rather uniformly spread around the torus, meaning that its complete disappearance could take much longer than $\taukcone{k}$.
To overcome the latter problem, we introduce a tessellation $\mathcal T$ of the torus $\mathbb{T}^d$ into cubes of volume $\Theta(r^d/\log(1/r))$ and show that typically the remote region can be covered by a `small' number of cubes in $\mathcal T$ (\emph{remote} cubes).
(In fact, the total volume of the bad cubes and the expected proportion of vertices of degree $k-1$ turn out to be comparable.)
Finally, we guarantee that the remote cubes typically group into small clusters which remain at distance at least $3r$ from each other (\Cref{cl:far_cubes}), and each of them is surrounded by many vertices in the giant $k$-connected component at distance $r+o(r)$ from themselves (similar statement holds for the vertices of degree $k-1$). 
As a result, whp vertices of degree $k-1$ become absorbed by the giant $k$-connected component by time $\taukcone{k}$, and the entire remote region becomes within reach of this giant before being visited for the first time, which implies the $k$-connectivity part of \Cref{thm:1}.
The Hamiltonicity part follows by combining the established result for $2$-connectivity and (minor adaptations of) several structural lemmas from~\cite{BBKMW11}.

\vspace{1em}

The proof of \Cref{thm:2} is the most involved part and contains the main technical innovations of this work.
The additional difficulty comes from the fact that typically there are about $\sqrt{\taukctwo{k}}$ small connected subgraphs in $G_{2\taukctwo{k}}$, each with no more than $k-1$ neighbours.
With high probability, almost all such subgraphs consist of a single vertex of degree $k-1$: pairs of such vertices are the `main obstruction' for $k$-connectivity, which intuitively justifies why the hitting time result for $k$-connectivity in \Cref{thm:2} is to be expected.
However, confirming this requires the analysis of significantly larger `islands of remote cubes'.
In particular, contrary to \Cref{thm:1}, the vertices in separate remote connected components do not necessary induce cliques in $G_{2\taukctwo{k}}$, which requires non-trivial distinction with respect to the structure of each component (\Cref{algo:U1}).
Nevertheless, our remote components have two crucial typical properties which simplify our task: each of them contains at most one `dangerous' vertex whose partner is near the remote region (\Cref{lem:special}(a)) and no remote component with diameter more than $r/10$ contains any such vertex (\Cref{lem:special}(b)).
Describing which vertices in components without a `dangerous' vertex enter the choice set is a simpler task: their partners fall in the middle of dense spots in the giant and, since these partner points do not form large clusters, $k$-connectivity of the graph is preserved when these partners are deleted.
Describing our choice in remote components containing a vertex whose partner is near (or in) another remote component is more delicate and relies on a careful structural analysis (\Cref{lem:algo_good}): this part contains several novel ideas of our work.
Once the vertices near the remote area are dealt with, the construction of the choice set at the hitting time is easily extended for `points in dense regions' of the giant component.
For the probabilistic monotonicity part of the result, we show that the remote area can be `eaten up from outside' in an online fashion before being visited again. While the idea is similar to that of \Cref{thm:1}, the approach requires a bit more care close to $\taukctwo{k}$.
The proof of the Hamiltonicity part of \Cref{thm:2} requires only minor changes in the strategy from the proof for $k$-connectivity, and in the justification from \Cref{thm:1}.

\vspace{1em}

Finally, the proof of \Cref{thm:3} relies on two main ideas. 
For the first sentence of \Cref{thm:3}, we show that $(\theta_d^{-1}+\eps)r^{-d}\log\log(1/r)$ steps before the (expected) hitting time for connectivity, the graph still contains $\omega(\log(1/r))$ isolated vertices with high probability. 
Furthermore, typically $\omega(1)$ of these vertices arrive at a point in the process when each of them and their neighbours is originally isolated.
As a result, at the arrival of such a pair, deciding which of the two partner vertices is going to stay isolated for longer is impossible and each such step presents $50\%$ chance for choosing the `wrong' vertex.

The bounds \ref{item:A1} and \ref{item:A2} come down to analysing a particular consecutive choice function $g$ constructed as follows. First, $g$ outputs $1$ at each of the first $C r^{-d}\log\log(1/r)$ steps for a suitably large constant $C = C(d)$.
At this point, the constructed random geometric graph already has one dense giant component containing almost all vertices, and few sparse areas surrounding small `remote islands' far from each other.
Once the above structure is established, the $k$-connectivity of the graph only needs to be `repaired' near the small remote islands.
From this point on, the function $g$ selects a point landing near those islands at every step when at least one such point is presented.
However, the probability that two partner vertices simultaneously land near the remote area (so each of them is desirable) is asymptotically much smaller compared to the probability that only one vertex in a partner pair lands near the remote area (in which case the consecutive choice function $g$ chooses this vertex).
As a result, the online 2-choice process and the 1-choice process typically `run with the same speed' near the remote region. Quantifying this statement is the main technical innovation in this part of our work and is sufficient to conclude; formally, our argument goes via a coupling of the online 2-choice process and the 1-choice process (defined before \Cref{cl:coupleXY}).

\vspace{1em}

\noindent
\paragraph{Plan of the paper.} In \Cref{section:prelims}, we state and show some preliminary results. \Cref{section:1-choice,section:offline-2-choice,section:online-2-choice} contain the proofs of \Cref{thm:1,thm:2,thm:3}, respectively. We conclude with a few comments and open problems in \Cref{section:conclusion}.

\section{Preliminaries}\label{section:prelims}

\subsection{A concentration inequality}

We first present a version of the Chernoff bound for binomial and Poisson random variables, see e.g.\ Theorems~A.1.12 and~A.1.15 in \cite{AS16}.

\begin{lemma}[Chernoff's bound]\label{lem:chernoff}
For a binomial or Poisson random variable $X$ with mean $\mu$ and any $\varepsilon>0$,
\[\mathbb P(X\le (1-\varepsilon)\mu)\le \exp(-\varepsilon^2\mu/2)\qquad\text{and}\qquad \mathbb P(X\ge (1+\varepsilon)\mu)\le (\e^\varepsilon (1+\varepsilon)^{-(1+\varepsilon)})^\mu.\]
\end{lemma}

\subsection{Geometric preliminaries}

In our considerations, we make use of the following volume computation. 

\begin{lemma}\label{lem:vols}
Fix $\nu = \nu(r) = o(r)$. Then, there is a constant $C = C(d) > 0$ such that the volume of the union $U$ of two Euclidean balls with radii equal to $r$ and centres at distance $\nu$ is $\theta_d r^d + (C+o(1))\nu r^{d-1}$.
\end{lemma}
\begin{proof}
Denote by $\Gamma$ the Gamma function and set $\iota = \iota(r) = \arccos(\nu/2r)$ (see \Cref{fig:lem vols}).
By using the formula for the volume of hyperspherical caps (see e.g.\ \cite{Li10}), the volume of $U$ is given by
\[\frac{\pi^{(d-1)/2} r^d}{\Gamma((d+1)/2)}\cdot 2\bigg(\int_{u=0}^{\pi}(\sin u)^{d} du - \int_{u=0}^{\iota}(\sin u)^{d} du\bigg) = \theta_d r^d + \frac{\pi^{(d-1)/2} r^d}{\Gamma((d+1)/2)}\int_{u=\iota}^{\pi-\iota}(\sin u)^{d} du.\]
Using that $\iota = \pi/2-(1+o(1))\nu/2r$ and $\sin(u) = 1-o(1)$ for all $u\in [\iota,\pi-\iota]$ shows that
\[\frac{\pi^{(d-1)/2} r^d}{\Gamma((d+1)/2)}\int_{u=\iota}^{\pi-\iota}(\sin u)^{d} du = (1+o(1))\frac{\pi^{(d-1)/2} \nu r^{d-1}}{\Gamma((d+1)/2)}\]
and completes the proof.
\end{proof}

Next, we state the Brunn-Minkowski inequality, see e.g.\ Theorem~5.11 in \cite{Pen04}.
\begin{lemma}\label{lem:brunn-minkowski}
For every two non-empty compact sets $A, B\subseteq \bbR^d$,
    \[ \vol{d}(A \oplus B) \geq (\vol{d}(A)^{1/d} + \vol{d}(B)^{1/d})^d. \]
\end{lemma}

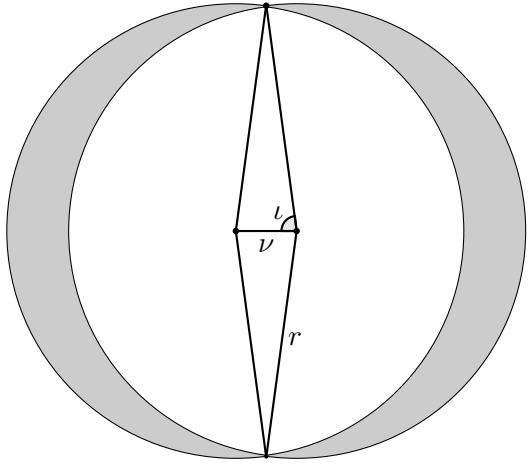
\begin{figure}
\centering
\begin{tikzpicture}[line cap=round,line join=round,x=1cm,y=1cm,scale=0.4]
\clip(-9.7,-9.2) rectangle (10,8.2);

\draw [shift={(0.5,-0.5)},line width=0.8pt,color=black,fill=black,fill opacity=0.1] (0,0) -- (96.5:0.5) arc (96.5:180:0.5) -- cycle;

\draw [line width=0.8pt] (-1.5,-0.5) circle (7.516648189186455cm);
\draw [line width=0.8pt] (0.5,-0.5) circle (7.516648189186455cm);

\fill[gray!40,even odd rule] (0.5,-0.5) circle (7.516648189186455cm) (-1.5,-0.5) circle (7.516648189186455cm);

\draw [line width=0.8pt] (-1.5,-0.5)-- (0.5,-0.5);
\draw [line width=0.8pt] (-1.5,-0.5)-- (-0.5,7);
\draw [line width=0.8pt] (-0.5,7)-- (0.5,-0.5);
\draw [line width=0.8pt] (-1.5,-0.5)-- (-0.5,-8);
\draw [line width=0.8pt] (0.5,-0.5)-- (-0.5,-8);
\begin{scriptsize}
\draw [fill=black] (-1.5,-0.5) circle (2.5pt);
\draw [fill=black] (-0.5,6.97) circle (2.5pt);
\draw [fill=black] (0.5,-0.5) circle (2.5pt);
\draw[color=black] (-0.5,-1) node {\large{$\nu$}};
\draw [fill=black] (-0.5,-7.97) circle (1.5pt);
\draw[color=black] (0.45,-4.0590303805683545) node {\large{$r$}};
\draw[color=black] (-0.1,0.1) node {\large{$\iota$}};
\end{scriptsize}
\end{tikzpicture}
\caption{An illustration from the proof of \Cref{lem:vols} when $d=2$.}
\label{fig:lem vols}
\end{figure}

\subsection{Preliminaries on random geometric graphs}\label{subsec:prelim-rgg}

Next, we derive concentration of the number of vertices of degree $\kappa \ge 0$ in a random geometric graph via a second moment computation.
For every $\kappa, \kappa',n\ge 1$ with $\kappa\le \kappa'$, denote by $Z_{[\kappa,\kappa'],n}$ the number of vertices in $G_n$ of degree in the interval $[\kappa,\kappa']$. When $\kappa=\kappa'$, we simply write $Z_{\kappa,n}$.
In particular,
\begin{equation}\label{eq:expZ}
\mathbb E[Z_{\kappa,n}] = n \binom{n-1}{\kappa} (\theta_d r^d)^{\kappa} (1-\theta_d r^d)^{n-1-\kappa}.
\end{equation}
While the following lemma is relatively standard, we have not found an appropriate reference and provide a proof for completeness.

\begin{lemma}\label{lem:conc}
Fix $\varepsilon > 0$, integers $\kappa\ge 0$ and $n = n(r)\ge 1$ such that $\mathbb E[Z_{\kappa,n}]\to \infty$. 
Then, whp we have $Z_{\kappa,n} = (1\pm \eps) \mathbb E[Z_{\kappa,n}]$.
\end{lemma}
\begin{proof}
We use the second moment method. 
To this end, note that~\eqref{eq:expZ} gives $\mathbb E[Z_{\kappa,n}] = \Theta(n(nr^d)^{\kappa} \e^{-\theta_d nr^d})$, and the divergence condition in the statement implies that $nr^d < 3d\log(1/r)/(2\theta_d)$.
Denote by $\cA$ the event that each of the first two vertices in $G_n$ has degree $\kappa$ in $G_n$.
While $\mathbb E[Z_{\kappa,n}]$ has been computed in~\eqref{eq:expZ}, we observe that
\begin{equation}\label{eq:sndmom}
\mathbb E[Z_{\kappa,n}^2] = \mathbb E[Z_{\kappa,n}]+n(n-1)\mathbb P(\cA).
\end{equation}

It remains to bound $\mathbb P(\cA)$ from above. Given $0\le a<b\le \infty$, denote by $\cB(a,b)$ the event that $X_1$ and $X_2$ land greatr than $a$ and at most $b$ from each other.
Then,
\begin{equation}\label{eq:Bn1}
\mathbb P(\cA\cap \cB(2r,\infty))\le \binom{n-2}{\kappa}\binom{n-2-\kappa}{\kappa} (\theta_d r^d)^{2\kappa} (1-2\theta_d r^d)^{n-2-2\kappa} = (1+o(1))\frac{\mathbb E[Z_{\kappa,n}]^2}{n^2}.
\end{equation}
Next, denote by $\cA'$ the event that the set $\{X_1, X_2 \}$ has at most $2\kappa$ neighbours in $G_n$. We have
\begin{equation}\label{eq:Bn2}
\begin{split}
\mathbb P(\cA\cap \cB(r,2r))
& \le \bbP(\cA' \cap \cB(r, 2r))
\le \theta_d (2r)^d\cdot \sum_{i=0}^{2\kappa} \binom{n-2}{i} (2\theta_d r^d)^i (1-3\theta_d r^d/2)^{n-2-2\kappa}\\
&\le 2 \theta_d (2r)^d\cdot (2\theta_d nr^d)^{2\kappa} \exp(-3\theta_d nr^d/2) \\
& = O \left( \frac{n^2 (nr^d)^{2\kappa} \exp(-2\theta_d nr^d)}{n^2} \cdot r^d \exp(\theta_d nr^d/2) \right) = o(\mathbb E[Z_{\kappa,n}]^2/n^2),
\end{split}
\end{equation}
where we used that, on the event $\cB(r,2r)$, the union of the balls with radius $r$ around $X_1$ and $X_2$ has volume between $(3\theta_d/2) r^d$ and $2\theta_d r^d$. 

Finally, when the points $X_1$ and $X_2$ are at distance smaller than $r$, they are adjacent in $G_n$.
Denote by $\cA''$ the event that $X_1$ has exactly $\kappa - 1$ neighbours other than $X_2$ in $G_n$. Then,
\begin{equation}\label{eq:Bn4}
\begin{split}
\mathbb P(\cA\cap \cB(0, r))
& \le \bbP(\cA'' \cap \cB(0, r))
\leq \theta_d r^d \cdot \binom{n - 2}{\kappa - 1} \cdot (\theta_d r^d)^{\kappa - 1} \cdot (1 - \theta_d r^d)^{n - \kappa - 1} \\
& \leq \theta_d r^d \cdot (\theta_d n r^d)^{\kappa - 1} \cdot \exp( - \theta_d nr^d )
= O(\bbE[Z_{\kappa, n}] / n^2) = o(\bbE[Z_{\kappa, n}]^2 / n^2),
\end{split}
\end{equation}
where the last equality uses the assumption that $\bbE[Z_{\kappa, n}] \rightarrow \infty$.
Then, combining~\eqref{eq:Bn1},~\eqref{eq:Bn2} and~\eqref{eq:Bn4} yields
\[\mathbb P(\cA) = \mathbb P(\cA\cap \cB(2r,\infty))+\mathbb P(\cA\cap \cB(r,2r))+\mathbb P(\cA\cap \cB(0,r)) = (1+o(1))\mathbb E[Z_{\kappa, n}]^2/n^2.\]
Finally, the latter relation together with~\eqref{eq:sndmom} and Chebyshev's inequality finishes the proof.
\end{proof}

For an integer $\kappa\ge 0$ and a graph $G$, a non-empty set $S\subseteq V(G)$ is called \emph{$\kappa$-separated (in $G$)} if $S$ has at most $\kappa$ neighbours.
The following theorem combines several results from~\cite{Pen04} to describe $\kappa$-separated sets around the threshold for minimum degree $\kappa+1$.

\begin{theorem}[see Theorem 8.1, Theorem 13.11 and Propositions 13.12--13.13 in~\cite{Pen04}]\label{thm:Pen04}
Fix $\beta > 0$, integers $\kappa\ge 0$ and $n = n(r)$ such that $nr^d\to \infty$ and $\mathbb E[Z_{\kappa,n}]\to \beta$. Then,
\begin{enumerate}[label=\emph{(\alph*)}]
    \item 
    $Z_{\kappa,n}$ converges in distribution to $\mathrm{Po}(\beta)$ as $r\to 0$,
    \item 
    if $\kappa \ge 1$, whp there is no $\kappa$-separated set of size between $2$ and $n-\kappa-2$ in $G_n$,
    \item 
    if $\kappa = 0$, whp there is a unique connected component of $G_n$ with more than $1$ vertex.
\end{enumerate}
\end{theorem}

\section{\texorpdfstring{Hitting times and probabilistic monotonicity in the $1$-choice process}{}}\label{section:1-choice}

We turn our attention to \Cref{thm:1}. We first concentrate on showing the part concerning $k$-connectivity.

\subsection{\texorpdfstring{$k$-connectivity in the $1$-choice process}{}}\label{sec:k-conn 1-choice}

In this section, $k\ge 1$ is a fixed integer. 
Our first lemma is mostly technical and implies that there is a time in the random geometric graph process when the graph consists of a giant $k$-connected component and many (but not too many) vertices attached to it via $k-1$ edges.

\begin{lemma}\label{lem:props}
There exists $t_* = t_*(r)\ge 1$ such that the following statements hold simultaneously.
\begin{enumerate}[label=\emph{(\alph*)}]
    \item
    $t_* r^d\to \infty$,
    \item
    $\mathbb E[Z_{k-1,t_*}] \to \infty$ but also $\mathbb E[Z_{k-1,t_*}]\le \log\log(1/r)$,
    \item $\mathbb E[Z_{[0,k-2],t_*}]\to 0$; in particular, whp $G_{t_*}$ contains no vertex of degree at most $k-2$,
    \item
    if $k\ge 2$, whp $G_{t_*}$ contains no $(k-1)$-separated set of size between $2$ and $t_* - k - 1$, 
    \item
    if $k = 1$, whp $G_{t_*}$ contains a unique connected component with more than $1$ vertex.
\end{enumerate}
\end{lemma}
\begin{proof}
The proof combines~\Cref{thm:Pen04} for $\kappa=k-1$ with a diagonal extraction argument.
Fix $\beta > 0$. Then, by ignoring rounding, there exists $r_0 = r_0(\beta) > 0$ satisfying that, for every $r\in (0,r_0]$, we may (and do) fix $n = n(\beta,r)\ge \beta r^{-d}$ such that 
\[\mathbb E[Z_{k-1,n}] = n \binom{n-1}{k-1} (\theta_d r^d)^{k-1} (1-\theta_d r^d)^{n-1-(k-1)} = \beta.\]
In particular, for any fixed $\beta > 0$, $n = (1+o(1))d\log(1/r)/\theta_d r^d$, implying that
\[\mathbb E[Z_{[0,k-2],n}] = (1+o(1)) \mathbb E[Z_{k-2,n}] = \Theta(n (n r^d)^{k-2} \exp(-\theta_d n r^d)) = \Theta(\mathbb E[Z_{k-1,n}]/nr^d).\]
In particular, there is $r_1 = r_1(\beta)\in (0,r_0]$ such that, for every $r\in (0,r_1]$, $nr^d\ge \beta$ and $\mathbb E[Z_{[0,k-2],n}]\le 1/\beta$. 

Further, by \Cref{thm:Pen04}(a) and Chernoff's inequality (\Cref{lem:chernoff}), for every $\beta > 0$, there is $r_2 = r_2(\beta)\in (0,r_1]$ such that, for every $r\in (0,r_2]$, $Z_{k-1,n}\ge \beta/2$ with probability at least $1-2\e^{-\beta/8}$ (where the factor of 2 accounts for the discrepancy between the actual distribution of $Z_{k-1,n}$ and its Poisson limit).
By combining the latter fact with \Cref{thm:Pen04}(b) and (c), there exists $r_3 = r_3(\beta)\in (0,r_2]$ such that, for every $r\le r_3$, \Cref{lem:props}(d) and (e) and the inequality $Z_{k-1,n}\ge \beta/2$ are satisfied simultaneously with probability at least $1-3\e^{-\beta/8}$.

Now, for every $r\in (0,r_3(1)]$, define 
\[\gamma = \gamma(r)\coloneqq \sup\{\beta \in [1,\log\log(1/r)]: r\le r_3(\beta)\}\qquad \text{and}\qquad t_* = n(\gamma,r).\]
In particular, for every $r\le r_3(1)$, each of $t_* r^d$ and $\mathbb E[Z_{k-1,t_*}]$ is at least $\gamma$, $\mathbb E[Z_{[0,k-2],n}]\le 1/\gamma$, and \Cref{lem:props}(d) and (e) hold with probability at least $1-3\e^{-\gamma/8}$. As $\gamma(r)\to \infty$ when $r\to 0$, this implies the result.
\end{proof}

\begin{remark}\label{rem:asymptotic}
Together~\eqref{eq:expZ}, \Cref{lem:props}(b) and (c) imply that $t_* = (1+o(1)) d\log(1/r)/\theta_d r^d$.
\end{remark}

In the rest of this section, we fix $t_* = t_*(r)$ as in \Cref{lem:props}. 

\begin{lemma}\label{lem:isol}
Whp each of the graphs $(G_i)_{i=1}^{t_*}$ contains at least one vertex of degree at most $k-1$.
\end{lemma}
\begin{proof}
Denote by $\cA_1$ the event that $G_{t_*}$ contains at least $\mathbb E[Z_{k-1,t_*}]/2 = \omega(1)$ vertices of degree $k-1$.
By \Cref{lem:conc} and \Cref{lem:props}(b), $\cA_1$ holds whp.
Also, denote by $\cA_2$ the event that there is a point among $(X_i)_{i=1}^{t_*/2}$ which has degree $k-1$ in $G_{t_*}$. 
Observe that such a point would have degree at most $k-1$ in each of the graphs $(G_i)_{i=t_*/2}^{t_*}$.
Moreover, conditionally on $\cA_1$ and on the unlabelled random graph $G_{t_*}'$ obtained from $G_{t_*}$ by forgetting the indices (or, equivalently, the arrival times) of its vertices, every bijective assignment of the indices in $[t_*]$ to the vertices in $G_{t_*}'$ has the same probability.
In particular, in this case, $\cA_2$ is satisfied with probability at least $1-2^{-\mathbb E[Z_{k-1,t_*}]/2} = 1-o(1)$.

Now, denote by $\cA_3$ the event that $G_{t_*/2}$ contains at least $\mathbb E[Z_{0,t_*/2}]/2$ isolated vertices.
Note that, by~\eqref{eq:expZ} and \Cref{rem:asymptotic},  $\mathbb E[Z_{0,t_*/2}]= r^{-d/2+o(1)}$.
Combined with \Cref{lem:conc}, this shows that $\cA_3$ holds whp.
Also, set $\alpha = \alpha(r) = r^{-d/2-0.1}$ and denote by $\cA_4$ the event that there is a point among $(X_i)_{i=1}^{\alpha}$ which is isolated in $G_{t_*/2}$.
Again, using \Cref{rem:asymptotic}, conditionally on $\cA_3$ and on the unlabelled random graph $G_{t_*/2}'$ corresponding to $G_{t_*/2}$,
the probability that $\cA_4$ fails is at most
\[\bigg(1-\frac{\alpha}{t_*/2}\bigg)^{\mathbb E[Z_{0,t_*/2}]/2} = \exp(-r^{-0.1+o(1)}) = o(1).\]

Finally, the probability that $X_1$ is not isolated in $G_{\alpha}$ is $1-(1-\theta_d r^d)^{\alpha-1} = o(1)$. Denote this event by $\cA_5$. Then, the statement of the lemma fails with probability less than
\begin{align*}
\mathbb P(\cA_1^c\cup \cA_2^c\cup \cA_3^c\cup \cA_4^c\cup \cA_5^c)
&\le \mathbb P(\cA_1^c\cup \cA_2^c) + \mathbb P(\cA_3^c\cup \cA_4^c) + \mathbb P(\cA_5^c)\\
&\le \mathbb P(\cA_1^c) + \mathbb P(\cA_2^c\mid \cA_1) + \mathbb P(\cA_3^c) + \mathbb P(\cA_4^c\mid \cA_3) + \mathbb P(\cA_5^c) = o(1),
\end{align*}
as desired.
\end{proof}

For every $t\ge 1$, denote by $H_t$ the graph obtained from $G_t$ by deleting all vertices of degree at most $k-1$ (in $G_t$).
The end of the section is dedicated to the proof of the following proposition.

\begin{proposition}\label{prop:degk}
Whp $H_{t_*}$ is $k$-connected and, for every $t\ge t_*$, $X_{t+1}$ is adjacent to $k$ or more vertices~in~$H_t$.
\end{proposition}

To prepare the ground for the proof of \Cref{prop:degk}, we show that vertices of degree at most $k-1$ in $G_{t_*}$ are typically `far' from each other, and are also adjacent to vertices whose degree is significantly higher.

\begin{lemma}\label{lem:Ear}
\begin{enumerate}[label=\emph{(\alph*)}]
    \item
    Whp, for every choice of distinct $i,j\in [t_*]$ with $X_i$ of degree at most $k-1$ in $G_{t_*}$ and $|X_i-X_j|\le 3r$, $X_j$ has degree at least $2k$.
    \item
    Fix a constant $c > 0$, $i\in [t_*]$ and a region $A\subseteq A(X_i,r,2r)$ of volume at least $cr^d$. Then, the probability that $X_i$ has degree at most $k-1$ in $G_{t_*}$ and $A$ contains at most $k-1$ points is $o(1/t_*)$.
\end{enumerate}

\end{lemma}
\begin{proof}
We show (a) first. Define $\nu = \nu(r)\coloneqq (\log(1/r))^{-1/2}r$.
Then, for any $i\neq j$, the probability that $|X_i-X_j|\le \nu$ and $B(X_i,r)$ contains at most $k-1$ points among $(X_t)_{t=1}^{t_*}$ is at most
\[\theta_d \nu^d \cdot \mathbb P(\mathrm{Bin}(t_*-2,\theta_d r^d)\le k-2) = \Theta(\nu^d (t_* r^d)^{k-2} (1-\theta_d r^d)^{t_*-2-(k-2)}).\]
By a union bound over all possible choices of $i\neq j$, our choice of $t_*$ in \Cref{lem:props} and~\eqref{eq:expZ}, the probability that the latter event holds for some pair $i, j$ is 
\[O(t_*^2 \nu^d (t_* r^d)^{k-2} \exp(-\theta_d r^d t_*)) = O((\nu/r)^d \log\log(1/r))=o(1).\]

At the same time, by \Cref{lem:vols}, the union of any two balls with radius $r$ and distance between their centres at least $\nu$ has volume at least $\lambda = \lambda(r) \coloneqq \theta_d r^d+ C\nu r^{d-1}/2$ for the constant $C>0$ chosen therein.
Thus, the probability that $|X_i-X_j|\in [\nu,3r]$ and $B(X_i,r)\cup B(X_j,r)$ contains at most $3k$ points among $(X_t)_{t=1}^{t_*}$ is at most
\begin{align*}
\theta_d (3r)^d \cdot \mathbb P(\mathrm{Bin}(t_*-2,\lambda)\le 3k-2) 
&= \Theta(r^d\cdot (t_* \lambda)^{3k-2} (1-\lambda)^{t_*-2-(3k-2)})\\
&= r^d\cdot (t_* r^d)^{3k-2} \exp(-\theta_d t_* r^d-C (\nu/r)\cdot t_* r^d/2+O(1)).    
\end{align*}
Again, by a union bound over all possible choices of $i\neq j$ and using \Cref{rem:asymptotic}, the above expression is of order 
\[t_* (t_* r^d)^{3k-1} \exp(-\theta_d t_* r^d) \cdot \exp(-\omega(\log\log(1/r))) = o(1),\] 
where the last equality follows from our choice of $t_*$ in \Cref{lem:props}. This justifies part (a).

We turn to the proof of part (b). Denote by $a\ge cr^d$ the volume of $A$.
Then, a direct computation shows that the probability of the event complementary to (b) for a fixed $i\in [t_*]$ is dominated by 
\begin{align*}
\mathbb P(\mathrm{Bin}(t_*-1,\theta_d r^d +a) \le 2k-2) = O((t_* r^d)^{2k-2}\exp(-(\theta_d + c) t_* r^d))= t_*^{-1-c/\theta_d+o(1)},
\end{align*}
as desired.
\end{proof}

\begin{proof}[Proof of \Cref{prop:degk}]
First, we show that $H_{t_*}$ is a $k$-connected graph. An important part of the proof consists in the following two claims.

\begin{claim}[see e.g.\ Chapter~5, Lemma~5.2 in~\cite{Hav11}]\label{cl:Hav}
For any graph $G$ and vertex $v$ in $G$ such that $G\setminus v$ is a $k$-connected graph and $d_G(v)\ge k$, the graph $G$ itself is $k$-connected
\end{claim}

\begin{claim}\label{cl:local}
Whp, for every vertex $X_i$ of degree $k-1$ in $G_{t_*}$, the subgraph of $H_{t_*}$ induced by the vertices in the ball $B(X_i,2r)$ different from $X_i$ is $k$-connected.
\end{claim}
\begin{proof}
By \Cref{lem:Ear}(a), whp, for every point $X_i$ as required, every point $X_j\neq X_i$ at distance at most $r$ from $X_i$ has at least $k$ neighbours (in $G_{t_*}$) in the annulus $A(X_i,r,2r)$.
Hence, by \Cref{cl:Hav}, it is enough to show that the subgraph of $G_{t_*}$ induced by the vertices in $A(X_i,r,2r)$ is $k$-connected.

For fixed $i\in [t_*]$ and suitably small constant $c = c(d) > 0$, denote by $\cP_i$ an arbitrary partition of the annulus $A(X_i,r,2r)$ into regions with diameter at most $r/2$ and volume at least $c r^d$.
Define $\cA_i$ to be an auxiliary graph with vertices given by the regions in $\cP_i$ and edges between regions whose boundaries intersect: in particular, the graph $\cA_i$ is connected.
By \Cref{lem:Ear}(b) and a union bound over $i\in [t_*]$ and $|\cP_i|$, whp, for every point $X_i$ of degree $k-1$ in $G_{t_*}$, each region in $\cP_i$ contains at least $k$ points.
By deleting up to $k-1$ points in $A(X_i,r,2r)$, there remains at least one point on every region in $\cP_i$ and, by definition of $\cP_i$, the subgraph of $G_{t_*}$ induced by the remaining points in $A(X_i,r,2r)$:
\begin{itemize}
    \item contains a copy of $\cA_i$ where every vertex belongs to a different region in $\cP_i$,
    \item for every remaining point $X$, some vertex in this copy of $\cA_i$ at distance at most $r/2$ from $X$.
\end{itemize}
Thus, the surviving points induce a connected graph of $G_{t_*}$, which implies the desired $k$-connectivity.
\end{proof}

We are ready to show that $H_{t_*}$ is $k$-connected. The case $k = 1$ is trivial due to \Cref{lem:props}(e). 
Suppose $k \ge 2$. Fix any set $S$ of $k-1$ vertices in $H_{t_*}$ and, for the sake of contradiction, suppose that the graph $H_{t_*}\setminus S$ is not connected.
Condition on properties \Cref{lem:props}(d), \Cref{lem:Ear}(a) and (b) (for all $i\in [t_*]$ and arbitrary partitions of the balls around $X_1,\ldots,X_{t_*}$ into $O(1)$ sets of diameter at most $r/2$ and volume at least $cr^d$) and the statement of \Cref{cl:local}.
Then, by \Cref{cl:local}, for every point $X_i$ of degree (at most) $k-1$, the subgraph of $H_{t_*}$ induced by the vertices in the ball $B(X_i,2r)$ different from $X_i$, say $J_i$, is $k$-connected.
In particular, $J_i\setminus S$ is a non-empty connected graph for every $i$ as above.
We show by contradiction that, for all relevant $i,j$, $J_i\setminus S$ and $J_j\setminus S$ must be in the same connected component of $H_{t_*}\setminus S$. 
Suppose not and let $L_i, L_j$ be distinct connected components of $H_{t_*} \setminus S$ which contain $J_i \setminus S$ and $J_j \setminus S$, respectively.
Since $X_i$ has no neighbours outside of $L_i$, $\{X_i\}\cup V(L_i)$ is a $(k-1)$-separated set in $G_{t_*}$ and
\[2\le |\{X_i\}\cup V(J_i\setminus S)|\le |\{X_i\}\cup V(L_i)|\le t_*-|S|-|\{X_j\}\cup V(J_j\setminus S)|\le t_*-k-1,\]
which contradicts \Cref{lem:props}(d).
Finally, any connected component of $H_{t_*}\setminus S$ not containing the subgraphs $(J_i)$ must also be a connected component in $G_{t_*}\setminus S$.
Moreover, as such, it has to contain a single vertex of degree at most $k-1$, which is a contradiction since $H_{t_*}$ contains no such vertices.

\vspace{0.5em}

It remains to show the second part of the proposition. Say that a point $X$ is \emph{remote at time $t$} if $B(X,r)$ contains no more than $k$ points in $H_t$. 
Our goal is to show that, from time $t_*$ onwards, the remote volume gradually shrinks and eventually disappears before any new point could land there.

Consider a tessellation $\cT$ of $\bbT^d$ into $\ell = \ell(r) = \lceil \log(1/r)/r\rceil^d$ congruent hypercubes.
Next, we analyse the geometry and approximate the total volume of cubes containing a remote point at time $t_*$.
For every cube $s\in \cT$, denote by $\widetilde{B}(s,r)$ the intersection of the balls $B(x, r)$ for all $x\in s$, that is, $\widetilde{B}(s,r) := \bigcap_{x \in s} B(x, r)$. 
We define the set of \emph{remote hypercubes} (at time $t_*$) by setting $\cS = \{s \in \cT: |\widetilde{B}(s,r) \cap V(H_{t_*})| \le k-1\}$.
Remote hypercubes at time $t\ge t_*$ are defined analogously.

\begin{claim}\label{cl:far_cubes}
There is a constant $C_1 = C_1(k,d) > 0$ such that whp no two remote cubes are at distance between $\zeta = \zeta(r) = C_1\log\log(1/r)\cdot r/\log(1/r)$ and $3r$ from each other.
\end{claim}
\begin{proof}[Proof of \Cref{cl:far_cubes}]
By assuming the property in \Cref{lem:Ear}(a) (which holds whp), for every cube $s\in\cS$, there are up to $k$ vertices of $G_{t_*}$ in $\widetilde{B}(s,r)$: at most $k-1$ vertices in $H_{t_*}$ and at most one in $G_{t_*}\setminus H_{t_*}$.
Moreover, for every pair of cubes $s',s''\in \cS$ at distance in the interval $[\zeta,3r]$ from each other, 
each of $\widetilde{B}(s',r)$ and $\widetilde{B}(s'',r)$ contains a ball of radius $r-\sqrt{d}r/\log(1/r) = r-o(\zeta)$, and the centres of these two balls are at distance at least $\zeta$.
By \Cref{lem:vols} (recall the constant $C = C(d)$ therein), $\vol{d}(\widetilde{B}(s',r)\cup \widetilde{B}(s'',r))$ is at least
\[\theta_d (r-o(\zeta))^d + (C+o(1)) \zeta (r-o(\zeta))^{d-1}\ge \theta_d r^d + C\zeta r^{d-1}/2.\]
Hence, the probability that $\widetilde{B}(s',r) \cup \widetilde{B}(s'',r)$ contains up to $2k$ points among $(X_i)_{i=1}^{t_*}$ is at most
\begin{equation}\label{eq:chern}
\mathbb P(\mathrm{Bin}(t_*,\theta_d r^d + C\zeta r^{d-1}/2)\le 2k) = \Theta((t_* r^d)^{2k} \exp(-\theta_d t_* r^d - C\zeta t_* r^{d-1}/2)).
\end{equation}
Moreover, a union bound over at most 
\[\bigg(\frac{\log(1/r)}{r}\bigg)^d\cdot \frac{\theta_d (4r)^d}{(r/\log(1/r))^d} = O(t_* (\log(1/r))^{2d-1})\] 
pairs of cubes $s',s''$ together with~\eqref{eq:chern} implies that the claim holds whp for large enough $C_1$, as desired.
\end{proof}

Assuming \Cref{cl:far_cubes}, we partition the remote cubes into \emph{clusters} where $s',s''\in \cT$ are part of the same cluster if they are at distance at most $\zeta$ from each other.
As the number of clusters covering vertices of degree $k-1$ in $G_{t_*}$ is not hard to control, we focus on the number of remote cubes and clusters away from such vertices.
More precisely, we estimate the number of cubes in the family $\cS' = \{s \in \cT: |\widetilde{B}(s,r) \cap V(G_{t_*})| \le k-1\}$. To this end, observe that, for every $s\in \cT$, 
\[\vol{d}(\widetilde{B}(s,r))\ge \vol{d}(B(0,r-\sqrt{d}(\vol{d}(s))^{1/d})) \ge \theta_d (r-\Theta(r/\log(1/r)))^d \ge (1-c/\log(1/r))\theta_d r^d\]
for some constant $c>0$. Thus, the probability that a cube belongs to $\cS'$ is dominated by
\begin{align*}
\mathbb P(\mathrm{Bin}(t_*, (1-c/\log(1/r))\theta_d r^d)&\le k-1)\\ 
&= (1+o(1)) \binom{t_*}{k-1} (\theta_d r^d)^{k-1}(1-(1-c/\log(1/r))\theta_d r^d)^{t_*-k-1}\\
&= \Theta((t_* r^d)^{k-1}\exp(-\theta_d t_* r^d)).
\end{align*}
By recalling that $\ell = |\cT|$, Markov's inequality implies that whp there are at most 
\begin{equation}\label{eq:vol'}
\begin{split}
\ell (t_* r^d)^{k-1} \exp(-\theta_d t_* r^d)\cdot (\log(1/r))^{1/4} 
&= O((\log(1/r))^{d-1+1/4} t_* (t_* r^d)^{k-1} \exp(-\theta_d t_* r^d))\\
&= o((\log(1/r))^{d-1/2})
\end{split}
\end{equation}
cubes in $\cS'$, where we used that $\ell = \Theta((\log(1/r))^{d-1} t_*)$ and $t_* (t_* r^d)^{k-1} \exp(-\theta_d t_* r^d) = O(\log\log(1/r))$ thanks to \Cref{lem:props}(b) and~\eqref{eq:expZ}.

To finish the proof, we show that, for every remote cube $s$, there are $k$ points among $(X_i)_{i=t_*+1}^{\infty}$ landing at distance at most $r-\zeta$ from $s$ before any of these points lands in $s$.
A remote cube not satisfying the above property is called \emph{bad}.
Indeed, denote by $\cC_1,\cC_2,\ldots,\cC_h$ the clusters of remote cubes; we slightly abuse notation and denote by $\vol{d}(\cC_i)$ the volume of the union of all cubes constituting $\cC_i$.
By \Cref{cl:far_cubes} each of these clusters has diameter at most $\zeta$ and, therefore, for each cluster $\cC_i$ containing a bad cube, at least one of first $k$ points among $(X_j)_{j=t_*+1}^{\infty}$ landing at distance at most $r-2\zeta$ from $\cC_i$ actually lands in a cube in $\cC_i$. 
Moreover, for each $i\in [h]$, the probability of the latter event is at most
\[1-\bigg(1-\frac{\vol{d}(\cC_i)}{\vol{d}(B(0,r-2\zeta))}\bigg)^k = O\bigg(\frac{\vol{d}(\cC_i)}{r^d}\bigg),\]
where the constant in the $O$-term is uniform over all $i\in [h]$.
Conditionally on the event $Z_{k-1,t_*}\le 2\log\log(1/r)$ (which holds whp by \Cref{lem:props}(b) and \Cref{lem:conc}), the statement of \Cref{cl:far_cubes} and~\eqref{eq:vol'}, a union bound over $i$ shows that a bad cube exists with probability at most
\[O\bigg(Z_{k-1,t_*} \frac{\zeta^d}{r^d}\bigg)+O\bigg(\sum_{\cC_i:\; \cC_i\cap \cS'\neq \varnothing}\frac{\vol{d}(\cC_i)}{r^d}\bigg) = o(1)+O\bigg(\frac{\log(1/r)^{d-1/2} (r/\log(1/r))^d}{r^d}\bigg) = o(1).\]
As a result, for every remote cube $s$, the first $k$ points among $(X_i)_{t=t_*+1}^{\infty}$ landing at distance at most $r-\zeta$ from $s$ merge with the giant $k$-connected component instantaneously.
Thus, whp the remote volume disappears before being visited after time $t_*$, as desired.
\end{proof}

We are ready to deduce the first part of \Cref{thm:1}.

\begin{proof}[Proof of \Cref{thm:1} for $k$-connectivity]
We condition on the events in \Cref{prop:degk} and \Cref{lem:Ear} and show that each of the graphs $(H_t)_{t\ge t_*}$ is $k$-connected, which clearly implies the desired statement.
The base case is guaranteed by the first part of \Cref{prop:degk}. Suppose that $H_t$ is a $k$-connected graph for some $t\ge t_*$.
If $X_{t+1}$ is at distance more than $r$ from all vertices of degree $k-1$ in $G_t$, then $H_{t+1}$ is $k$-connected by \Cref{cl:Hav}.
Otherwise, $X_{t+1}$ connects to one vertex $X$ of degree $k-1$ in $G_t$ (see \Cref{lem:Ear}(a)), implying that $H_{t+1} = G_{t+1}[V(H_t)\cup \{X,X_{t+1}\}]$.
Then, by applying \Cref{cl:Hav} with $H_t$ and $X_{t+1}$, and then once again with $G_{t+1}[V(H_t)\cup \{X_{t+1}\}]$ and $X$, we deduce that $H_{t+1}$ is $k$-connected.
This finishes the induction and the proof.
\end{proof}

\subsection{Hamiltonicity}\label{subsection-hamiltonicity}

In this subsection, we prove that whp each of the graphs $(G_t)_{t \geq \taukcone{2}}$ is Hamiltonian. To this end, we adapt the proof from~\cite{BBKMW11} concerning the analogous phenomenon in the dynamic random graph with growing radius.
This proof is constructive and relies on several structural properties of the graph and the underlying point set which ensure the existence of a Hamilton cycle, and turn out to be typically satisfied.
Throughout the paper, we call the construction exhibited in the sequel the \emph{reference construction}.

\vspace{0.5em}
\noindent
\textbf{Step $1$: Tessellation.} We introduce a tessellation $\cT$ of $\mathbb T^d$ with hypercubes of side length $r/c$ for some large constant $c = c(d)$. 
The (Euclidean or $\ell_{\infty}$) \emph{distance} between two cubes is defined as the (Euclidean or $\ell_{\infty}$) distance between their centres: note that the distances between cubes in~\cite{BBKMW11} is rescaled by a factor of $r/c$ but we avoid this for coherence with the following sections.
In particular, any two points among $(X_i)_{i=1}^{t_*}$ in cubes at distance at most $(c-\sqrt{d})r/c$ form an edge in $G_{t_*}$.

\vspace{0.5em}
\noindent
\textbf{Step $2$: Identifying clusters of non-full hypercubes.}
Fix $k=2$ and recall the choice of $t_*$ from \Cref{lem:props}.

We consider an auxiliary graph $\cG$ with vertex set $\cT$ where two cubes form an edge if they are at $\ell_{\infty}$-distance at most $4r$ from each other.
Fix an integer $M$ to be quantified later.
A cube in $\cT$ is called \emph{$M$-nonfull} if it has less than $M$ vertices of $G_{t_*}$ inside, and \emph{$M$-full} otherwise. 
We bound the size of the largest connected component of $M$-nonfull cubes in $\cG$. The presented lemma is the only modification of the proof from~\cite{BBKMW11} (to be compared with Lemmas~4 and~12 in \cite{BBKMW11}).

\begin{lemma}\label{lem:3}
Fix integers $M \ge 1$ and $U = \lceil \theta_d (c + 1)c^{d-1}\rceil$. Then, whp, at time $t_*$, $\cG$ contains no connected components of $M$-nonfull hypercubes with more than $U$ vertices.
\end{lemma}
\begin{proof}
First, note that every connected subgraph of $\cG$ with more than $U$ vertices contains itself a connected subgraph of with exactly $U$ vertices: such may be constructed by pruning an arbitrary spanning tree. 
Moreover, the number of such subgraphs with exactly $U$ vertices is bounded from above by $|\cT| (\prod_{i=1}^{U-1} i(9c)^d) = \Theta( r^{-d} )$.

At the same time, recalling \Cref{rem:asymptotic}, we obtain that
\[U t_* r^d/c^d \ge (1+o(1)) d (1+1/c) \log(1/r).\]
Hence, the probability that some $U$ cubes in $\cT$ jointly cover at most $MU$ points among $(X_i)_{i=1}^{t_*}$ is
\begin{equation}\label{eq:polyr}
\sum_{i=0}^{MU} \binom{t_*}{i} \bigg(\frac{U}{(c/r)^d}\bigg)^i\bigg(1-\frac{U}{(c/r)^d}\bigg)^{t_*-i} = (t_* r^d)^{MU+o(1)} \exp(-(1+o(1)) U t_* r^d/c^d)= o(r^d).
\end{equation}
A union bound finishes the proof.
\end{proof}

\begin{remark}\label{rem:lem:3}
Note that the property of a cube to be $M$-full is monotone: if it holds at time $t_*$ for some realisation of the process $(X_i)_{i = 1}^\infty$, then it holds for every $t > t_*$ as well.
In particular, the property from \Cref{lem:3} is also monotone.
\end{remark}

Given an integer $t\ge 0$, the remainder of the proof in~\cite{BBKMW11} provides a way to deterministically construct a Hamilton cycle in $G_t$ under two assumptions:
\begin{enumerate}[\upshape{\textbf{A\arabic*}}]
   \item\label{prop:cG} the statement of Lemma~\ref{lem:3} holds at time $t$ (instead of $t_*$) for some $M > 2U + 2 + 2(2c + 1)^d$,
   \item\label{prop:G} $G_t$ is 2-connected.
\end{enumerate}

On the one hand, \ref{prop:cG} holds whp for every $t\ge \taukcone{2}$: indeed, it suffices to combine \Cref{lem:3}, the monotonicity exhibited in \Cref{rem:lem:3} and the fact that whp $\taukcone{2}\ge t_*$. 
On the other hand, \ref{prop:G} also holds whp for every $t\ge \taukcone{2}$, as shown in the first part of the proof of \Cref{thm:1} in \Cref{sec:k-conn 1-choice}.
Below, we briefly explain the remaining steps of the reference construction using \ref{prop:cG} and \ref{prop:G} to conclude.
We state some claims without proofs; a more complete account can be found in~\cite{BBKMW11}.

\vspace{0.5em}
\noindent
\textbf{Step $3$: Structure of non-full hypercubes.}
Fix $M$ as in~\ref{prop:cG} and a graph $G_t$ satisfying~\ref{prop:cG} and~\ref{prop:G}. Denote by $\cN$ the set of all $M$-nonfull hypercubes (at time $t$). 
For a set of cubes $S$, denote by $S_b$ the \emph{$b$-blow-up of $S$}, that is, the set of cubes at $\ell_\infty$-distance at most $br/c$ from $S$.

Consider the graph $\widetilde{\cG}$ with vertex set $\cT$ where two cubes form an edge if the Euclidean distance between them is at most $(c-d)r/c$.
We define three types of cubes in $\cT$.
First, $\widetilde{\cG} \setminus \cN$ has a connected component (denoted by $\widetilde{\cA}$ and called the \emph{sea}) containing $(1 - o(1))|\widetilde{\cG}|$ cubes in $\cT$. 
Second, we call a cube in $\cN$ \emph{close} if it has an edge towards $\widetilde{\cA}$ in $\widetilde{\cG}$. 
Third, we call a cube \emph{far} if its description does not fit in the previous categories, that is, it is not close and deleting the close cubes in $\widetilde{\cG}$ disconnects it from $\widetilde{\cA}$.
Note that the far cubes can be $M$-full or $M$-nonfull.

Fix a connected component $N\subseteq \cN$ of $M$-nonfull cubes in the graph $\cG$. 
Then, thanks to \ref{prop:cG}, every two far cubes separated from $\widetilde{\cA}$ by $N$ turn out to be at Euclidean distance at most $r/10$ from each other (see Lemma~5 in~\cite{BBKMW11}).
The rough explanation behind this statement is that every set of cubes of larger diameter has neighbourhood in $\widetilde{\cG}$ of size more than $U$, so some of these cubes must be $M$-full.
As a consequence, the restriction of the graph $G_t$ to the far cubes surrounded by $N$ is a complete graph.

Finally, for any two $M$-nonfull connected components $N', N''$ in $\cG$, the blow-ups $N'_{2c}, N''_{2c}$ are disjoint.
This property guarantees that, in the construction of the Hamilton cycle, each $M$-nonfull component can be treated separately.

\vspace{0.5em}
\noindent
\textbf{Step $4$: Dealing with far hypercubes.} Fix an $M$-nonfull component $N'$ in $\cG$, and consider the union of all far cubes covered by $N_c'$.
We consider three cases. If this region is empty or contains no vertex of $G_t$, we transition to the analysis of the close cubes in $N$ in the end of step 4.
If this region contains one point (second case) or more than one point (third case), by using the assumption of 2-connectivity and the fact that vertices inside far cubes in $N'_c$ form a complete subgraph of $G_t$, we can find a path that
\begin{itemize}
    \item is entirely contained in $N'_{2c}$,
    \item contains all vertices of $G_t$ inside far cubes in $N'_c$,
    \item starts and ends in the same cube in $\widetilde{\cA}$,
    \item contains at most one vertex of $G_t$ in every cube in $\widetilde{\cA}$ except the cube containing the endpoints (which contains two vertices).
\end{itemize}
We refer to this path as the \emph{far-reaching path} for the component $N'$ (see \Cref{subfig:hamilton-a}).

Next, we construct paths which absorb the vertices of $G_t$ in close cubes in $N'$ which remain outside the far-reaching path.
If no such vertices of $G_t$ exist, we can safely transition to step 5.
Otherwise, fix the cubes $p_1,\ldots,p_k\in N'$ containing at least one vertex of $G_t$, and let $q_1,\ldots,q_k\in \widetilde{\cA}$ be arbitrary (and, in particular, not necessarily distinct) neighbours of $p_1,\ldots,p_k$ in $\widetilde{\cG}$, respectively.
By using that $k\le U$ and $M > 2U+2$, one can ensure the existence of $k$ paths with the following properties:
\begin{itemize}
    \item the paths are (vertex-)disjoint between themselves and (vertex-)disjoint from the far-reaching path,
    \item the $i$-th path has its first vertex in $q_i$, then goes through all vertices of $G_t$ in $p_i$ in some order, and has its last vertex in $q_i$ (see \Cref{subfig:hamilton-b}).
\end{itemize}

\begin{figure}[t!]
\centering
\begin{subfigure}[t]{0.3\textwidth}
    \includegraphics[width=\textwidth]{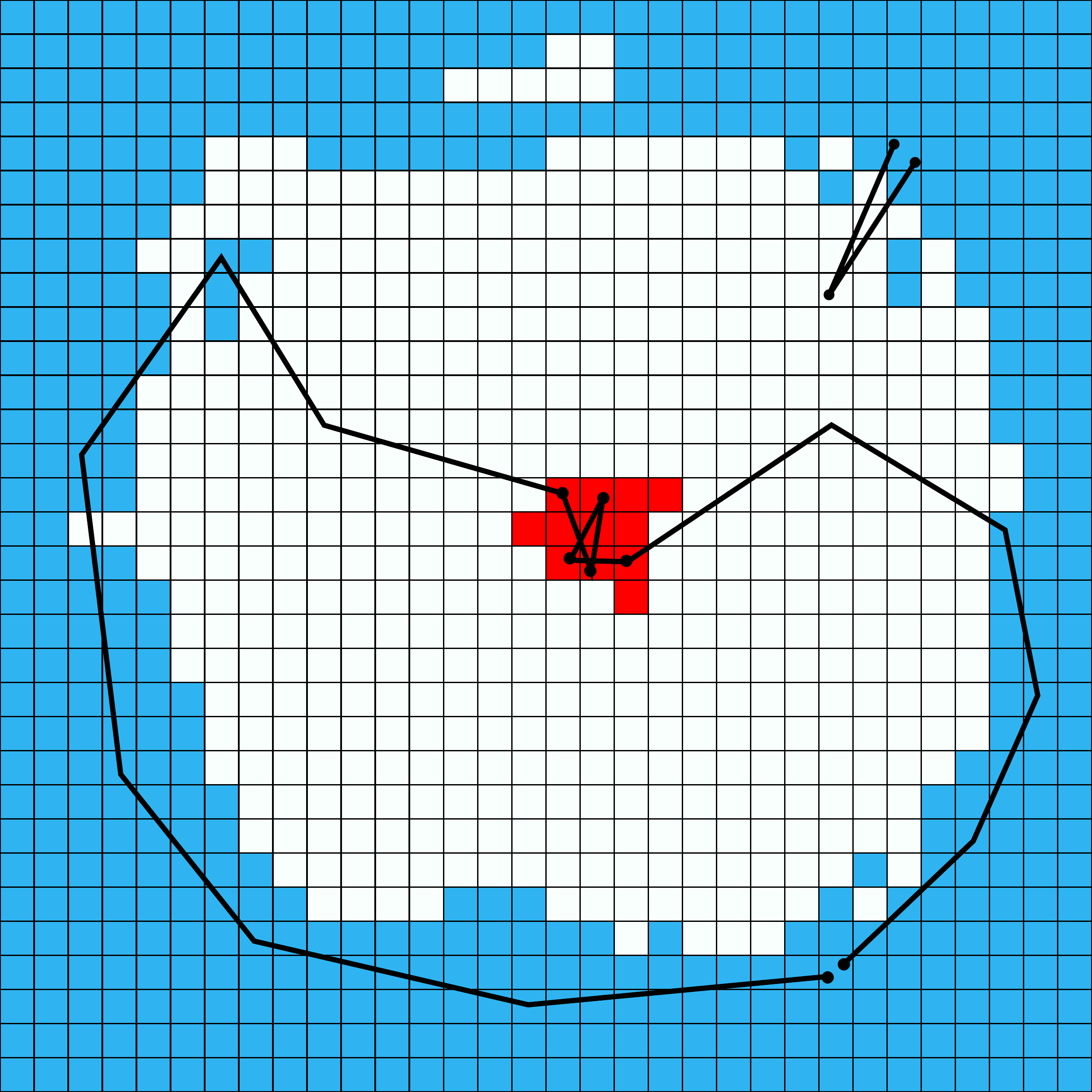}
    \caption{An exemple of an $M$-nonfull component. The far-reaching path absorbs the vertices in the far cubes. The other short black path is one of the multiple paths absorbing unused vertices in the close cubes.}
    \label{subfig:hamilton-a}
\end{subfigure}
\hfill
\begin{subfigure}[t]{0.3\textwidth}
    \includegraphics[width=\textwidth]{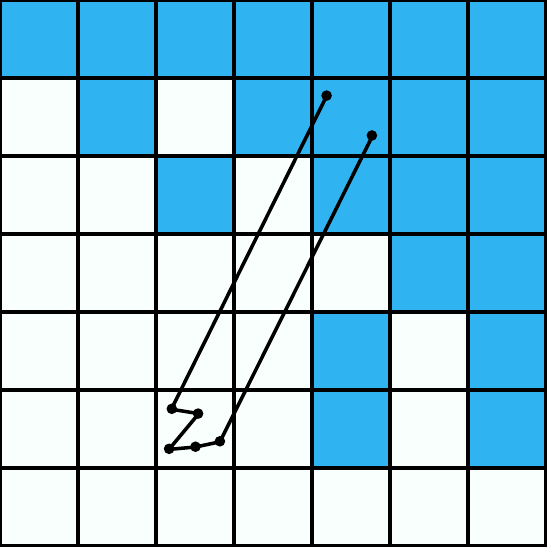}
    \caption{A zoom on the path absorbing vertices in the close cubes from subfigure (A).}
    \label{subfig:hamilton-b}
\end{subfigure}
\hfill
\begin{subfigure}[t]{0.3\textwidth}
    \includegraphics[width=\textwidth]{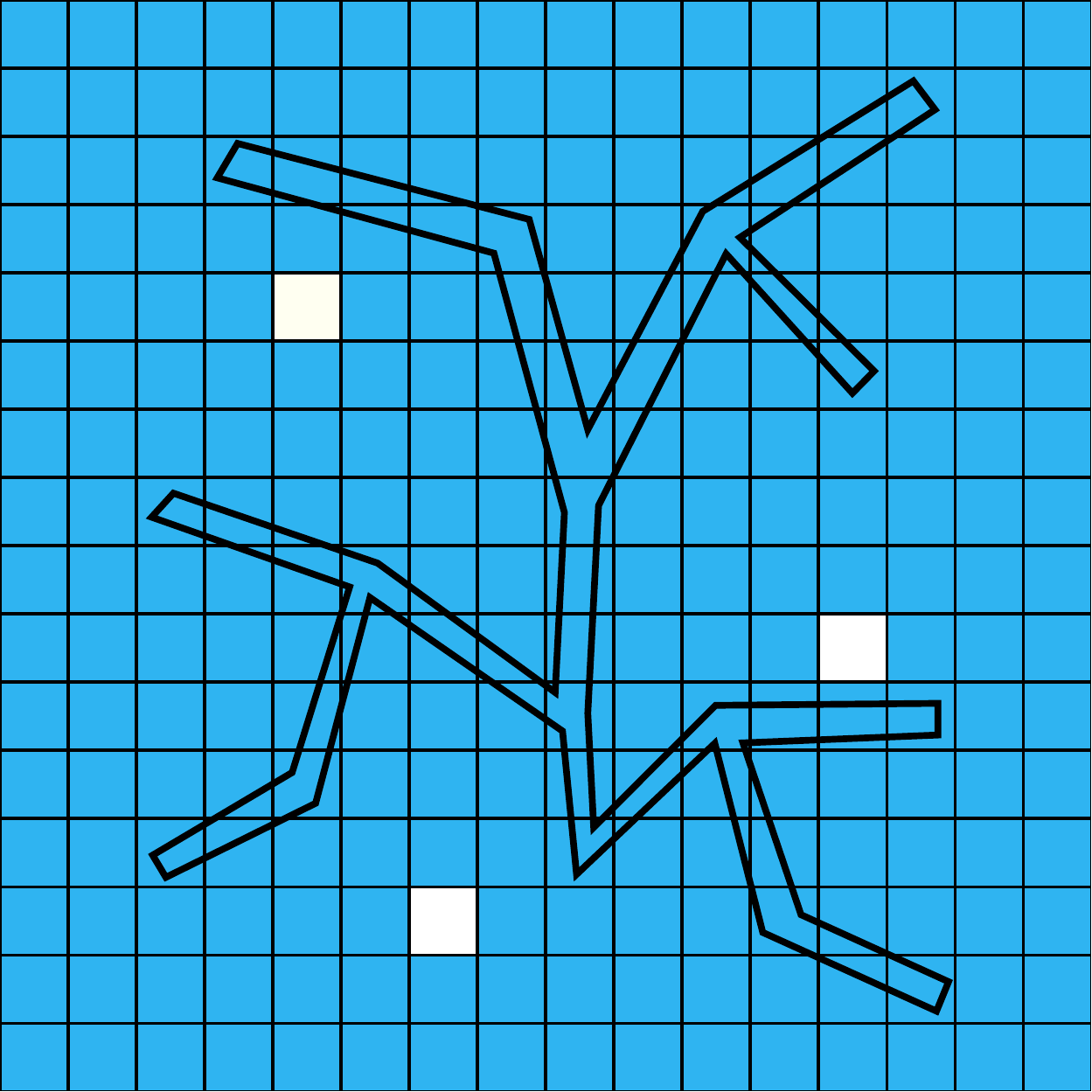}
    \caption{A symbolic representation of the idea of `doubling' the edges of the spanning tree of $\widetilde \cA$ described in Step 5.}
    \label{subfig:hamilton-c}
\end{subfigure}
\caption{An illustration of Steps 4 and 5 of the reference construction for $d = 2$. Blue squares {\fboxsep=0pt\fbox{\color{sea}\rule{2mm}{2mm}}} represent sea cubes, white squares {\fboxsep=0pt\fbox{\color{close}\rule{2mm}{2mm}}} – close cubes, and red squares {\fboxsep=0pt\fbox{\color{far}\rule{2mm}{2mm}}} – far cubes.}
\label{fig:hamilton-reference-construction}
\end{figure}

\vspace{0.5em}
\noindent
\textbf{Step $5$: Joining everything together using the vertices in $\widetilde{\cA}$.}
Finally, after performing step 4 for all $M$-nonfull components, we merge the constructed paths and the remaining vertices of $G_t$ into a single Hamilton cycle. 
To begin with, recall that $\widetilde{\cA}$ is a connected component in $\widetilde{\cG}$ and, therefore, one can find a spanning tree $T$ of $\widetilde{\cA}$. 
Moreover, by `doubling' every edge of $T$ (see \Cref{subfig:hamilton-c}) and using that $\widetilde \cG[\widetilde \cA]$ has maximum degree at most $(2c+1)^d$, one obtains a cycle $C_T$ visiting each cube in $\widetilde{\cA}$ at least twice and at most $2(2c+1)^d$ times.

One can convert this structure into a Hamilton cycle of $G_t$ as follows. Start at an unused vertex in a cube in $\widetilde{\cA}$. Then, move to any unused vertex in the next cube in the cycle $C_T$. If this is the last time when $C_T$ visits this cube, visit all remaining vertices and merge all paths constructed in step 4, and then leave for the next cube in $C_T$. If it is not the last time the cycle $C_T$ visits this cube, then move to any unused vertex in the next cube in $C_T$.
Note that, conditionally on the property from \Cref{lem:3}, the choice of $M$ in~\ref{prop:cG} guarantees that the algorithm does not run out of unused vertices at any point, which concludes the proof.

\section{\texorpdfstring{Hitting times and probabilistic monotonicity in the offline $2$-choice process}{Hitting times and probabilistic monotonicity in the offline 2-choice process}}\label{section:offline-2-choice}

In this section, we analyse the offline 2-choice process. 
First, we focus on the $k$-connectivity part of \Cref{thm:2}.

\subsection{\texorpdfstring{$k$-connectivity in the offline $2$-choice process}{k-connectivity in the offline 2-choice process}}\label{sec:k-conn-2off}

Following the proof outline presented in the introduction, we first draw some parallels with the approach used in \Cref{section:1-choice}.

\subsubsection{\texorpdfstring{Preparation and comparison with the $1$-choice process}{Preparation and comparison with the 1-choice process}}

For an integer $k\ge 1$ and suitably chosen $\varepsilon = \varepsilon(k) > 0$ (characterised in the sequel), we define
\begin{equation}\label{eqn:t**}
t_{\min} = t_{\min}(r,k,\eps,d) := \frac{d}{4\theta_d} \frac{\log(1/r)}{r^d} + (1 - \varepsilon)\frac{2k - 1}{4\theta_d} \frac{\log\log(1/r)}{r^d}.
\end{equation}

Our first lemma provides an analogue of \eqref{eq:expZ} and Lemma~\ref{lem:conc}.
For any $t = t(r) \geq 1$ and $\kappa \geq 0$, we denote by $Z'_{\kappa, t}$ the random variable equal to the number of partner pairs $X_{2i - 1}, X_{2i}$ in $(X_i)_{i = 1}^{2t}$ such that each of $X_{2i - 1}$ and $X_{2i}$ has degree $\kappa$ in $G_{2t}$.

\begin{lemma}\label{lem:conc-pairs}
Consider $\varepsilon > 0$ and integers $\kappa \geq 0$ and $t = t(r) \geq 1$ ensuring that $\mathbb E[Z_{\kappa,t}']\to\infty$. Then, whp $Z'_{\kappa, t} = (1 \pm \varepsilon) \mathbb E[Z_{\kappa,t}'] = (1 \pm 2\varepsilon) \mathbb E[Z_{\kappa,2t}]^2/4t$.
\end{lemma} 
\begin{proof}
First, note that $\mathbb E[Z_{\kappa,t}']\to\infty$ implies that $\mathbb E[Z_{\kappa,2t}]\to\infty$. In particular, \Cref{lem:conc} shows that whp $Z_{\kappa,2t} = (1+o(1)) \mathbb E[Z_{\kappa,2t}]$. Next, we condition on the configuration of the first $2t$ points: this amounts to revealing the positions of the points $(X_i)_{i = 1}^{2t}$ (and, in particular, the random variable $Z_{\kappa,2t}$) but not their indices. 
Then, 
\begin{equation}\label{eq:EZ'}
\mathbb E[Z_{\kappa,t}'\mid Z_{\kappa,2t}] = t \frac{Z_{\kappa,2t} (Z_{\kappa,2t}-1)}{2t(2t-1)}.
\end{equation}
Furthermore,
\[\mathbb E[(Z_{\kappa,t}')^2\mid Z_{\kappa,2t}] = \mathbb E[Z_{\kappa,t}'\mid Z_{\kappa,2t}] + t(t-1)\frac{Z_{\kappa,2t} (Z_{\kappa,2t}-1)(Z_{\kappa,2t}-2) (Z_{\kappa,2t}-3)}{2t(2t-1)(2t-2)(2t-3)},\]
where the last ratio stands for the probability that each of the vertices $X_{2i-1},X_{2i},X_{2j-1},X_{2j}$ has degree $\kappa$ in $G_{2t}$, for all choices of distinct $i,j\in [t]$.
Then, Chebyshev's inequality yields
\[\mathbb P(|Z_{\kappa,t}' - \mathbb E[Z_{\kappa,t}'\mid Z_{\kappa,2t}]|\ge \eps\mathbb E[Z_{\kappa,t}'\mid Z_{\kappa,2t}]/2\mid Z_{\kappa,2t})\le \frac{\mathrm{Var}(Z_{\kappa,t}'\mid Z_{\kappa,2t})}{(\eps/2)^2 \mathbb E[Z_{\kappa,t}'\mid Z_{\kappa,2t}]^2}.\]
Next, we show that whp
\[\mathbb E[Z_{\kappa,t}'\mid Z_{\kappa,2t}] = (1+o(1))\mathbb E[Z_{\kappa,t}'] = (1 + o(1)) \mathbb E[Z_{\kappa,2t}]^2/4t\quad \text{and}\quad \mathbb E[(Z_{\kappa,t}')^2\mid Z_{\kappa,2t}] = (1+o(1)) \mathbb E[Z_{\kappa,t}']^2.\]
Indeed, the three equalities directly follow from the fact that whp $Z_{\kappa,2t} = (1+o(1))\mathbb E[Z_{\kappa,2t}]\to \infty$ (shown in \Cref{lem:conc}) and $\mathbb E[Z_{\kappa,2t}^2]=(1+o(1))\mathbb E[Z_{\kappa,2t}]^2$ (following from the second moment in the proof of \Cref{lem:conc}). 
As a result, whp $Z_{\kappa,t}' = (1\pm (\eps/2+o(1))) \mathbb E[Z_{\kappa,t}'] = (1\pm \eps) \mathbb E[Z_{\kappa,t}']$, as desired.
\end{proof}

The next lemma estimates the number of partner pairs of two vertices with degrees at most $k-1$ before the hitting time $\tau_{2,k}$.

\begin{lemma}\label{lem:hitting-time-offline-lb}
For every $\varepsilon > 0$, the following statements hold jointly whp:
\begin{enumerate}[label=\emph{(\alph*)}]
    \item\label{lem:hitting-time-offline-lb-item-a} each of the graphs $(G_{2i})_{i=\log\log(1/r)}^{t_{\min}}$ contains at least $\log\log(1/r)$ pairs of partner points of degree at most $k-1$ each;
    \item\label{lem:hitting-time-offline-lb-item-b} all vertices in the graph $G_{2\log\log(1/r)}$ are isolated;
    \item\label{lem:hitting-time-offline-lb-item-c} $\taukctwo{k} - t_{\min} = \omega(1)$.
\end{enumerate}
\end{lemma}
\begin{proof} 
First of all, by combining~\eqref{eq:expZ} and~\eqref{eqn:t**}, we obtain that $\mathbb E[Z_{k-1,2t_{\min}}]\ge t_{\min}^{1/2}\log\log(1/r)$.
Denote by $\cB_1$ the event that $Z'_{k-1, t_{\min}} > \mathbb E[Z'_{k-1, t_{\min}}]/2$.
Then, by~\eqref{eq:EZ'} and the relation 
\[\mathbb E[Z_{k-1,2t_{\min}}^2]\ge \mathbb E[Z_{k-1,2t_{\min}}]^2 \ge t_{\min} (\log\log(1/r))^2\to \infty,\] 
we deduce that $\mathbb E[Z_{k-1,t_{\min}}']=\omega(\log\log(1/r))$ and thus \Cref{lem:conc-pairs} applies; in particular, $\bbP(\cB_1^c) = o(1)$.

Next, denote by $\cB_2$ the event that there are at least $\log\log(1/r)$ pairs of points in $G_{t_{\min}}$ which both have degree $k - 1$ in $G_{2t_{\min}}$.
Observe that, conditioning on the locations of the points $(X_i)_{i=1}^{2t_{\min}}$ and the partnership relation without revealing the order in which the partner pairs arrived in the process, the permutation of the partner pairs remains uniform. Thus, Chernoff's bound yields
\[\bbP(\cB^c_2\mid\cB_1) \leq \bbP(\text{Bin}(\mathbb E[Z'_{k-1, t_{\min}}]/2, 1/2) < \log\log(1/r)) = o(1).\]

Now, denote by $\cB_3$ the event that $G_{t_{\min}}$ contains at least $\mathbb E[Z_{0,t_{\min}/2}']/2 = r^{-d/2 +o(1)}$ partner pairs of isolated vertices.
By Lemma~\ref{lem:conc-pairs}, $\bbP(\cB_3) = 1 - o(1)$. Set $\gamma = \gamma(r) = r^{-d/2 -0.1}$ and denote by $\cB_4$ the event that there are at least $\log\log(1/r)$ partner pairs of isolated vertices in $G_{2\gamma}$ both of which remain isolated in $G_{t_{\min}}$.
Then, again conditionally on the locations of the unlabelled vertices,
\[ \bbP(\cB_4^c \mid \cB_3) \leq \bbP\left(\text{Bin}\left(\frac{\mathbb E[Z_{0,t_{\min}/2}']}{2}, \frac{\gamma}{t_{\min}} \right) < \log\log\left( \frac{1}{r} \right)\right) = o(1), \]
by Lemma~\ref{lem:chernoff}.

Next, denote by $\cB_5$ the event that every two points among $X_1, X_2, \ldots, X_{2\log\log(1/r)}$ are at distance less than $2r$ from each other.
By a simple union-bound, $\bbP(\cB_5^c) \leq \theta_d (2r)^d (2\log\log(1/r))^2 = o(1)$.

Finally, denote by $\cB_6$ the event that each of the points $X_1, X_2, \ldots, X_{2\log\log(1/r)}$ remains isolated in $G_{2\gamma}$.
Then, $\bbP(\cB_6^c \mid \cB_5) = (1 - 2\theta_d r^d \log\log(1/r))^{\gamma - 2\log\log(1/r)} = o(1)$.

To conclude the validity of (a), note that it suffices that the event $\cB_1 \cap \cB_2 \cap \cB_3 \cap \cB_4 \cap \cB_5 \cap \cB_6$ holds, and $\cB_5$ is sufficient for (b). Moreover,
\begin{align*}
    \bbP(\cB_1^c \cup \cB_2^c \cup \cB_3^c \cup \cB_4^c \cup \cB_5^c \cup \cB_6^c) \leq \bbP(\cB_1^c) + \bbP(\cB_2^c \mid \cB_1) + \bbP(\cB_3^c) + \bbP(\cB_4^c \mid \cB_3) + \bbP(\cB_5^c) + \bbP(\cB_6^c \mid \cB_5) = o(1).
\end{align*}
For point (c), it suffices to note that (a) holds whp with $\eps/2$ instead of $\eps$ as well, implying that $\tau_{2,k}\ge t_{\min}(r,k-1,\eps/2,d)$, and $t_{\min}(r,k-1,\eps/2,d)-t_{\min}(r,k-1,\eps,d)=\omega(1)$.
\end{proof}

We turn to describing how uniformly the points $(X_i)^{2t_{\min}}_{i=1}$ are spread. 
To this end, recall the tessellation $\cT$ of $\mathbb T^d$ into cubes of side length $r/c$ and the graph $\cG$ from Steps 1 and 2 in \Cref{subsection-hamiltonicity}.
For integers $M \geq 1$ and $t\ge 0$, a cube in $\mathcal{T}$ is said to be \emph{$M$-nonfull} (\emph{with respect to $(X_i)_{i=1}^t$}) if it contains less than $M$ vertices of $G_t$, and \emph{$M$-full} otherwise. 
By default, we use the terms $M$-nonfull and $M$-full with respect to $(X_i)_{i=1}^{2t_{\min}}$.
The next lemma is a counterpart of Lemma~\ref{lem:3} for $G_{2t_{\min}}$ with the only difference that the constant $U$ is replaced by $2U$.
The proof follows the one of Lemma~\ref{lem:3} verbatim and is therefore omitted.

\begin{lemma}\label{lem:3-counterpart}
    Fix constants $M \ge 1$ and $U = \lceil \theta_d (c+1) c^{d - 1}\rceil$. Then, whp $\cG$ contains no connected components of $M$-nonfull cubes with respect to $(X_i)_{i=1}^{2t_{\min}}$ having more than $2U$ vertices.
\end{lemma}

As Lemma~\ref{lem:3-counterpart} holds for any choice of constants $M$ and $c$, in the sequel, we assume its statement for some large enough $M, c > 0$ whose exact values are determined later.

While the replacement of the constant $U$ in Lemma~\ref{lem:3} by $2U$ in Lemma~\ref{lem:3-counterpart} may seem minor, it actually has significant effect on the structure of small $(k-1)$-separated components in the geometric graph.
First, we describe some similarities.
For an integer $t\ge 0$, denote by $\fN(t)$ the set of connected components of $M$-nonfull cubes with respect to $(X_i)_{i=1}^{2t}$ in $\cG$, and write $\fN = \fN(t_{\min})$ for short.
Further, recall the graph $\widetilde{\cG}$ from Step 3 in \Cref{subsection-hamiltonicity}.
Then, given $N\in \fN$, by \Cref{lem:3-counterpart}, the graph $\widetilde{\cG} \setminus N$ has one component $A(N)$ consisting of all but a bounded number of cubes.
One can similarly define the \emph{sea} $\widetilde{\cA} := \bigcap_{N \in \fN} A(N)$ and partition the set $\cT\setminus \widetilde{\cA}$ into \emph{close} cubes which are adjacent to some cube in $\widetilde{\cA}$ in $\widetilde{\cG}$, and \emph{far} cubes. 
In particular, every close cube is $M$-nonfull (as otherwise it would be part of the sea) but far cubes can be $M$-full as well. 
Note that controlling the clusters of far cubes is a main additional challenge in our case: indeed, typically there are many more such clusters in $G_{2\taukctwo{k}}$ and some of these are much larger compared to $G_{\taukcone{k}}$.

Recall the definition of a $b$-blow-up of a set from Step 3 in \Cref{subsection-hamiltonicity}.
The following lemma is an analogue of Corollary 11 in \cite{BBKMW11}.

\begin{lemma}\label{lem:sea-connected}
    The set of sea cubes $\widetilde{\cA}$ is connected in $\widetilde{\cG}$. 
\end{lemma}

\Cref{lem:sea-connected} is a corollary of the fact that the $2c$-blow-ups of different components in $\fN$ are disjoint and part (a) of the following lemma (analogous to Corollary 7 in~\cite{BBKMW11}). Part (b) is a result of general utility, which also implies (a).
Before presenting the statement, we introduce a notation. For a set of cubes $\cQ \subseteq \cT$, we write $\bigcup \cQ$ for the union of cubes in $\cQ$, that is, the set of points in $\mathbb T^d$ contained in at least one cube of $\cQ$. Whenever we refer to the diameter of $\cQ$, we mean the maximal Euclidean distance between two points in $\bigcup \cQ$.

\begin{lemma}\label{lem:cutoff-in-2c-blowup}
Conditionally on the statement of \Cref{lem:3-counterpart} for sufficiently large $c$, for all $N\in\fN$, each of the following holds:
\begin{enumerate}[label=\emph{(\alph*)}]
    \item the set $\cT\setminus A(N)$ is contained in $N_c$. In particular, the set of neighbours of $\cT\setminus A(N)$ in $\widetilde{\cG}$ is contained in $N_{2c}$;
    \item for every far cube $q \in \cT\setminus A(N)$, there is a close cube $q' \in N$ adjacent to $q$ in $\widetilde{\cG}$.
\end{enumerate}
\end{lemma}

\begin{proof}
Define $\nu = \nu(c, d) = 1 - 2d/c$ and write $N(\cdot) = N_{\widetilde{\cG}}(\cdot)$. We first show the following two claims: 
\begin{enumerate}[\upshape{\textbf{B\arabic*}}]
    \item\label{item:E1} Every close cube in $\cT \setminus A(N)$ belongs to $N$.
    \item\label{item:E2} For every cube $q\in \cT$ and a point $x\in q$, $B(x,\nu r)\subseteq \bigcup N[q]$.
\end{enumerate}
\ref{item:E1} follows from the fact that every cube adjacent to the sea $\widetilde\cA$ is either close or itself in the sea. For \ref{item:E2}, fix a point $x'\in B(x,\nu r)$ in a cube $q'$. 
Then, by the triangle inequality, the distance between the centres of $q$ and $q'$ is at most $\diam(q')/2 + \nu r + \diam(q)/2\le (1-2d/c+\sqrt{d}/c)r < (1-d/c)r$, so $qq'\in E[\widetilde{\cG}]$.

We show part (b) by contradiction. Suppose that a far cube $q \in \cT \setminus A(N)$ with $N\in \fN$ is not~adjacent to a close cube.
Denote by $\cQ$ the maximal $\widetilde{\cG}$-connected set of far cubes containing $q$.
By \ref{item:E1}, $N(\cQ) \subseteq N$.
Also, by \ref{item:E2}, $(\bigcup \cQ) \oplus B(0, \nu r)\subseteq \bigcup N[\cQ]$ and furthermore, since $q$ is not adjacent to any close cube, $\mathrm{Vol}_d(\bigcup \cQ) \ge \mathrm{Vol}_d(B(0, \nu r))$.
Thus, by the Brunn-Minkowski inequality (Lemma~\ref{lem:brunn-minkowski}), for all $d \ge 2$,
\begin{align*}
    \mathrm{Vol}_d \left( \bigcup N(\cQ) \right) & \geq \mathrm{Vol}_d( Q \oplus B(0, \nu r)) - \mathrm{Vol}_d\left( \bigcup \cQ \right)\\ 
    & \ge \mathrm{Vol}_d (B(0, \nu r)) + \binom{d}{1} \mathrm{Vol}_d (B(0, \nu r))^{(d-1)/d} \mathrm{Vol}_d\left( \bigcup \cQ \right)^{1/d} \ge (d + 1) \mathrm{Vol}_d(B(0, \nu r)).
\end{align*}
Since all cubes in $\cT$ have volumes $(r/c)^d$, we have that $|N|$ is at least
\begin{align*}
|N(\cQ)| \ge (d + 1) \mathrm{Vol}_d(B(0, \nu r))/ (r/c)^d \ge 3\theta_d \nu^d r^d / (r/c)^d = 3\theta_d (c - 2d)^d > 2 \lceil \theta_d (c+1) c^{d - 1} \rceil = 2U,
\end{align*}
where the last inequality holds for large enough $c$. This leads to a contradiction and shows part (b).~Part (a) follows directly from part (b). 
\end{proof}

Next, we state and prove a geometric lemma aiming to estimate the size of an $M$-nonfull component depending on the number of far cubes it cuts off from the sea. It extends results from \cite{BBKMW11} (to be compared with Lemmas~5 and~14 therein).

\begin{lemma}\label{lem:nonfull-component-size}
Fix an $M$-nonfull component $N \in \fN$ with size $|N|\le 2U$ and the set of far cubes $F(N)$ within $\cT\setminus A(N)$. Then, each of the following statements holds:
    \begin{enumerate}[label=\emph{(\alph*)}]
        \item\label{lem:nonfull-component-size-item-a} If $F(N)$ is non-empty, then $|N| > (1 - 3d^2/c) U$;
        \item\label{lem:nonfull-component-size-item-b} If $F(N)$ has diameter at least $r/10$, then $|N| > (1 + 9d^2/c)U$;
        \item\label{lem:nonfull-component-size-item-c} If $F(N)$ is connected in $\cG$, then $\bigcup F(N)$ has diameter at most $20dr$.
    \end{enumerate}
\end{lemma}

\begin{proof}
For a set $A\subseteq [0,\infty)^d$ and $\rho > 0$, define $A_{\rho} = (A \oplus B(0, \rho)) \cap [0,\infty)^d$. The following claim is (a rescaled version of) Proposition 5.15 in \cite{Pen04}.

\begin{claim}\label{lem:positive-orthant}
Fix $d \ge 2$ and $\varepsilon > 0$. Then, there is $\eta = \eta(\varepsilon) > 0$ such that, for every compact set $A \subseteq [0,\infty)^d$ with $\ell_\infty$-diameter at least $\varepsilon \rho$ and point $x\in A$ with minimal $\ell_1$-norm, we have 
\[\vol{d}(A_\rho) \geq \vol{d}(A) + \vol{d}(\{ x\}_\rho) + \eta \rho^d.\]
\end{claim}

We turn to the proof of the lemma.
As $\bigcup N$ has diameter $o(1)$, upon suitable translation if necessary, we can (and do) identify $\bigcup N$ with a subset of $[0.1, 0.9]^d$.
Set $A = \bigcup F(N)$.
If $F(N)$ contains at least one cube, then $A$ has $\ell_\infty$-diameter at least $r/c$.
Set $\rho = (c - 2d)r/c$.
By definition, all the cubes at distance at most $(c - d)r/c$ from $A$ are either far or close.
Therefore, $A_{\rho} \setminus A$ is entirely covered by the cubes of $N$.
Fixing $x \in A$ to be the point in $A$ with minimal $\ell_1$-norm, \Cref{lem:positive-orthant} implies that
\[ \vol{d}(A_{\rho} \setminus A) \geq \vol{d}(\{ x \}_{\rho}) = \theta_d (c - 2d)^d r^d/c^d. \]
Hence, the number of close cubes in $N$ is at least 
\[\theta_d (c - 2d)^d \ge \theta_d\left(1 - \frac{3d^2}{c}\right)(c + 2)c^{d - 1}\ge \left(1 - \frac{3d^2}{c}\right)U.\]
Note that the first inequality holds for sufficiently large $c$ thanks to the fact that the difference of the left and right hand side is a polynomial in $c$ with positive coefficient in front of the most significant term.

For part (b), the only difference in the above consideration is that now $A$ has $\ell_\infty$-diameter at least $r/(10d) \geq \rho/(10d)$. Thus, $\eta = \eta(1/(10d))>0$ does not depend on $c$ and, for $c$ large enough,
\[ \text{Vol}_d(A_{\rho} \setminus A) \geq \vol{d}(\{ x \}_{\rho}) + \eta \rho^d \geq (\theta_d + \eta) \left(1 - \frac{2d}{c}\right)^d r^d \geq \theta_d \left(1 + \frac{9d^2}{c}\right)(c + 2)c^{d - 1} \frac{r^d}{c^d}, \]
implying that $|N| \geq \theta_d (1 + 9d^2/c)(c + 2)c^{d - 1} > (1 + 9d^2/c)U$ and showing (b).

For part (c), we apply a different strategy.
First, observe that Euclidean diameter at least $20dr$ implies $\ell_\infty$-diameter at least $20r$.
Upon rotation if necessary, assume that the $\ell_\infty$-diameter of $\bigcup F(N)$ goes between two points differing only in their first coordinate.
Next, partition $[0, 1]^d$ into strips $S_1,\ldots,S_m$ of thickness $4r$ (except possibly the last one, which may be a slightly wider) which are orthogonal to the first dimension and every cube in $\cT$ belongs to a single strip. Note that $(\bigcup F(N))\cap S_i$ is non-empty for at least five consecutive strips.

For every such strip $S_i$, denote by $x_i^+$ and $x_i^-$ any points in $(\bigcup F(N))\cap S_i$ with maximal and minimal second coordinate, respectively.
Consider the family of cubes $\cC_i^+$ in $S_i$ at (Euclidean) distance at most $(c-2d)r/c$ from $x_i^+$ and centres with larger second coordinate than $x_i^+$. Then, on the other hand, $\bigcup\cC_i^+$ covers a quarter of the ball $B(x_{i}^+, (c-3d)r/c)$ (see \Cref{fig:geometry-coarse}).
On the other hand, the cubes in $\cC_i^+$ cannot be in the sea since some cube in $F(N)$ would then be close. They also cannot be far since they do not belong to $F(N)$ by the extremal choice of $x_i^+$, and do not belong to other far components as these are at distance at least $4r$ from $\bigcup F(N)$.
Thus, $\cC_i^+\subseteq N$.
By similar considerations with $x_i^-$ and the analogously defined family $\cC_i^-$, we conclude that
\[|N|\ge \sum_{i:\, (\bigcup F(N))\cap S_i\neq \varnothing} |\cC_i^+|+|\cC_i^-|\ge 5\cdot 2\frac{\theta_d ((c-3d)r/c)^d/4}{(r/c)^d}\ge 2 (c+2)c^{d-1}\ge 2U,\]
where the penultimate inequality holds for sufficiently large $c$. This finishes the proof of (c).
\end{proof}

\begin{figure}
    \centering
    \includegraphics[width=0.9\linewidth]{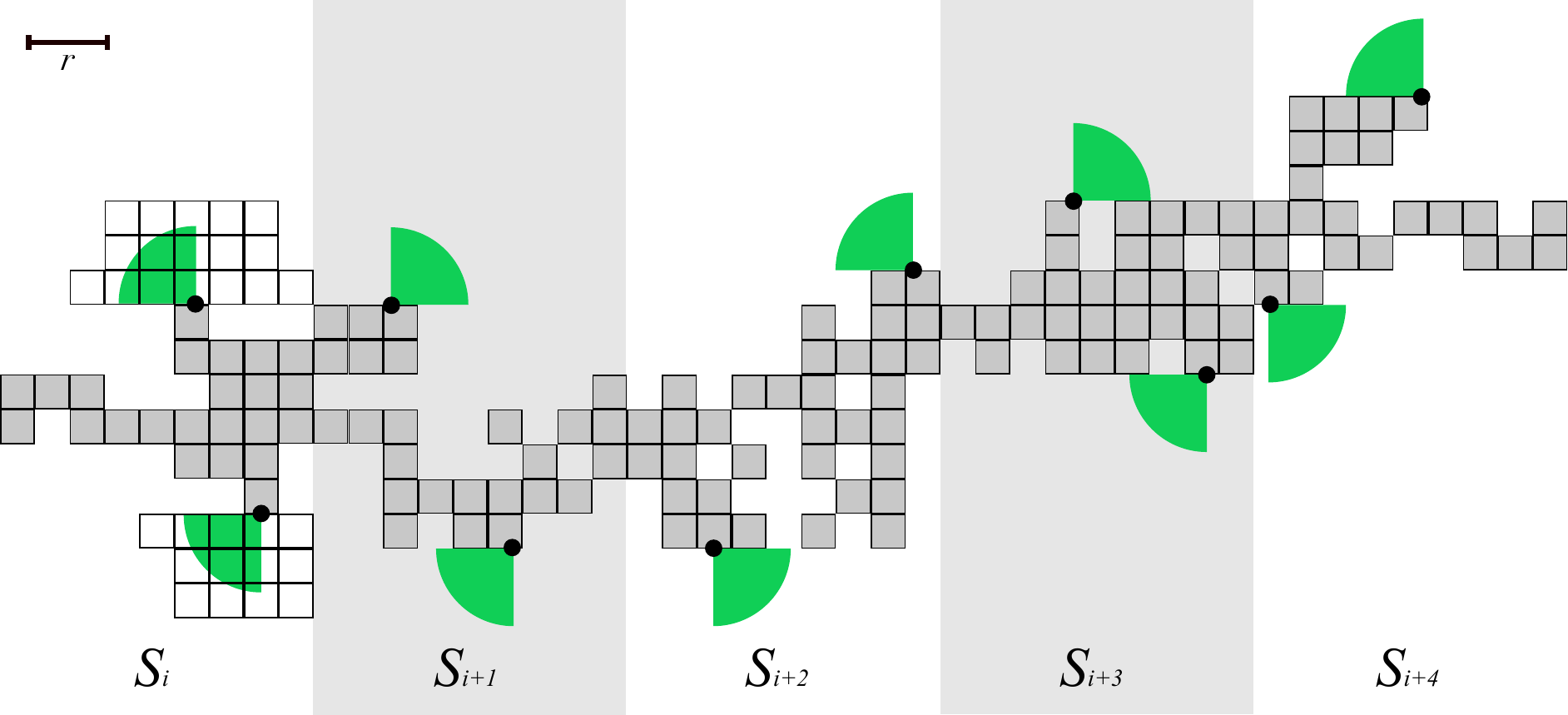}
    \caption{An illustration of the proof of \Cref{lem:nonfull-component-size}(c) in the case $d = 2$. The white squares {\fboxsep=0pt\fbox{\color{close}\rule{2mm}{2mm}}} in $S_i$ represent $\cC_i^+$ and $\cC_i^-$. Their counterparts inside other slices are omitted. The black dots $\bullet$ represent possible choices of the points $x_j^\pm$ for $j = i, \ldots, i+4$. The quarter-balls used to bound $|N|$ from below are drawn in green.}
    \label{fig:geometry-coarse}
\end{figure}

One purpose of the tessellation $\cT$ is to identify the positions of small $(k-1)$-separated sets: indeed, by choosing $M\ge k$, such sets necessarily lie inside far cubes.
In practice, we are going to choose $M$ slightly larger so that, after removing the vertices outside the choice set, typically each $M$-full cube still contains at least $k$ points.
The said approach has one disadvantage: the diameters of most small $(k-1)$-separated sets in $G_{2t_{\min}}$ are typically of order $o(r)$ and their positions in the corresponding $M$-nonfull cubes will be important for their connectivity properties.
As a result, the tessellation $\cT$ turns out to be too coarse to fully describe the structure of $(k-1)$-separated sets with small diameter.
To this end, we consider the finer tessellation $\mathcal{T}_f$ of $\mathbb T^d$ into congruent cubes of side length
\[s_f = s_f(r) := \frac{r\sqrt{\log\log(1/r)}}{\log (1/r)}.\]
Note that a typical cube in $\cT_f$ does not contain any point among $(X_i)_{i=1}^{2t_{\min}}$ when $d\ge 2$: indeed, we do not use $\cT_f$ to detect unusually sparse regions (task already achieved by $\cT$) but to better approximate the positions of the points in far cubes in $\cT$.
For convenience, we assume that $\cT_f$ is a refinement of $\cT$ in the sense that every cube in $\cT_f$ is entirely contained in one cube in $\cT$. 
Additionally, we define the auxiliary graph $\cG_f$ with vertex set $\cT_f$ where two cubes $q_1,q_2$ are adjacent if the Euclidean distance between their centres is at most $2r$.

For a cube $q\subseteq \bbT^d$, denote by $\boxtimes(q)$ its centre.
By adapting the notation $\bigcup \cQ$ to sets $\cQ \subseteq \mathcal{T}_f$, define
\begin{equation}\label{eqn:ann}
    \mathrm{Ann}(\cQ) := \bigg(\bigcup_{q \in \cQ} B\bigg(\boxtimes\hspace{-0.2em}(q), r-\frac{s_f\sqrt{d}}{2}\bigg) \bigg) \bigg{\backslash} \bigg(\bigcup \cQ\bigg).
\end{equation}

For a finite set of points $\bY \subseteq \bbT^d$ and an integer $\kappa\ge 1$, a set $\cQ\subseteq \cT_f$ is called $\kappa$-\emph{remote} (with respect to $\bY$, often abbreviated wrt $\bY$) if each of the following properties holds:
\begin{enumerate}[\upshape{\textbf{C\arabic*}}]
    \item\label{item:F1} $\cQ$ is connected in $\cG_f$;
    \item\label{item:F2} none of the cubes in $\cQ\subseteq \cT_f$ is contained in a cube in the sea $\widetilde{\cA}\subseteq \cT$;
    \item\label{item:F3} there are at most $\kappa-1$ points of $\bY$ in $\mathrm{Ann}(\cQ)$.
\end{enumerate}
We note that distinct $\kappa$-remote sets can overlap.

\begin{claim}\label{rem:remote-in-far}
Conditionally on the statement of \Cref{lem:3-counterpart} for sufficiently large $c$, given $\kappa \in [M]$ and a $\kappa$-remote set $\cC$ wrt $(X_i)_{i=1}^{2t_{\min}}$, there is $N \in \fN$ such that $\bigcup \cC\subseteq \bigcup F(N)$.
In particular, for every such set $\cC$, $\bigcup \cC$ has diameter less than $20dr$.
\end{claim}
\begin{proof}
Fix $q_f \in \cC$. By~\ref{item:F2}, there are $N \in \fN$ and $q \in \cT \setminus A(N)$ such that $q_f \subseteq q$.
We show that $q \in F(N)$. Assume the contrary, namely, that $q$ is a close cube.
Then, let $q' \in \widetilde \cA$ be a sea cube adjacent to $q$ in $\widetilde \cG$.
Fixing a point $x \in q'$ and using the triangle inequality gives
    \begin{align*}
        \text{dist}(\boxtimes(q_f), x)
        &\leq \text{dist}(\boxtimes(q_f), \boxtimes(q)) + \text{dist}(\boxtimes(q), \boxtimes(q')) + \text{dist}(\boxtimes(q'), x) \\
        &\leq \frac{\sqrt{d}}{2} \cdot \frac{r}{c} + \frac{(c - d)r}{c} + \frac{\sqrt{d}}{2} \cdot\frac{r}{c} \leq r - \frac{s_f \sqrt{d}}{2},
    \end{align*}
    where the last inequality holds if $r$ is sufficiently small. 
    Since for every component $\cC\subseteq \cT_f$ containing $q_f$ and disjoint from the sea region $\bigcup \widetilde\cA$ we have that all points in $q'$ are also in $\mathrm{Ann}(\cC)$, $\cC$ cannot be $\kappa$-remote for any $\kappa\le M$, contradicting our assumption.
    
We proved that, for each $q_f \in \cC$, there are $N \in \fN$ and $q \in F(N)$ such that $q_f \subseteq q$.
To conclude, it remains to show that all cubes in $\cC$ fall into $\bigcup F(N)$ for a unique $N\in \fN$.
Fix adjacent cubes $q_f, q'_f \in \cC$ in $\cG_f$ and denote by $q, q' \in \cT$ the cubes of the coarse tessellation containing $q_f, q_f'$, respectively. 
Furthermore, denote by $N, N' \in \fN$ the components of close cubes satisfying $q \in F(N)$ and $q' \in F(N')$.
By Lemma~\ref{lem:cutoff-in-2c-blowup}, there exist $u \in N$, $u' \in N'$ that are adjacent respectively to $q, q'$ in $\widetilde{\cG}$.
Then,
    \begin{align*}
        \text{dist}(\boxtimes(u), \boxtimes(u')) &\leq \text{dist}(\boxtimes(u), \boxtimes(q)) + \text{dist}(\boxtimes(q), \boxtimes(q_f)) + \text{dist}(\boxtimes(q_f), \boxtimes(q_f')) \\
        &+ \text{dist}(\boxtimes(q_f'), \boxtimes(q')) + \text{dist}(\boxtimes(q'), \boxtimes(u')) \\
        &\leq \frac{(c - d)r}{c} + \frac{dr}{c} + 2r + \frac{dr}{c} + \frac{(c - d)r}{c} = 4r.
    \end{align*}
    In particular, $u, u'$ are adjacent in $\cG$ and therefore $N=N'$. 
    An induction over a spanning tree of $\cG_f[\cC]$ finishes the proof.
\end{proof}

\begin{remark}\label{rem:remote-monotonicity}
    The $\kappa$-remote property of a family of cubes $\cQ\subseteq \cT_f$ is decreasing, that is, if $\cQ$ is $\kappa$-remote wrt a finite point set $\bY$ and $\bY' \subseteq \bY$, then $\cQ$ is $\kappa$-remote wrt $\bY'$.
\end{remark}

\subsubsection{A geometric tool}\label{sec:geom}

In this section, we introduce a geometric tool used later to estimate the number of connected $\kappa$-remote sets $\cC$ wrt $(X_i)_{i=1}^{2t_{\min}}$ with particular size and diameter.
Roughly speaking, we balance two competing effects: while $\mathrm{Ann}(\cC)$ must contain an atypically small number of points for any connected $\kappa$-remote set $\cC$ (which is less likely for larger $\cC$), the enumeration of possible connected sets grows with $|\cC|$.
To control the number of connected sets $\cC \subseteq \cT_f$ with diameter at most $20dr$, we find a subset $\cS\subseteq \cC$ of bounded size and a region $Cr(\cS) \subseteq \mathrm{Ann}(\cC)$ such that  $Cr(\cS)$ depends only on $\cS$ and constitutes a significant part of the volume of $\mathrm{Ann}(\cC)$ (a formal statement follows in Proposition~\ref{prop:main-geometric}).
By analysing $Cr(\cS)$ for $\cS\subseteq \cC$, we reduce the more complicated task of enumerating $\kappa$-remote sets to the simpler task of counting constant-size subsets of certain type. One can see the geometric tool defined in this section as an extension of the approach from the proof of \Cref{lem:nonfull-component-size}(c).
Nevertheless, there are some significant differences. Contrary to the said lemma, the construction presented here works with the finer tessellation $\cT_f$, the previously used quarter-balls are replaced with half-cylinders as the latter are better suited to describe components of small (i.e.\ $o(r)$) diameter, and lastly, the following construction uses lines and hyperplanes which are not necessarily parallel to the axes of the coordinate system.

Throughout \Cref{sec:geom}, we assume that $r$ is small and the analysis of sets of diameter at most $20dr$ can be conducted using the Euclidean metric in $\mathbb R^d$ (instead of the locally Euclidean metric on $\bbT^d$).

\begin{proposition}\label{prop:main-geometric}
For every non-empty $\cG_f$-connected set $\cC\subseteq \cT_f$ with $\diam(\bigcup \cC) = \lambda\in [s_f, 20dr)$, there are sets $\cS \subseteq \cC$ and $Cr(\cS) \subseteq \bbT^d$ such that each of the following properties holds:
    \begin{enumerate}[label=\emph{(\alph*)}]
        \item
        $|\cS| < 20d$ and $\diam(\bigcup \cS) = \lambda$,
        \item 
        $Cr(\cS)$ depends only on $\cS$ and $Cr(\cS) \subseteq \mathrm{Ann}(\cC)$,
        \item 
        $\theta_d r^d + (\theta_{d-1} \lambda/2^{d+2} - 3\theta_d d^2 s_f)r^{d-1} \leq \vol{d}(Cr(\cS)) \leq 21d \theta_d r^d$.
    \end{enumerate}
\end{proposition}
\begin{proof}
Denote by $Z_f$ the set of corners and centres of cubes in $\cT_f$ and fix an arbitrary total order $\prec$ on $Z_f$.
First, we construct $Cr(\cR)$ for an arbitrary $\cG_f$-connected set $\cR\subseteq \cT_f$ with $\diam(\bigcup\cR)=\lambda\le 20dr$ and deduce the construction of $\cS$ and $Cr(\cS)$ from it.
More concretely, we fix points $a=a(\cR),b=b(\cR)\in (\bigcup\cR)\cap Z_f$ at distance $\lambda$: note that this is possible since $Z_f$ contains the extremal points of $\bigcup \cR$. In case of ties, we fix $a$ minimal with respect to $\prec$ and $b$ minimal with respect to $\prec$ given $a$. Also, fix $(d-1)$-dimensional hyperplanes $m_1, m_2, \ldots, m_l, m_{l+1}$ orthogonal to $ab$ so that each of the following holds:
\begin{itemize}
    \item $a \in m_1$,
    \item for every $i\in [l-1]$, we have $\dist(m_i,m_{i+1}) = 2r$,
    \item $\dist(m_l, m_{l+1}) \le 2r$ and $b \in m_{l+1}$.
\end{itemize}
Note that the value of $l=l(\cR)$ depends on the length of $ab$: in particular, if $|ab| \le 2r$, then $l = 1$.

Next, for each $i\in [l]$, we let $\alpha_i=\alpha_i(\cR)$ denote the point in $(\bigcup \cR)\cap Z_f$ between $m_i$ and $m_{i+1}$ which realises the maximal distance to the straight line segment $ab$ (including the points lying on the planes; ties broken following $\prec$ again). 
Then, define $h_i$ as the half-space with boundary $\partial h_i\ni \alpha_i$ such that $\dist(h_i,ab)=\dist(\alpha_i,ab)$.
We remark that this choice is unique since $\alpha_i$ lies outside of the straight line $ab$: indeed, the fact that $\dist(m_i,m_{i+1}) = 2r$ and the connectivity of $\cR$ guarantee that there is at least one cube $q_f \in \cR$ whose centre and two adjacent corners lie between $m_i$ and $m_{i + 1}$, and at least one of these three points lies outside the straight line $ab$.

Then, denote by $B_i(\cR)$ the intersection of the ball $B(\alpha_i, r - 2 s_f \sqrt{d})$, $h_i$ and the slice between $m_i$ and $m_{i+1}$.
Additionally, we let $B_0(\cR)$ be the half of the ball $B(a, r - 2 s_f \sqrt{d})$ lying on the opposite side of $m_1$ than $b$, and $B_{l+1}(\cR)$ be the half of the ball $B(b, r - 2 s_f \sqrt{d})$ lying on the opposite side of $m_{l+1}$ than~$a$.
Finally, we define
\begin{equation}\label{eqn:cr}
Cr(\cR) := \bigcup_{i=0}^{l+1} B_i(\cR).
\end{equation}
Thus, our construction implies that $Cr(\cR)\subseteq \mathrm{Ann}(\cR)$ since $B_i(\cR)\subseteq \mathrm{Ann}(\cR)$ for every $i\in [0,l+1]$.

Now, we construct the sets $\cS$ and $Cr(\cS)$ given $\cC$. For every point $y\in 
\bigcup \cC$, denote by $q(y)$ an arbitrary cube in $\cC$ containing $y$. Then, we set
\[\cS := \{q(a(\cC)),q(b(\cC)),q(\alpha_1(\cC)),q(\alpha_2(\cC)),\ldots,q(\alpha_{l}(\cC))\}.\]
It remains to show that $\cS$ and $Cr(\cS)$ satisfy parts (a)--(c). First, (a) is implied by the inequalities 
\[|\cS| = l + 2 \le \lceil 20dr/2r\rceil+2<20d\qquad \text{and}\qquad \lambda = \dist(a,b) \le \diam(\bigcup \cS)\le \diam(\bigcup \cC) = \lambda.\]
Second, (b) follows from the construction of $Cr(\cdot)$: note that $a(\cS)=a(\cC)$, $b(\cS)=b(\cC)$, $l(\cS)=l(\cC)$ and $\alpha_i(\cS)=\alpha_i(\cC)$ for all $i\in [l]$. Thus, $Cr(\cS) = Cr(\cC)\subseteq \mathrm{Ann}(\cC)$.
    
To derive (c), note that the pairwise intersections of $B_0, B_1, \ldots, B_l, B_{l + 1}$ are either empty or contained in $(d-1)$-dimensional hyperplanes, implying that $\vol{d}(Cr(\cS)) = \vol{d} (B_0)+\vol{d} (B_1)+\ldots+\vol{d} (B_{l+1})$.
First, we focus on the lower bound. On the one hand, since $s_f=o(r)$, when $r$ is sufficiently small,
\[\vol{d}(B_0\cup B_{l+1})=\theta_d(r-2s_f\sqrt{d})^d\ge \theta_dr^d - 3\theta_d d^2 s_f r^{d-1}.\]
For the remainder of the sum, we consider the three cases:

\noindent
\textbf{Case $\lambda < r/2$ and $l=1$.} Note that the region $B_1$ contains half a cylinder with bases of radius $((r-2s_f \sqrt{d})^2-\lambda^2)^{1/2}\ge r/2$ contained in the hyperplanes $m_1$ and $m_2$, altitude $\lambda$ and axis contained in $\partial h_1$ (see \Cref{subfig:geometry-case-1}). In particular,
\[ \vol{d}(Cr(\cS)) = \vol{d}(B_0\cup B_2)+\vol{d}(B_1) \geq \theta_d r^d - 3\theta_d d^2 s_f r^{d-1} + \theta_{d - 1} r^{d - 1} \lambda/2^{d-1}.\]

\noindent
\textbf{Case $\lambda\in [r/2, 2r]$ and $l=1$.} Then, similarly to the previous case, the region $B_1$ contains half a cylinder with bases of radius $r/2$ (but not necessarily contained in the hyperplanes $m_1$ and $m_2$) and altitude $r/2 \geq \lambda/4$ (see \Cref{subfig:geometry-case-2}). In particular,
\[ \vol{d}(Cr(\cS)) \geq \theta_d r^d - 3\theta_d d^2 s_f r^{d-1} + \theta_{d - 1} r^{d - 1} (r/2)/2^{d-1}\ge \theta_d r^d - 3\theta_d d^2 s_f r^{d-1} + \theta_{d - 1} r^{d-1}\lambda/2^{d+1}.\]

\noindent
\textbf{Case $\lambda > 2r$ and $l>1$.} Note that, for every $i\in [l-1]$, $B_i$ contains half a cylinder with dimensions as in the previous case (see \Cref{fig:geometry-case-3}). Using that $\lambda \leq 2rl$, we obtain that
\begin{align*}
\vol{d}(Cr(\cS)) &\ge \vol{d}(B_0\cup B_{l+1}) + \sum_{i = 1}^{l-1} \vol{d}(B_i) \geq \theta_d r^d - 3\theta_d d^2 s_f r^{d-1} + (l-1)\frac{\theta_{d - 1} r^{d-1}(r/2)}{2^{d-1}}\\
&\geq \theta_d r^d - 3\theta_d d^2 s_f r^{d-1} + \frac{l-1}{4l}\frac{\theta_{d - 1} r^{d-1}(2rl)}{2^{d-1}}\ge  \theta_d r^d - 3\theta_d d^2 s_f r^{d-1} + \frac{\theta_{d - 1} r^{d-1}\lambda}{2^{d+2}}.
\end{align*}
For the upper bound, it suffices to notice that the volume of each $B_i$ is at most $\theta_d r^d$ and $l < 20d$.
\end{proof}

\begin{figure}[t!]
\centering
\begin{subfigure}[t]{0.35\textwidth}
    \includegraphics[width=\textwidth]{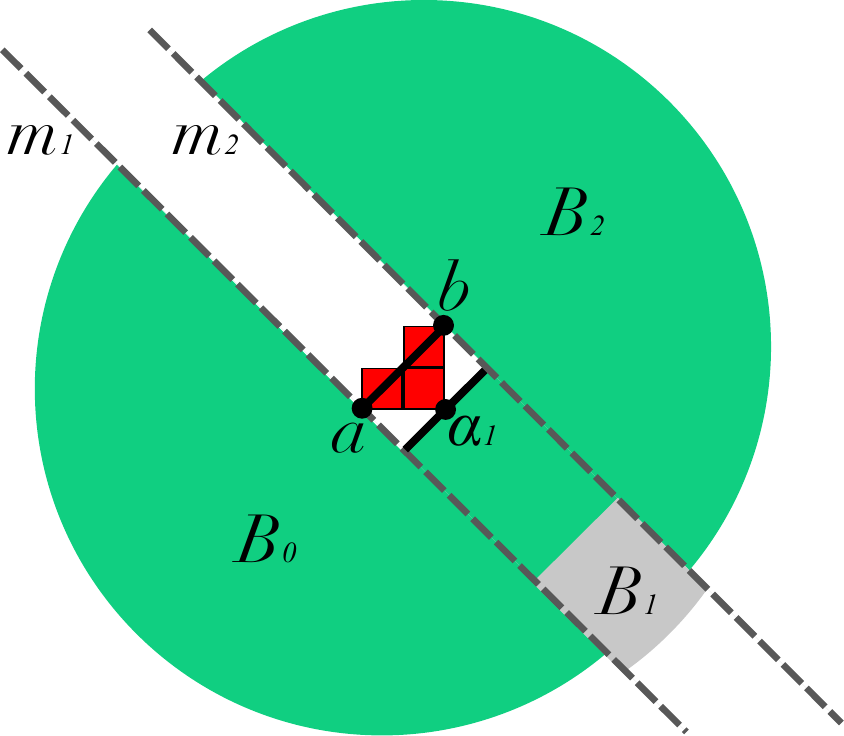}
    \caption{Case $\lambda < r/2$ and $l = 1$.}
    \label{subfig:geometry-case-1}
\end{subfigure}
\hspace{0.1\textwidth}
\begin{subfigure}[t]{0.35\textwidth}
    \includegraphics[width=\textwidth]{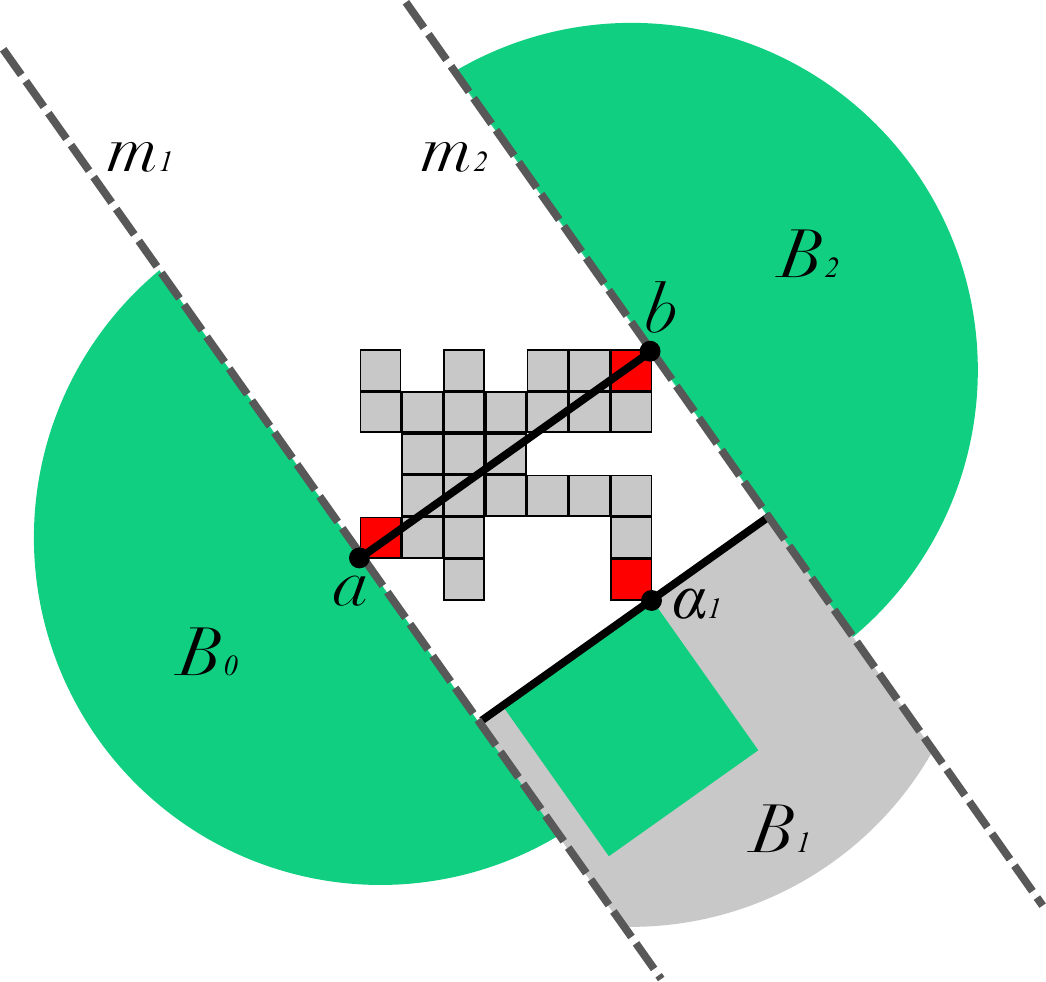}
    \caption{Case $\lambda \in [r/2, 2r]$ and $l = 1$.}
    \label{subfig:geometry-case-2}
\end{subfigure}

\caption{The first two cases in the proof of \Cref{prop:main-geometric} when $d = 2$. The cubes in $\cS$ are marked red. $Cr(\cS)$ is drawn in green.}
\label{fig:geometry}
\end{figure}

\begin{figure}
    \centering
    \includegraphics[width=0.8\linewidth]{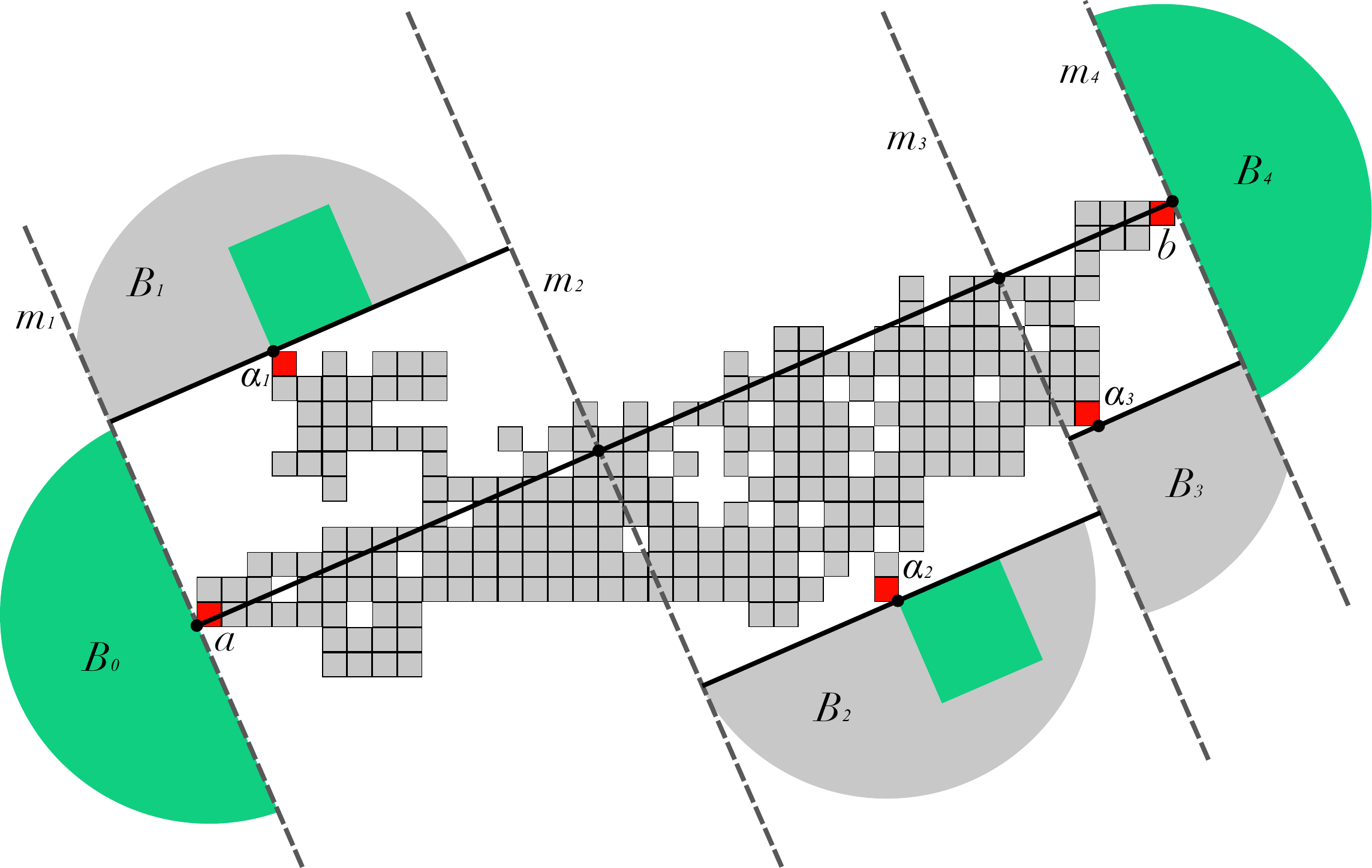}
    \caption{The case $\lambda > 2r$, $l > 1$, $d = 2$ in the proof of \Cref{prop:main-geometric}. }
    \label{fig:geometry-case-3}
\end{figure}

For a cube $q\in \cT_f$ and $\lambda\in [s_f,20dr)$, we define $\sat(q,\lambda)$ to be the set of all subsets $\cS \subseteq \cT_f$ such that $|\cS|<20d$ and $\diam(\bigcup(\cS\cup \{q\}))\le \lambda$. The next lemma bounds $\sat(q,\lambda)$ from above for the purposes of an upcoming union bound.

\begin{lemma}\label{lem:count-sat}
For every $\lambda\in [s_f,20dr)$ and $q\in \cT_f$, $|\sat(q,\lambda)|\le (2\lambda/s_f)^{20d^2}$.
\end{lemma}
\begin{proof}
Recalling that $\boxtimes(q)$ denotes the centre of $q$, note that every cube in $\cS$ is contained in the $\ell_{\infty}$-ball with centre $\boxtimes(q)$ and radius $\lambda$. As there are at most $(2\lambda/s_f)^d$ such cubes and $|\cS|<20d$, the statement follows.
\end{proof}

\subsubsection{\texorpdfstring{The total volume of connected $\kappa$-remote sets}{The total volume of connected к-remote sets}}

Recall $t_{\min} = t_{\min}(r,k,\eps,d)$ from~\eqref{eqn:t**}.
For every $\kappa\ge 1$, denote by $\fR_{\kappa}$ the family of $\kappa$-remote connected sets wrt $(X_i)_{i=1}^{2t_{\min}}$.
In this subsection, we estimate the total volume $\vol{d}(\bigcup \fR_{\kappa}) := \vol{d}(\bigcup_{\cR\in \fR_{\kappa}} (\bigcup \cR))$ via a first moment computation. The main result in this subsection is the following proposition; its proof is delayed to the end of the subsection.

\begin{proposition}\label{prop:main-volume}
For all suitably small $\eps=\eps(k)>0$, whp, for every integer $\kappa\in [M]$, each of the following properties holds:
    \begin{enumerate}[label=\emph{(\alph*)}]
        \item $\vol{d}(\bigcup \fR_{\kappa})\le r^{d/2} (\log(1/r))^{\kappa-k-1/3}$;
        \item the total volume of all $\cG_f$-connected $\kappa$-remote sets wrt $(X_i)_{i=1}^{2t_{\min}}$ of diameter at least 
\begin{equation}\label{eqn:df}
    d_f = d_f(r) := \frac{r (\log\log(1/r))^2}{\log(1/r)} = s_f (\log\log(1/r))^{3/2}
\end{equation}
is $r^{d/2} (\log(1/r))^{-\omega(1)}$.
\end{enumerate}
\end{proposition}

\begin{remark}\label{rem:remote-monotone}
By Remark~\ref{rem:remote-monotonicity}, the statement of Proposition~\ref{prop:main-volume} holds more generally for $\kappa$-remote sets wrt $(X_i)_{i=1}^{2t}$ for every $t \ge t_{\min}$.
\end{remark}

Proposition~\ref{prop:main-volume} will be deduced from the next lemma which provides an upper bound on the probability that each cube in a fixed bounded subset of $\cT_f$ belongs to a $\kappa$-remote set.
Given $q \in \cT_f$, $\kappa \ge 1$ and a finite set of points $\bY \subseteq [0, 1]^d$, we define
\begin{equation}\label{eqn:event-R}
R(q, \bY, \kappa) := \{\text{there is a $\cG_f$-connected $\kappa$-remote set $\cC\ni q$ wrt $\bY$}\text{ and } \diam(\bigcup \cC) < 20dr \},
\end{equation}
and
\begin{equation}\label{eqn:event-R'} 
\hspace{-0.8em}R'(q, \bY, \kappa) := \{\text{there is a $\cG_f$-connected $\kappa$-remote set $\cC\ni q$ wrt } \bY \text{ and } \diam(\bigcup \cC)\in (d_f,20dr) \}.
\end{equation}

\begin{lemma}\label{lem:prob-far}
Fix suitably small $\eps=\eps(k)>0$, integers $\kappa \in [M]$, $m \geq 1$ and a set $\cR = \{q_1,\ldots,q_m\} \subseteq \cT_f$ where $\dist(q_i, q_j) > 50dr$ for all distinct $q_i, q_j \in \cR$. Then,
\begin{enumerate}[label=\emph{(\alph*)}]
        \item $\bbP(\bigcap_{j=1}^m R(q_j, (X_i)_{i = 1}^{2t_{\min}}, \kappa)) \le r^{md/2} (\log(1/r))^{m(\kappa - k -2/5)}$, and
        \item $\bbP(R'(q_1, (X_i)_{i = 1}^{2t_{\min}}, \kappa) \cap \bigcap_{j = 2}^{m} R(q_j, (X_i)_{i = 1}^{2t_{\min}}, \kappa)) = r^{md/2} (\log(1/r))^{-\omega(1)}.$
\end{enumerate}
\end{lemma}
\begin{proof}
Denote $\bY = (X_i)_{i = 1}^{2t_{\min}}$.
By our assumption on $\cR$, the sets $\sat(q_1, 20dr),\ldots$, $\sat(q_m, 20dr)$ are all pairwise disjoint. Furthermore, for every choice of $\cS_1 \in \sat(q_1, 20dr),\ldots,\cS_m\in \sat(q_m, 20dr)$, the sets $Cr(\cS_1),\ldots,Cr(\cS_m)$ are also pairwise disjoint. Hence,
    \begin{align}
        \bbP\bigg(\bigcap_{i = 1}^m R(q_i, (X_i)_{i = 1}^{2t_{\min}}, \kappa)\bigg)
        &\leq \sum_{\cS_1 \in \text{sat}(q_1, 20dr)} \ldots \sum_{\cS_m \in \text{sat}(q_m, 20dr)} \bbP(|\bY \cap Cr(\cS_1)| < \kappa, \ldots, |\bY \cap Cr(\cS_m)| < \kappa)\nonumber \\
        &\leq \sum_{\cS_1 \in \text{sat}(q_1, 20dr)} \ldots \sum_{\cS_m \in \text{sat}(q_m, 20dr)} \bbP\bigg(\bigg|\bY \cap \bigcup_{i = 1}^m Cr(\cS_i)\bigg| \leq m(\kappa - 1)\bigg).\label{eq:UB-sat}
    \end{align}
    To estimate~\eqref{eq:UB-sat}, we consider two cases.

    \noindent
    \textbf{Case 1:} $\cS_i\in \sat(q_i,d_f)$ for every $i\in [m]$. Then, given $\cS_1,\ldots,\cS_m$, by \Cref{prop:main-geometric}(c), the random variable in~\eqref{eq:UB-sat} has binomial distribution with parameters $2t_{\min}$ and 
    \[\vol{d}(Cr(\cS_1)\cup\ldots\cup Cr(\cS_m))\ge m\theta_d r^d-3m\theta_dd^2s_fr^{d-1} =: \psi(m,r).\]
    Since $\mathbb P(\mathrm{Bin}(n,p)\le x)$ is an decreasing function of $p$ for fixed $n,x$, by using \Cref{lem:count-sat}, we obtain that the part of~\eqref{eq:UB-sat} falling into the current case is bounded from above by
    \[\bigg(\frac{2d_f}{s_f}\bigg)^{20d^2m}\; \sum_{i=0}^{m(\kappa-1)}(2t_{\min})^i \psi(m,r)^i (1-\psi(m,r))^{2t_{\min}-i}.\]
    As every term in the latter sum is negligible compared to the following ones, by using the relation $(1-x)^t=\exp(-xt+O(x^2t))$ as $x\to 0$, the expression is bounded from above by
    \[2\bigg(\frac{2d_f}{s_f}\bigg)^{20d^2m} (2 t_{\min} \psi(m,r))^{m(\kappa-1)} \exp(-2t_{\min} \psi(m,r))= r^{md/2} \bigg(\log\bigg(\frac1r\bigg)\bigg)^{m(\kappa-1)-(1-\eps)m(2k-1)/2-o(1)}.\]
    As the exponent of the logarithm is smaller and bounded away from $m(\kappa-k-2/5)$ for any $\varepsilon \in (0, (10k)^{-1})$, this part of the sum~\eqref{eq:UB-sat} is asymptotically negligible compared to the upper bound in(a).

    \noindent
    \textbf{Case 2:} $\cS_i \in \text{sat}(q_i, 20dr) \setminus \text{sat}(q_i, d_f)$ for at least one $i \in [m]$. 
    Then, by \Cref{prop:main-geometric}(c),
    \[\vol{d}(Cr(\cS_1)\cup\ldots\cup Cr(\cS_m))\ge m\theta_d r^d+\theta_{d-1} d_f r^{d-1}/2^{d+2}-3m\theta_dd^2s_fr^{d-1} =: \varphi(m,r).\]
    Similarly to the previous case, the part of~\eqref{eq:UB-sat} falling into the current case is bounded from above by
    \[\begin{split}
    2\bigg(\frac{20dr}{s_f}\bigg)^{20d^2m} (2 t_{\min} \varphi(m,r))^{m(\kappa-1)} \exp(-2t_{\min} \varphi(m,r)) 
    &= r^{md/2} \exp(-\omega(\log\log(1/r)))\\
    &= r^{md/2} (\log(1/r))^{-\omega(1)},
    \end{split}\]
    which finishes the proof of each of (a) and (b).
\end{proof}

\begin{proof}[Proof of~\Cref{prop:main-volume}]
Recall from \Cref{rem:remote-in-far} that whp every $\kappa$-remote set wrt $(X_i)_{i = 1}^{2t_{\min}}$ has diameter less than $20dr$. Assuming this event, fix $\varepsilon > 0$ such that the statement of \Cref{lem:prob-far} holds, $m = 1$ and $\kappa\in [M]$. 
Then, part (a) (respectively (b)) of the proposition follows from part (a) (respectively (b)) of \Cref{lem:prob-far} and a first moment argument for the number of cubes in $\cT_f$ in connected $\kappa$-remote sets.
\end{proof}

\subsubsection{Favourable properties of the partner pairs}

Before we proceed to the actual construction of the choice set $U$, we pinpoint three typical properties of the partner pairs which are crucial for the construction.
Define
\begin{equation}\label{eqn:t**+}
    t_{\max} = t_{\max}(r,k,d) := \frac{d}{4\theta_d} r^{-d} \log\bigg(\frac{1}{r}\bigg) + \frac{k}{\theta_d} r^{-d} \log\log\bigg(\frac{1}{r}\bigg),
\end{equation}
which will serve as a upper bound for our hitting time $\taukctwo{k}$.
More precisely, we fix suitably small $\eps = \varepsilon(k) > 0$ such that the statements of \Cref{prop:main-volume} and \Cref{lem:prob-far} hold, and denote by $A_k$ the event $\taukctwo{k}\in I_k := [t_{\min} + 4, t_{\max}]$. Our next lemma shows that $A_k$ holds whp.

\begin{lemma}\label{lem:hitting-time-offline-ub}
$\mathbb P(A_k)=1-o(1)$.
\end{lemma}

\begin{proof}
    By Lemma~\ref{lem:hitting-time-offline-lb}, $\bbP(\taukctwo{k} < t_{\min} + 4) = o(1)$. Thus, we focus on the upper bound. We will show that, with high probability, there are no partner pairs among $(X_i)_{i = 1}^{2t_{\max}}$ which simultaneously have degree in that interval $[0, k - 1]$. In such case, $\taukctwo{k} \leq t_{\max}$ will follow from the definition of $\taukctwo{k}$.

    Observe that the probability that $X_1, X_2$ both have degree smaller than $k$ in $G_{2t_{\max}}$ is dominated by
    \begin{align*}
        \bbP(\dist(X_1, X_2) &\leq 2r) \bbP(\mathrm{Bin}(2t_{\max} - 2, \theta_d r^d) < 2k) + \bbP(\dist(X_1, X_2) > 2r) \bbP(\mathrm{Bin}(2t_{\max} - 2, 2\theta_d r^d) < 2k) \\
        & = O(r^d (t_{\max} r^d)^{2k - 1} (1 - \theta_d r^d)^{2t_{\max}} ) + O(  (t_{\max} r^d)^{2k - 1} (1 - 2\theta_d r^d)^{2t_{\max}} ) \\
        & = O(r^d/(\log(1/r))^{2k+1}) = o(t_{\max}^{-1}).
    \end{align*}
    A union bound over all partner pairs in $(X_i)_{i = 1}^{2t_{\max}}$ finishes the proof.
\end{proof}

We come back to describing the promised properties.
The first favourable property of partner pairs essentially states that, whp, no partner pair is located in the proximity of a single $M$-nonfull cube. 
More formally, denote by $D_M$ the event that there do not exist a partner pair $Y,\bar Y\in (X_i)_{i=1}^{2t_{\max}}$ and a cube $q \in \mathcal{T}$ such that $q$ is $M$-nonfull wrt $(X_i)_{i = 1}^{2t_{\min}}$ and $\max\{\dist(q,Y), \dist(q,\bar Y)\} \leq 100dr$.

\begin{lemma}\label{lem:D}
$\bbP(D_M) = 1-o(1)$.
\end{lemma}
\begin{proof}
To begin with, we describe two favourable properties of the distribution of $(X_i)_{i = 1}^{t_{\min}}$ with respect to the coarse tessellation $\cT$. 
These properties will be useful in wider generality and, therefore, we state and prove them in a stronger form than needed here. Recall the constant $U = U(c, d)$ defined in \Cref{lem:3-counterpart}. We show that, for sufficiently large constant $Z = Z(c, d)$, the following properties hold jointly whp:
\begin{enumerate}[\upshape{\textbf{D\arabic*}}]
    \item\label{item:H1} there is $\alpha=\alpha(c,d)>0$ such that the number of $M$-nonfull cubes wrt $(X_i)_{i = 1}^{2t_{\min}}$ is at most $r^{\alpha-d}$,
    \item\label{item:H2} for every cube $q \in \cT$,there are at most $Z \log(1/r)$ vertices among $(X_i)_{i=1}^{2t_{\max}}$ at distance at most $100dUr$ from $q$.
\end{enumerate}
For \ref{item:H1} note that, by Chernoff's bound, the expected number of $M$-nonfull cubes wrt $(X_i)_{i = 1}^{2t_{\min}}$ is~at~most 
\[|\cT|\cdot \mathbb P(|q\cap (X_i)_{i=1}^{2t_{\min}}|\le M)=r^{-d+o(1)}\cdot \mathbb P(\mathrm{Bin}(2t_{\min},r^d/c^d)\le M) = r^{-d+o(1)} (t_{\min}r^d)^{O(1)} (1-r^d/c^d)^{2t_{\min}}.\]
Using that $(1-r^d/c^d)^{2t_{\min}} = \exp(-\Omega(\log(1/r)))$ together with Markov's inequality shows \ref{item:H1}.
For \ref{item:H2}, a direct application of Chernoff's bound shows that, for sufficiently large $Z$, a cube has more than $Z \log(1/r)$ vertices in $(X_i)_{i = 1}^{2t_{\max}}$ at distance at most $100dUr$ from itself with probability $o(|\cT|^{-1})$. A union bound over $|\cT|$ cubes shows~\ref{item:H2}.

Next, reveal the positions of the points $\{X_i: i\in [2t_{\min}]\}$ and $\{X_i: i\in [2t_{\min}+1,2t_{\max}]\}$ without revealing their labels, and condition on the events \ref{item:H1} and \ref{item:H2}.
Note that, by doing so, the labels of the said points remain uniformly distributed.
Then, by a union bound,
\[\mathbb P(D_M^c)\le r^{\alpha-d} (Z\log(1/r))^2\cdot \max\bigg\{t_{\min}\cdot  \binom{2t_{\min}}{2}^{-1},(t_{\max}-t_{\min})\binom{2t_{\max}-2t_{\min}}{2}^{-1}\bigg\} = o(1),\]
where the first term chooses an $M$-nonfull cube $q$ wrt $(X_i)_{i = 1}^{2t_{\min}}$, the second term chooses a pair of points at distance at most $100dr \leq 100dUr$ from $q$ simultaneously in $\{X_i: i\in [2t_{\min}]\}$ or in $\{X_i: i\in [2t_{\min}+1,2t_{\max}]\}$, and the last term dominates the probability that these points form a partner pair.
\end{proof}

The next two lemmas are expressed in terms of the finer tessellation $\cT_f$ and focus on partner pairs near $\kappa$-remote sets for different values of $\kappa$. 
For the first lemma, denote by $E_k$ the event that there do not exist points $Z,Y,\Bar{Y} \in (X_i)_{i = 1}^{2\taukctwo{k}}$ with $Y,\Bar{Y}$ being a partner pair, and $k$-remote connected sets $\cC, \Bar{\cC}$ in $\cG_f$ wrt $(X_i)_{i = 1}^{2\taukctwo{k}}$ such that $Z\in \bigcup \cC$ and $\max\{\dist(Y,Z), \dist(\bigcup \Bar{\cC}, \Bar{Y})\} \leq d_f$, where we recall the auxiliary length $d_f$ from \eqref{eqn:df}.

\begin{lemma}\label{lem:E}
$\bbP(E_k) = 1-o(1)$.
\end{lemma}
\begin{proof}
Denote by $B_k$ the event all connected $k$-remote sets wrt $(X_i)_{i = 1}^{2t_{\min}}$ have diameter at most $20dr$; this event holds whp by \Cref{rem:remote-in-far}.
Also, recall the event $A_k$, which holds with high probability by Lemma~\ref{lem:hitting-time-offline-ub}.
Then, \Cref{lem:D} implies
\[\mathbb P(E_k^c)\le \mathbb P(D_M^c)+\mathbb P(B_k^c)+\bbP(A_k^c)+\mathbb P(E_k^c\cap D_M\cap B_k \cap A_k) = o(1)+\mathbb P(E_k^c\cap D_M\cap B_k \cap A_k),\]
so it remains to estimate the last term.

To this end, we employ a union bound.
Fix two cubes $q,\bar q\in \cT_f$ and indices $j \in [2t_{\max}]$, $j' \in [t_{\max}]$, $j \notin \{ 2j' - 1, 2j'\}$. 
The cubes $q,\bar q$ serve to construct $\cC\ni q$ and $\bar\cC\ni \bar q$ and thus, assuming the event $E_k^c\cap D_M$, we can (and do) assume that $q, \bar q$ are at distance at least $50dr$ from each other.
We also define $t_{\text{aux}} := 2t_{\min} + \mathbbm{1}_{j \leq 2t_{\min} } + 2 \cdot \mathbbm{1}_{j' \leq t_{\min}}$.

Next, reveal the set of points $\bY := \{ X_i: i \leq t_{\text{aux}}, i \notin \{ j, 2j' - 1, 2j' \} \}$.
From the construction of $t_{\text{aux}}$ it follows that $|\bY| = 2t_{\min}$.
Assuming the event $A_k$, we have $\bY \subseteq (X_i)_{i = 1}^{2\taukctwo{k}}$. Consequently, by \Cref{rem:remote-monotonicity}, if each of $q,\bar q$ belongs to a $k$-remote set wrt $(X_i)_{i =1}^{2\taukctwo{k}}$, then they also belong to $k$-remote sets wrt $\bY$.
Since the indices $j, j'$ are fixed, the distribution of $\bY$ is the same as $(X_i)_{i=1}^{2t_{\min}}$, and we are in position to apply \Cref{lem:prob-far}(a) with $\kappa=k\le M$ and $m=2$.
As a result, the probability that each of $q,\bar q$ belongs to a $k$-remote set wrt $(X_i)_{i = 1}^{2\taukctwo{k}}$ is at most $r^d (\log(1/r))^{-4/5}$.

Furthermore, the probability that $X_j\in q$ is $|\cT_f|^{-1}$, and the probability that $X_{2j'-1},X_{2j'}$ are at distance at most $d_f$ from $q,\bar q$, respectively,
is at most $(s_f+2d_f)^{2d}\le (3d_f)^{2d}$.
All in all,
\[\mathbb P(E_k^c\cap D_M\cap B_k  \cap A_k)\le |\cT_f|^2\cdot 2t_{\max}^2\cdot \frac{r^d}{ (\log(1/r))^{4/5}}\cdot \frac{1}{|\cT_f|}\cdot (3d_f)^{2d} = O \left( \frac{t_{\max}(\log(1/r))^{1/5}  d_f^{2d} }{s_f^d} \right)=o(1),\]
where we used that $|\cT_f|=s_f^{-d}$, $t_{\max} r^d = O(\log(1/r))$ and $d_f^{2d}=r^d s_f^d (\log(1/r))^{-d+o(1)}$. This leads to the desired conclusion.
\end{proof}

Our last property considers pairs of $M$-remote connected sets $\cC,\bar\cC$ where $\cC$ has atypically large diameter.
Denote by $F_M$ the event that there do not exists a partner pair $Y, \Bar{Y} \in (X_i)_{i = 1}^{2\taukctwo{k}}$ and $M$-remote connected sets $\cC, \Bar{\cC} \subseteq \cT_f$ wrt $(X_i)_{i = 1}^{2\taukctwo{k}}$ with $\diam(\bigcup \cC) > d_f$ and $\max\{\dist(\bigcup \cC,Y), \dist(\bigcup \Bar{\cC},\Bar{Y})\} < 20dr$.

\begin{lemma}\label{lem:F}
$\bbP(F_M) = 1-o(1)$.
\end{lemma}

\begin{proof}
Denote by $B_M$ the event all $M$-remote connected sets wrt $(X_i)_{i = 1}^{2t_{\min}}$ have diameter at most $20dr$; this event holds whp by \Cref{rem:remote-in-far}. Then, using Lemmas \ref{lem:hitting-time-offline-ub} and \ref{lem:D} implies
\[\mathbb P(F_M^c)\le \mathbb P(D_M^c)+\mathbb P(B_M^c)+\bbP(A_k^c)+\mathbb P(F_M^c\cap D_M\cap B_M \cap A_k) = o(1)+\mathbb P(F_M^c\cap D_M\cap B_M \cap A_k),\]
so it remains to estimate the last term.
This is done by a union bound similar to the one from the proof of \Cref{lem:E}: indeed, by using \Cref{lem:prob-far}(b) and sparing the terms $2t_{\max}$ for the choice of $X_j$ and $|\cT_f|^{-1}$ for the probability that $X_j\in q$, we obtain
\[\mathbb P(F_M^c\cap D_M\cap B_M \cap A_k)\le |\cT_f|^2\cdot t_{\max}\cdot r^d (\log(1/r))^{-\omega(1)}\cdot (s_f+40dr)^{2d}=o(1),\]
as desired.
\end{proof}

\subsubsection{Construction of the choice set}

We proceed to the construction of the choice set $U$. Along the way, we assume the event $A_k\cap D_M \cap E_k \cap F_M$, which holds whp by Lemmas~\ref{lem:hitting-time-offline-ub}--\ref{lem:F}. 
As previously mentioned, we separate the construction in three parts: namely, $CS = CS_1\cup CS_2 \cup CS_3$ where 
\[ CS_1\cup CS_2 \subseteq \{X_i: i\in [2\taukctwo{k}]\} \quad \text{ and } \quad CS_3 \subseteq \{X_i: i\in [2\taukctwo{k} + 1,\infty]\}.\]

First, we present the construction of $CS_1$.

\vspace{0.5em}
\noindent
\textbf{Stage $1$: Vertices around far cubes.} In this stage, we consider the vertices near far cubes in $\cT$. 
Define the family $\fN_{hit} = \fN(\taukctwo{k})$ of components of $M$-nonfull cubes in $\cG$ wrt $(X_i)_{i = 1}^{2\taukctwo{k}}$. 
Note that, by the definition of $\cG$, the $2c$-blowups of the components in $\fN_{hit}$ are pairwise disjoint.
We denote by $\fN' \subseteq \fN_{hit}$ the family of the components $N$ surrounding at least one far cube wrt $(X_i)_{i = 1}^{2\taukctwo{k}}$, and write $F(N)$ to denote this set of far cubes. 
Further, define $\cB = \{X_i:i\in[2\taukctwo{k}]\} \cap \bigcup \{N_{2c}: N \in \fN'\}$. 
Our goal will be to ensure that the vertices in $\cB$ in $(k-1)$-separated sets in $G_{2\taukctwo{k}}$ do not end up in $CS$ but most of the remaining vertices in $\cB$ do: the latter will form the set $CS_1$.

We consider the structure of the partner pairs in $\cB$. By the event $D_M$, there is no partner pair with both points located in $N_{2c}$ for a single component $N\in \fN'$. 
The next lemma ensures another structural property. Given $N \in \fN'$, we define $P(N) := \{Y \in \{X_i:i\in[2\taukctwo{k}]\} \cap \bigcup N_{2c}: \Bar{Y} \in \cB \}$, that is, the set of vertices in $N_{2c}$ with partners in $\cB$. 
\begin{lemma}\label{lem:special}
The following events hold jointly whp:
\begin{enumerate}[label=\emph{(\alph*)}]
\item\label{clm:special-case-a} for every $N' \in \fN'$, $|P(N')| \leq 1$;
\item\label{clm:special-case-b} for every $N' \in \fN'$ such that $\diam(\bigcup F(N')) \geq r/10$, $|P(N')| = 0$.
\end{enumerate}
\end{lemma}
\begin{proof}
To begin with, we present some additional favourable properties of the graph $G_{2\taukctwo{k}}$.
These are expressed in terms of the far cubes in the coarse tessellation $\cT$.
More precisely, we will show that there exists $\beta = \beta(c, d)$ such that the following properties hold jointly whp:
\begin{enumerate}[\upshape{\textbf{E\arabic*}}]
    \item\label{item:H3} the number of far cubes $q \in \cT$ wrt $(X_i)_{i = 1}^{2\taukctwo{k}}$ is at most $r^{-d/2 - \beta}$,
    \item\label{item:H4} the number of cubes $q\in \cT$ such that $q \in F(N')$ for some $N'\in \fN'$ with $\diam(\bigcup F(N')) \geq r/10$ is at most $r^{-d/2 + 2\beta}$.
\end{enumerate}

Set $\beta = 3d^3/c$.
Fix $q \in \cT$ and assume that $q\in F(N')$ for some $M$-nonfull component $N' \in \fN'$.
By \Cref{rem:lem:3}, under the assumption of the event $A_k$ from \Cref{lem:hitting-time-offline-ub}, we have $q \in F(N)$ for some $N \in \fN$.
Then, by Lemma~\ref{lem:cutoff-in-2c-blowup}(b), we can (and do) assume that there is a cube $q'\in N$ such that $q, q'$ are adjacent in $\widetilde{\cG}$.
Furthermore, by Lemma~\ref{lem:nonfull-component-size}(a), we know that $|N| > (1 - 3d^2/c)U$.

Next, we compute the probability that such a component $N$ exists for a given cube $q$.
The number of connected subgraphs of $\cG$ adjacent to $q$ and with at most $(1 - 3d^2/c)U$ cubes is at most $\prod_{i = 1}^{U - 1} i(9c)^d \le (9Uc)^{Ud}$.
Thus, the probability that $q \in F(N)$ for some $N\in \fN$ as above is at most
\begin{align*}
(9Uc)^{Ud}\cdot \sum_{i=0}^{MU} \binom{2t_{\min}}{i} &\bigg(\frac{(1-3d^2/c)U}{(c/r)^d}\bigg)^i\bigg(1-\frac{(1-3d^2/c)U}{(c/r)^d}\bigg)^{2t_{\min}-i}\\
&\le (t_{\min} r^d)^{O(1)} \exp(-(1 - 4d^2/c)U\cdot 2t_{\min} \cdot (r/c)^d) = o(r^{d/2 - \beta}).
\end{align*}
Applying Markov's inequality shows \ref{item:H3}.

For \ref{item:H4}, we follow the same strategy, the only difference being that part (b) of Lemma~\ref{lem:nonfull-component-size} implies a stronger condition $|N| > (1 + 9d^2/c)U$.
Hence, the probability that $q$ satisfies \ref{item:H4} is dominated by
\[ (t_{\min} r^d)^{O(1)} \exp{( -(1 + 8d^2/c)U \cdot 2t_{\min} \cdot r^d/c^d )} = o(r^{d/2 + 2\beta}), \]
which, together with Markov's inequality gives \ref{item:H4}.

We now proceed with the proof of Lemma~\ref{lem:special}. For (a), recall the typical event from \Cref{lem:3-counterpart} (which we call $K_M$ in this proof), \ref{item:H2} and the constant $Z$ defined there.
We condition on the sets of point $\{X_i: i\in [2t_{\min}]\}$ and $\{X_i:i\in [2t_{\min}+1,2t_{\max}]\}$ without revealing the labels, and on the events \ref{item:H2}, $K_M$ and \ref{item:H3}. In particular, the labels of the points within the point sets above remain uniformly distributed.
In the sequel, we use the following claim. 
\begin{claim}
Fix a component $N' \in \fN'$, a far cube $q \in F(N')\subseteq \cT$, and a point $x$ in the region $\bigcup N'_{2c}$. Then, $\mathrm{dist}(q, x) < 100dUr$.
\end{claim}
\begin{proof}
Conditionally on $K_M$ and by the monotonicity exhibited in \Cref{rem:lem:3}, we have $|N'| \leq 2U$ which in turn implies $\text{diam}(\bigcup N') < 8dUr$.
Furthermore, under the above setup and using \Cref{lem:cutoff-in-2c-blowup}(b), we obtain that $\text{dist}(q, \bigcup N') < r$. 
Indeed, since the cube $q$ was far wrt $(X_i)_{i=1}^{2t_{\min}}$, \Cref{lem:cutoff-in-2c-blowup}(b) implies that some close cube $q'$ wrt $(X_i)_{i=1}^{2t_{\min}}$ is adjacent to $q$ in $\widetilde{\cG}$.
Since $q$ is a far cube with respect to $(X_i)_{i=1}^{2\taukctwo{k}}$ as well, $q'$ must have remained a close cube.
Combining the latter conclusions with the fact that $\text{dist}(x, \bigcup N') < 2r$ and the triangle inequality conclude the proof.
\end{proof}
A union bound over all far cubes, whose number is bounded by the conditioning on \ref{item:H3}, and pairs of partner pairs shows that (a) fails with probability at most
\begin{equation}\label{eq:bd4.20}
r^{-d/2-\beta}\cdot (2t_{\max})^2\cdot \bigg(\frac{Z \log(1/r)}{2t_{\max} - 2t_{\min}}\bigg)^2\cdot (r^{-d/2-\beta})^2 \bigg(\frac{Z \log(1/r)}{2t_{\max} - 2t_{\min}}\bigg)^2=o(1).
\end{equation}
Here, the first and the fourth terms dominate the number of choices of three components of far cubes (some of which may coincide) by \ref{item:H3}, the second term dominates the number of pairs of pairs of partner points, and the third and the fifth terms dominate the probability that the points in the chosen pairs land within distance $100dUr$ from the fixed far cubes in $\cT$ by \ref{item:H2} and $K_M$.

For (b), conditioning further on the event \ref{item:H4}, a computation similar to~\eqref{eq:bd4.20} shows that (b) fails with probability at most 
\begin{equation*}
r^{-d/2+2\beta}\cdot 2t_{\max}\cdot \bigg(\frac{Z \log(1/r)}{2t_{\max} - 2t_{\min}}\bigg)\cdot r^{-d/2-\beta} \bigg(\frac{Z \log(1/r)}{2t_{\max} - 2t_{\min}}\bigg)=o(1),
\end{equation*}
as desired.
\end{proof}

From now on, we assume that the statement of \Cref{lem:special} holds. 
We construct $CS_1$ by consecutively processing the components in $N \in \fN'$ along an arbitrary ordering $\sigma$ of $\fN'$.
At every step, we use the following \emph{construction algorithm}:

\begin{algorithm}\label{algo:U1}
Fix $N\in \fN'$. First, assume that $|P(N)|=0$. We consider two cases:
\begin{enumerate}[\upshape{\textbf{F\arabic*}}]
    \item\label{item:I1} If $\diam(\bigcup F(N)) \geq r/10$, add all vertices in $N_{2c}\setminus F(N)$ and the partners of all points in $F(N)$ to $CS_1$.
    \item\label{item:I2} If $\diam(\bigcup F(N)) < r/10$, check if there is a $(k-1)$-separated set of vertices inside $F(N)$. If not, add all vertices in $N_{2c}$ to $CS_1$. Otherwise, iteratively and as long as possible, delete all vertices in $(k-1)$-separated sets within $N_{2c}$ (note that deleting vertices may lead to the formation of new such sets). Add all non-deleted vertices and the partners of all deleted vertices in $N_{2c}$~to~$CS_1$.
\end{enumerate}
Next, we assume that $|P(N)|=1$ and consider a pair $Y,\bar Y$ with $Y\in N_{2c}$ and $\bar Y\in \bar N_{2c}$ with $\bar N\in \fN'$. 
If $\bar N$ precedes $N$ in $\sigma$, proceed to the next step.
Otherwise, we distinguish several cases:
\begin{enumerate}[\upshape{\textbf{F\arabic*}},resume]    \item\label{item:I3} If $Y \notin F(N)$ and $\Bar{Y} \notin F(\Bar{N})$, we add all vertices in $(N_{2c}\cup \bar N_{2c})\setminus (\{\bar Y\}\cup F(N)\cup F(\bar N))$ and the partners of all vertices in $F(N)\cup F(\bar N)$ to $CS_1$.
    \item\label{item:I4} If $Y \notin F(N), \Bar{Y} \in F(\Bar{N})$ or vice-versa, we add all vertices in $(N_{2c}\cup \bar N_{2c})\setminus (F(N)\cup F(\bar N))$ and the partners of all vertices in $F(N)\cup F(\bar N)$ to $CS_1$.
    \item\label{item:I5} If $Y \in F(N), \Bar{Y} \in F(\Bar{N})$ and neither $F(N)$ nor $F(\bar N)$ contain any $(k-1)$-separated sets, we add all vertices in $(N_{2c}\cup \bar N_{2c})\setminus F(N)$ and the partners of all vertices in $F(N)$ to $CS_1$.
    \item\label{item:I6} If $Y \in F(N), \Bar{Y} \in F(\Bar{N})$ and one of $F(N),F(\bar N)$ contains a $(k-1)$-separated set (say $F(\tilde N)$ with $\tilde N\in \{N,\bar N\}$) but the other does not, we add all vertices in $(N_{2c}\cup \bar N_{2c})\setminus F(\tilde N)$ and the partners of all vertices in $F(\tilde N)$ to $CS_1$.
    \item\label{item:I7} If $Y\in F(N), \Bar{Y}\in F(\Bar{N})$ and each of $F(N),F(\bar N)$ contains a $(k-1)$-separated set, denote by $V$ (respectively $\bar V$) all vertices deleted during the procedure described in \ref{item:I2} applied for $F(N)$ (respectively for $F(\bar N)$). Then, we add all vertices in $(N_{2c}\cup \bar N_{2c})\setminus (V\cup \bar V)$ and the partners of all vertices in $V\cup \bar V$ to $CS_1$.
\end{enumerate}
\end{algorithm}

Next, we show that $CS_1$ contains exactly one vertex from each partner pair intersecting $\cB$.

\begin{lemma}\label{lem:algo_good}
Whp no iteration of the algorithm forces $|\{X,\bar X\}\cap CS_1|\neq 1$ for a partner pair $X,\bar X$.
\end{lemma}
\begin{proof}
The statement is clear for \ref{item:I1}--\ref{item:I6}. For \ref{item:I7}, the only partner pair which may simultaneously be included in $CS_1$ (and, equivalently, in $V\cup \bar V$) is $Y,\bar Y$. 
We condition on the typical events $E_{k}$ from \Cref{lem:E}, $F_{M}$ from \Cref{lem:F} and $H_M$ from \Cref{lem:special}(b), and show that the said inclusion does not hold.

Suppose for contradiction that $Y\in V$ and $\bar Y\in \bar V$. Fix a $(k-1)$-separated set $S\subseteq V$ (respectively $\bar S\subseteq \bar V$) which does not necessarily contain $Y$ (respectively $\bar Y$).
By definition of the hitting time $\taukctwo{k}$, it cannot occur that $S = \{ Y\}$ and $\bar S = \{ \bar Y\}$.
Without loss of generality, suppose that there is $Z \in S$ different from $Y$. We consider two cases.

First, by recalling $d_f$ from~\eqref{eqn:df}, suppose that $\max\{\diam(S \cup \{ Y \}), \diam(\bar S \cup \{\bar Y\})\}\leq d_f$.
We denote by $\cC \subseteq \cT_f$ a smallest set of cubes covering $S$, and by $\Bar{\cC} \subseteq \cT_f$ a smallest set of cubes covering $\Bar{S}$.
In particular, $Z \in \bigcup \cC$ and, therefore, the event $E_k^c$ from Lemma~\ref{lem:E} holds, contradicting our assumption.

Now, suppose that $\max\{\diam(S \cup \{ Y \}), \diam(\bar S \cup \{\bar Y\})\} > d_f$. As $S \cup \{ Y \} \subseteq V$ and $\bar S \cup \{\bar Y\} \subseteq \bar V$, we have $\max\{\diam(V),\diam(\bar V)\} > d_f$.
On the other hand, by $H_M$, we have $\max\{\diam(V),\diam(\bar V))\} < r/10$, implying that $G_{2\taukctwo{k}}$ restricted to each of $V, \bar V$ is a complete graph.
From this observation and the fact that $S,\bar S$ are $(k-1)$-separated, we deduce that $\max\{|V \setminus S|, |\Bar{V} \setminus \Bar{S}|\}\le k-1$.
Moreover, any vertex in $V, \Bar{V}$ has less than $k$ neighbours outside of $V, \Bar{V}$, respectively, as otherwise it could not participate in a $(k-1)$-separated set at any point during the deletion procedure in \ref{item:I7}.
As a result, $V$ has at most $|N(S)|+(k-1)|V\setminus S| < k^2 \leq M$ neighbours in its complement and the same holds for $\bar V$.
Defining $\cC \subseteq \cT_f$ as the smallest set of cubes covering $V$, and defining analogously $\bar \cC \subseteq \cT_f$ for $\bar V$, the event $F_M^c$ from Lemma~\ref{lem:F} holds, contradicting our assumption again and finishing the proof.
\end{proof}

Before we proceed to the second stage, we describe the picture outside $\cB$ and show that, roughly speaking, the partners of the vertices outside $CS_1$ do not form dense clusters.
Recall $\alpha$ from \ref{item:H1}.

\begin{lemma}\label{clm:stage-1-balance}
Whp there is no set $I \subseteq \{X_i:i\in [2t_{\max}]\}\setminus \cB$ with size $|I| \ge 2d/\alpha$ such that simultaneously $\diam(\{X_i:i \in I\}) < 100dr$ and $\{\bar X_i:i \in I\} \subseteq \bigcup \{N_{2c}: N\in \fN'\}$.
\end{lemma}
\begin{proof}
Suppose that a set $I$ as described exists (call such set \emph{bad} for short). 
Since every subset $I'$ of a bad set $I$ with $|I'|\ge 2d/\alpha$ is also bad, we can (and do) assume that $|I|=2d/\alpha$.

We condition on the point sets $\{\bar X_i:i \in [2t_{\min}]\}$ and $\{\bar X_i:i \in [2t_{\min}+1,2t_{\max}]\}$ without revealing their labels: this is enough to determine $\fN$.
Also, condition on the fact that the said point configuration satisfies the typical events \ref{item:H1} and \ref{item:H2}.
Note that the distribution of the indices in the two point sets remains uniform.
Observe that if $\diam(\{X_i:i \in I\}) < 100dr$, then in particular there exists $q \in \cT$ such that $\text{dist}(X_i, q) \leq 100dr$ for each $i \in I$.
Also, the inclusion $\{\bar X_i:i \in I\} \subseteq \bigcup \{N_{2c}: N\in \fN'\}$ together with \Cref{rem:lem:3} imply that, for each $i \in I$, there exists an $M$-nonfull cube $q_i \in \cT$ wrt $(X_i)_{i =1 }^{2t_{\min}}$ such that $\text{dist}(\bar X_i, q_i) \leq 100dr$.
All in all, the probability that there exists a bad set $I$ of size $2d/\alpha$ is at most
\[\binom{2t_{\max}}{2d/\alpha} \cdot |\cT| \cdot O\bigg(\frac{(\log(1/r))^{2d/\alpha}}{(t_{\max}-t_{\min}})^{2d/\alpha}\bigg)\cdot r^{(\alpha - d) \cdot 2d/\alpha} \cdot O\bigg(\frac{(\log(1/r))^{2d/\alpha}}{(t_{\max}-t_{\min}})^{2d/\alpha}\bigg) = r^{d+o(1)} = o(1),\]
where the first factor selects a set of indices $I$, the second factor selects the cube $q$, the third factor uses that $t_{\max}-t_{\min} = \min\{t_{\max}-t_{\min},t_{\min}\}$ to dominate the probability that all the points in $\{ X_i: i \in I\}$ are located close to $q$, the fourth factor selects the cubes $q_i$ for each $i \in I$, and the last factor dominates the probability that all points in $\{ \bar X_i: i \in I \}$ land close to the respective $q_i$.
\[\binom{2t_{\max}}{2d/\alpha} \cdot |\cT| \cdot ((300dr)^d)^{2d/\alpha}\cdot (r^{\alpha - d})^{2d/\alpha} \cdot ((300dr)^d)^{2d/\alpha} = r^{d+o(1)} = o(1),\]
where the first factor selects a set of indices $I$, the second factor selects the cube $q$, the third factor dominates the probability that all the points in $\{ X_i: i \in I\}$ are located within distance $100dr$ from $q$, the fourth factor selects the cubes $q_i$ for each $i \in I$, and the last factor dominates the probability that all points in $\{ \bar X_i: i \in I \}$ land within distance $100dr$ from the respective $q_i$.
\end{proof}

\vspace{0.5em}
\noindent
\textbf{Stage $2$: Remaining vertices.} So far, we have processed the few partner pairs containing a point near far cubes in $\cT$ wrt $(X_i)_{i=1}^{2\taukctwo{k}}$.
Next, we construct the set $CS_2 = (\{X_i:i\in [2\taukctwo{k}]\}\cap CS)\setminus CS_1$: note that its vertices belong to close cubes or cubes in the sea.
Our construction aims to not significantly alter the geometry of the graph $G_{2\taukctwo{k}}$; in particular, all $M$-full cubes end up having roughly half of the vertices contained in them in the choice set $CS$.

More formally, we define an auxiliary multigraph $\mathcal{M}$ with vertex set $\mathcal{T}$ where every edge $q\bar q$ corresponds to a partner pair $Y,\bar Y\in \{X_i:i\in [2\taukctwo{k}]\}$ with $Y\in q$ and $\bar Y\in \bar q$.
The next lemma appearing as Exercise~1.4.23 in \cite[Chapter~1]{Wes01} will be a main tool in our construction. We provide a proof for completeness.

\begin{lemma}\label{lem:West}
For every unoriented multigraph $\cM$, there exists an orientation of its edges such that, for every vertex $v$, its in-degree and out-degree differ by at most $1$.
\end{lemma}
\begin{proof}
We show the statement for connected multigraphs $\cM$; the general case follows by addressing every component separately.
Partition the (even number of) odd-degree vertices in $\cM$ arbitrarily into pairs $(u_1,v_1),\ldots,(u_l,v_l)$, and define a multigraph $\cM'$ by adding to $\cM$ one edge between each pair $(u_i,v_i)$. Note that every vertex has even degree in $\cM'$.
Then, find an Euler walk in $\cM'$ and orient it consistently; note that this task can be done by following a greedy algorithm (see e.g.\ \cite[Theorem~1.8.1]{Die05}). Finally, the orientation induced on $\cM$ has the desired property: indeed, the in-degree and the out-degree of every vertex in $\cM'$ are equal, and every vertex in $\cM'$ is incident to at most one edge missing in $\cM$.
\end{proof}

We are ready to finish the construction of $CS_2$. Fix $M\ge 3k^2+2d/\alpha$ with $\alpha$ as in \ref{item:H1} and condition on the statement of \Cref{clm:stage-1-balance}.

\paragraph{Constructing $CS_2$.} Note that, thanks to \Cref{clm:stage-1-balance}, every sea cube wrt $(X_i)_{i=1}^{2\taukctwo{k}}$ contains at least $M-2d/\alpha\ge 2k$ points $Y$ with $\bar Y$ belonging to a sea cube.
Consider the restriction $\cM_s$ of the multigraph $\cM$ to the cubes in $\cT\setminus \bigcup \{N_{2c}:N\in \fN'\}$ and note that $\cM_s$ has minimum degree at least $2k$. 
Fix an orientation $\sigma$ given by \Cref{lem:West} and, for every partner pair $Y\in q,\bar Y\in \bar q$ in the sea with the edge $q\bar q$ oriented from $q$ to $\bar q$ in $\sigma$, include $Y$ in $CS_2$.

\begin{lemma}\label{lem:CS1CS2}
Whp $G_{2\tau_{2,k}}[CS_1 \cup CS_2]$ is a $k$-connected graph.
\end{lemma}
\begin{proof}
Denote by $L_k$ the event that every sea cube in $G_{2\tau_{2,k}}$ contains at least $k$ vertices in $CS_1\cup CS_2$. 
By \Cref{algo:U1}, \Cref{clm:stage-1-balance} and the construction of $CS_2$, $L_k$ holds whp and we assume it in the sequel. 
Recalling that the sea is a connected subgraph of $\widetilde{\cG}$ by \Cref{lem:sea-connected}, the restriction of $G_{2\tau_{2,k}}[CS_1\cup CS_2]$ to the sea is a $k$-connected graph. 
Furthermore, by \Cref{cl:Hav}, adding back the vertices in close cubes one by one maintains $k$-connectivity. Denote by $H$ the $k$-connected graph obtained by restricting $G_{2\tau_{2,k}}[CS_1\cup CS_2]$ to the close and the sea cubes.

Now, suppose that $G_{2\tau_{2,k}}[CS_1 \cup CS_2]$ has a $(k-1)$-separator $S$. Then, $H\setminus S$ is a connected graph. 
Since the clusters of far cubes $F(N)$ for $N\in \fN'$ are at distance at least $4r$ from each other, $S$ must separate some of the vertices in a single cluster $F(N)$ from the rest of the graph (and from $H$ in particular).
However, by \ref{item:I2}, \ref{item:I6} and \ref{item:I7}, no such $(k-1)$-separated set exists, which concludes the proof.
\end{proof}

\vspace{0.5em}

We turn to the construction of $CS_3$ via a simple online strategy. 
Define $CS^{\taukctwo{k}}=CS_1\cup CS_2$ and, for every integer $i > \taukctwo{k}$, we inductively define $CS^i=CS^{i-1} \cup \{X_{2i-1}\}$, if $X_{2i-1}$ has at least $k$ neighbours in $CS^{i-1}$, and $CS^i=CS^{i-1} \cup \{X_{2i}\}$ otherwise. 
Finally, we define 
\[\tau = \min\{i > \taukctwo{k}: G(CS^i) \text{ is not $k$-connected}\}.\] 
To finish the proof, we show that whp $\tau=\infty$. Note that, by  Claim~\ref{cl:Hav} and the fact that $G_{2\tau_{2,k}}[CS_1 \cup CS_2]$ is typically $k$-connected, 
it suffices to show that whp, for every $i>\taukctwo{k}$, the partner pair $X_{2i-1},X_{2i}$ contains a vertex with at least $k$ neighbours in $CS^{i-1}$: 
note that this event certainly holds unless both points land inside $k$-remote cubes in $\cT_f$ with respect to the point set $CS_1\cup CS_2$. Our next lemma estimates the total volume of these cubes.

\begin{lemma}\label{lem:k-far-volume-u1}
There is $\delta = \delta(c, d) > 0$ such that whp the following statements hold simultaneously:
\begin{enumerate}[label=\emph{(\alph*)}]
\item the total volume of the far cubes in $F(N)$ over $N \in \fN'$ with $\diam(\bigcup F(N)) \geq r/10$ is $o(r^{d/2 + \delta})$;
\item the total volume of the far cubes in $F(N)$ over $N \in \fN'$ with $|P(N)| = 1$ is $o(r^{d/2 + \delta})$;
\item the total volume of the $k$-remote cubes with respect to $CS_1\cup CS_2$ is $o(r^{d/2} (\log(1/r))^{-\delta})$.
\end{enumerate}
\end{lemma}

\begin{proof}
Recall \ref{item:H4} and the $\beta$ defined there.
It follows that (a) is satisfied for any $\delta < 2\beta$.

For part (b), recall the typical properties \ref{item:H2}, \ref{item:H3} as well as the event from Lemma~\ref{lem:3-counterpart} (which we call $K_M$ in this proof).
Condition on the point sets $\{ X_i: i \in [2t_{\min}]\}$ and $\{ X_i: i \in [2t_{\min} + 1, 2t_{\max}]\}$ without revealing the labels, and on the events $K_M$, \ref{item:H2}, \ref{item:H3}.
The labels of the points within the respective sets remain uniformly distributed.
By a computation analogous to \eqref{eq:bd4.20}, the expected volume considered in (b) does not exceed
\[ \frac{r^d}{c^d} \cdot (r^{-d/2 - \beta})^2 \cdot 2t_{\max} \cdot \left( \frac{Z \log(1/r)}{2t_{\max} - 2t_{\min}} \right)^2 = \Theta \left(r^{d - 2\beta} \frac{(\log(1/r))^3}{(\log\log(1/r))^2}\right), \]
where the first factor represents the volume of a single cube in $\cT$, the second factor dominates the number of choices of two hypercubes, the third one dominates the number of choices of a pair of indices and the last one is the upper bound for the probability that said a pair of indices is assigned to the points at distance at most $100dr$ from the selected cubes.
Due to Markov's inequality, statement (b) holds whp for any choice of $\delta < d/2 - 2\beta$.
Since (a) and (b) yield stronger bounds than (c), we only estimate the total volume of $k$-remote sets inside far cubes in $\cT\setminus A(N)$ where $N\in \fN'$ satisfies $|P(N)| = 0$ and $\text{diam}(\bigcup F(N)) < r/10$. In particular, the graph $G_i$ restricted to $F(N)$ is complete for all $i\ge 1$.
    
    If $F(N)$ does not contain any $(k-1)$-separated set in $G_{2\taukctwo{k}}$, by construction, all vertices in $N_{2c}$ belong to $CS_1\subseteq CS$. Therefore, we could not have created or enlarge any $k$-remote set there by following \Cref{algo:U1} and, by Proposition~\ref{prop:main-volume}, the total volume of such $k$-remote sets is at most $r^{d/2} (\log(1/r))^{-1/3} = o(r^{d/2} (\log(1/r))^{-\delta})$.

    If $F(N)$ does contain a $(k-1)$-separated set $S$, then some of the vertices inside have been deleted following the process in \ref{item:I2}, and thus \Cref{prop:main-volume} alone is insufficient to conclude as new $k$-remote sets may have been generated. To address this inconvenience, define $W=(\{X_i: i\in [2\taukctwo{k}]\}\cap \bigcup F(N))\setminus S$. Since $\diam(\bigcup F(N)) < r/10$, the set of vertices $S \cup W$ induces a clique in $G_{2\taukctwo{k}}$, and thus $|W|\le k-1$ as otherwise $S$ would not have been $(k-1)$-separated in the first place.
    
    Next, denote by $\cC\subseteq \cT_f$ a minimal set of cubes covering $S$ entirely and fix an arbitrary set $\cC'\subseteq \cT_f$ within $\bigcup F(N)$ which becomes $k$-remote throughout the execution of \Cref{algo:U1}. Since $\cC\cup \cC'$ covers all vertices in $S$, there are at most $|W|\le k-1$ vertices in $N_{2c}$ outside $S$ deleted throughout the iterations of \ref{item:I2}. 
    As a result, $\cC\cup \cC'$ is a $3k$-remote set wrt $(X_i)_{i=1}^{2\taukctwo{k}}$. Two options arise: 
    \begin{itemize}
        \item If $\diam(\bigcup (\cC\cup \cC')) \ge d_f$, then $\cC\cup \cC'$ is a $\cG_f$-connected set with large diameter. The total volume of such sets is $r^{d/2}(\log(1/r))^{-\omega(1)}$ by \Cref{prop:main-volume}.
        \item If $\diam(\bigcup (\cC\cup \cC')) < d_f$, then $\cC'$ is within $\ell_{\infty}$-distance $d_f$ from any cube in $\cC$. Therefore, the total number of cubes in sets $\cC'$ is within a factor of $O((d_f/s_f)^d)\le (\log\log(1/r))^{2d}$ from the number of cubes in $\cC$.
    \end{itemize}
    Using the already mentioned upper bound of $r^{d/2}(\log(1/r))^{-1/3}$ on the total volume of $k$-remote sets, the total volume of $F(N)$ containing a $(k-1)$-separated set, with $|P(N)| = 0$ and $\text{diam}(\bigcup F(N)) < r/10$ is at most 
    \[r^{d/2} (\log(1/r))^{-\omega(1)} + r^{d/2} (\log(1/r))^{-1/3} (\log\log(1/r))^{2d} = o(r^{d/2} (\log(1/r))^{-\delta})\]
    for any $\delta < 1/3$, as desired.
\end{proof}    

We are ready to finish the proof of the $k$-connectivity part of \Cref{thm:2} by exhibiting the claimed monotonicity.

\begin{proof}[Proof of the $k$-connectivity part in \Cref{thm:2}]
We already showed in \Cref{lem:CS1CS2} that $G_{2\taukctwo{k}}[CS^{\taukctwo{k}}]$ is a $k$-connected graph whp.
Denote by $\cT_r \subseteq \cT_f$ the set of all cubes belonging to some $k$-remote set in $CS^{\taukctwo{k}}$. 
We condition on the set $\cT_r$ and on the assumption on $|\cT_r|$ from Lemma~\ref{lem:k-far-volume-u1}(c) holds.

Define $t_1 = C r^{-d} \log\log(1/r)$ and $t_2 = t_1 + C r^{-d} \log(1/r)$ for a large constant $C$ to be determined later.
We analyse the intervals $\{X_i:i\in [2\taukctwo{k}+1,2(\taukctwo{k}+t_1)]\}$ and $\{X_i:i\in [2(\taukctwo{k}+t_1)+1,2(\taukctwo{k}+t_2)]\}$ separately.
First, we show that whp no partner points in the first interval simultaneously land in cubes contained in $k$-remote sets wrt $CS^{\taukctwo{k}}$: indeed, thanks to Lemma~\ref{lem:k-far-volume-u1}(c) and a union bound, the probability of the latter event is at most $t_1 \cdot r^d (\log(1/r))^{-2\delta} = o(1)$.

Next, by \Cref{lem:cutoff-in-2c-blowup}(b), for every fixed far cube (and, in particular, every cube $q\in \cT_r$ in some far region $F(N)$), there is a close cube $q'\in \cT$ within distance $(1-d/c)r$ from $q'$.
Furthermore, for every partner pair $X_{2i-1},X_{2i}$ with $i\in [\taukctwo{k}+1,\taukctwo{k}+t_1]$ where $X_{2i-1}$ lands in $q'$, $X_{2i-1}$ automatically has $k$-neighbours in $CS^{\taukctwo{k}}$ within a nearby sea cube (wrt $CS^{\taukctwo{k}}$). By definition of our strategy, every such point $X_{2i-1}$ is added to the set $CS^i$.
Note that, upon the latter event happening $k$ times for fixed $q,q'$, the cube $q$ cannot participate in any $k$-remote set in $F(N)$ ever again.
Furthermore, the latter event fails with probability
\begin{equation}\label{eq:Tr}
\mathbb P(\mathrm{Bin}(t_1,(c/r)^d)\le k) = \sum_{i=0}^{k-1} \binom{t_1}{i} \bigg(\frac{r}{c}\bigg)^{id} \bigg(1-\bigg(\frac{r}{c}\bigg)^d\bigg)^{t_1-i} = O((r^d t_1)^{k-1}\exp(-(1+o(1))r^d t_1/c^d)).
\end{equation}
As a result, by choosing the constant $C$ suitably large and using \Cref{lem:k-far-volume-u1}(c), Markov's inequality implies that the set $\cT_r'\subseteq \cT_r$ of $k$-remote cubes wrt $CS^{\taukctwo{k}+t_1}$ satisfies $|\cT_r'|\le r^{-d/2} /\log(1/r)$ whp.

Now, consider $CS^{\taukctwo{k} + t_2} \setminus CS^{\taukctwo{k} + t_1}$ conditionally on the set $\cT_r'$ and the latter event.
Since the total volume of the cubes in $\cT_r'$ is $o(r^{d/2}/\log(1/r))$, the expected number of partner pairs simultaneously landing in them is bounded from above by $C r^{-d} \log(1/r) \cdot r^d (\log(1/r))^{-2} = o(1)$.
Moreover, similarly to~\eqref{eq:Tr}, the expected number of cubes $q\in \cT_r'$ such that less than $k$ vertices with odd indices $X_{2j-1}$ with $j\in [\taukctwo{k} + t_1+1,\taukctwo{k} + t_2]$ land in the corresponding close cube $q$ is 
\[\sum_{i=0}^{k-1} \binom{t_2-t_1}{i} \bigg(\frac{r}{c}\bigg)^{id} \bigg(1-\bigg(\frac{r}{c}\bigg)^d\bigg)^{t_2-t_1-i} = O((r^d (t_2-t_1))^{k-1}\exp(-(1+o(1))r^d (t_2-t_1)/c^d)) = o(|\cT'_r|^{-1}),\]
where the last equality holds for large enough $C$. An application of Markov's inequality finishes the proof.
\end{proof}

\subsection{\texorpdfstring{Hamiltonicity in the offline $2$-choice process}{Hamiltonicity in the offline 2-choice process}}\label{sec:Ham-2off}

In this section, we prove the part of \Cref{thm:2} regarding Hamiltonicity.
We use the construction of the choice set $CS$ described in \Cref{sec:k-conn-2off} for $k = 2$ with very few minor modifications, and we argue that it suffices to show that the reference construction described in \Cref{subsection-hamiltonicity} whp produces a Hamilton cycle.

\paragraph{Construction of $CS_1$.} Recall the constants $c$ from the definition of the tessellation $\cT$, $M$, $U$ from \Cref{lem:3-counterpart} and $\alpha$ from \ref{item:H1}.
In this section, we set $M = 2(4U + 6 + 2(2c + 1)^d) + 2d/\alpha$. The remaining part of the construction remains unchanged.

\paragraph{Construction of $CS_2$.} There are no modifications here. 
Observe that due to the choice of $M$ in the previous steps, \Cref{clm:stage-1-balance} and \Cref{lem:West}, after this step every sea cube outside the blow-ups $N_{2c}$ with $N\in \fN'$ contains at least $4U + 6 + 2(2c + 1)^d$ points of $CS_1 \cup CS_2$.

\paragraph{Construction of $CS_3$.} In the previous section, we have discarded a point $X_{2i - 1}$ for $i > \taukctwo{2}$ only if it had less than $2$ neighbours in $CS^{i - 1}$.
For our purposes here, we discard the point $X_{2i-1}$ if it has less than $2$ neighbours in $CS^{i-1}$, or if it has at least $2$ neighbours but it lands in a cube $q$ such that $q\in F(N)$ for some $N\in \fN'$ with $\diam(\bigcup F(N)) \geq r/10$.
In other words, we keep unusually large clusters of far cubes empty for longer, with the idea to avoid dealing with them in the reference construction.

\paragraph{$2$-connectivity.} In the modified construction, it still holds that whp each of $(G_{2i}[CS^i])_{i\ge \taukctwo{2}}$ is $2$-connected. The modifications go as follows.

\paragraph{Reference construction at time $t \geq \taukctwo{2}$.}
The correctness of steps 1-3 remains valid,
and the construction of the far-reaching paths in step 4 remains unmodified for all $N \in \fN'$ with $\diam(\bigcup F(N)) < r/10$.
Moreover, recalling $t_2=(1+o(1))Cr^{-d}\log(1/r)$ from the proof of \Cref{lem:k-far-volume-u1} and assuming \Cref{lem:k-far-volume-u1}(a),
we derive that whp the far sets $F(N)$ with $\diam(\bigcup F(N))\geq r/10$ remain disjoint from $CS^{\taukctwo{2} + t_2}$.
To finish the construction, we argue that typically every cube in the far sets $F(N)$ described above has become close in $CS^{\taukctwo{2} + t_2}$.

\begin{claim}
Whp, for every cube $q \in \cT$ with $q\in F(N)$ for some $N\in \fN'$, there exists a close cube $q' \in \cT$ in $G_{2\taukctwo{2}}$ such that $q'$ contains more than $M$ points of $CS^{\taukctwo{2} + t_2}$.
\end{claim}

\begin{proof}
Denote by $E$ the event from the statement from the lemma, recall the typical event $A_2$ from \Cref{lem:hitting-time-offline-ub}, and define $E'$ to be the event that every far cube in $G_{2\taukctwo{2}}$ has a close neighbour in $\widetilde{\cG}$. 
By the monotonicity of the far region (implying that far cubes in $G_{2\taukctwo{2}}$ are also far cubes in in $G_{2t_{\min}}$) and $A_2$, $E'$ holds whp thanks to \Cref{lem:cutoff-in-2c-blowup}(b).
Thus, it suffices to show that $\bbP(E^c \cap A_2 \cap E') = o(1)$.

Assuming $A_2 \cap E'$, fix a far cube $q$ in $G_{2\taukctwo{2}}$, and let $q' \in \cT$ be a close neighbour of $q$ in $\widetilde{\cG}$.
Observe that every point $X_{2i - 1}$ with $i>\taukctwo{2}$ landing in $q'$ enters $CS_3$.
Thus, the probability that $q'$ contains at most $M$ points of $CS^{\taukctwo{2} + t_2}$ is dominated by
\[\bbP(\text{Bin}(t_2-t_1, r^d/c^d) \leq M) = o(r^{-d}),\]
where the asymptotic statement holds by \Cref{lem:chernoff} with $C$ sufficiently large.
A union bound over $\cT$ finishes the proof.
\end{proof}

Now, fix any $t\ge \taukctwo{2}$.
Once again, constructing the far-reaching paths crossing large far areas may be done while using at most $2$ vertices of $CS^t$ in each of the sea cubes.
To complete the justification of step 4 in the reference construction, it suffices to show that one can still construct paths absorbing all vertices in close cubes which remain outside the exhibited far-reaching paths.
To this end, recall that whp each of the $M$-nonfull components has size at most $2U$ (by \Cref{lem:3-counterpart}). 
Combining this with the monotonicity exhibited in \Cref{rem:lem:3}, for each component $N\in \fN'$, we need to use at most $4U$ vertices of $CS^t$ for the construction of the paths absorbing the vertices in $N$.
By our choice of $M$, using at most $4U$ vertices per sea cube leaves at least $2(2c + 1)^d$ vertices of $CS^t$ in each sea cube, which is the only condition required for the validity of step 5 of the reference construction.
This concludes the proof of \Cref{thm:2}.

\section{Sharp threshold for the online 2-choice process}\label{section:online-2-choice}

First of all, we focus on the subcritical regime where, for every simultaneous choice function $f$, we show that each of the graphs $G(\mathbf{X}_t(f))$ with $2t\le t_0 = t_0(r) := \mathbb E[\taukcone{1}]-(\theta_d^{-1}+\eps)r^{-d}\log\log(1/r)$ remains typically disconnected. 
Define $t_1 = t_1(r) := r^{-d}/\log\log(1/r)$.

\begin{proof}[Proof of the first statement in \Cref{thm:3}]
For every choice function $f$, the probability that the unique point in $\mathbf{X}_1(f)$ remains isolated in $G_{t_1}$ is at least
\[\mathbb P(X_1\text{ and } X_2 \text{ are isolated in }G_{t_1})\ge (1-2\theta_d r^d)^{t_1-2}=1-o(1).\]
In particular, for every choice function $f$, whp $G(\mathbf{X}_t(f))$ contains an isolated vertex for every $t$ such that $2t\in [t_1]$.
Next, we show that the same holds for every $t$ such that $2t\in [t_1,t_0]$.

Denote by $A$ the event that at least $0.9 t_1$ of the points $(X_i)_{i=1}^{t_1}$ have no other point in the same set at distance at most $2r$ from itself.
By~\eqref{eq:expZ} and \Cref{lem:conc} for the random geometric graph on $(X_i)_{i=1}^{t_1}$ with radius of connectivity $2r$, $A$ holds whp.
Note that conditioning on the positions of the points $(X_i)_{i=1}^{t_1}$ without revealing their indices suffices to determine the presence of event $A$.
Moreover, upon such conditioning, the indices of the said points remain uniformly distributed.
Thus, conditionally on $A$, one can find a set $S\subseteq \{X_i: i\in [t_1]\}$ containing $0.4 t_1$ partner pairs $X_{2i-1},X_{2i}$ of isolated points in $G_{t_1}$ such that the balls $B(X_s,r)$ with $X_s\in S$ are pairwise disjoint.
Furthermore, denote by $S' \subseteq S$ the set of all partner pairs $X_{2i - 1}, X_{2i} \in S$ such that exactly one of these points is isolated in $G_{t_0}$.
Then, the expected number of pairs in $S'$ is given by
\begin{align*}
0.4 t_1 \mathbb P(\text{only one of }X_{2i-1} \text{ and }X_{2i} \text{ is isolated in }G_{t_0})
&= 0.4 t_1\cdot 2(1-\theta_d r^d)^{t_0-t_1} \bigg(1-\bigg(\frac{1-2\theta_d r^d}{1-\theta_d r^d}\bigg)^{t_0-t_1}\bigg)\\
&= (0.8+o(1)) t_1 \e^{-\theta_d r^d(t_0-t_1)} = \omega(\log\log(1/r)).
\end{align*}
A simple second moment computation ensures that whp $|S'|\ge \log\log(1/r)$; call this event $B$. 
Moreover, conditionally on the fact that a partner pair is in $S'$, each point in this pair has probability $1/2$ to be the isolated one in $G_{t_0}$, and this choice is done independently for different pairs. 
Based on this observation, for every choice function $f$, conditionally on the event $B$, the probability that the set $\mathbf{X}_t(f)$ for $2t \ge t_1$ contains a point from every partner pair which is not isolated in $G_{t_0}$ is at most $2^{-\log\log(1/r)} = o(1)$, as desired.
\end{proof}

We turn to the proof of \ref{item:A1}--\ref{item:A2}.
The construction of a consecutive choice function $g$ is the same (with a single minor modification) for each of the parts \ref{item:A1}--\ref{item:A2} and for every $k\ge 1$. 
It is divided into two stages. 
First, for every $t\le t_2 := 2c^d r^{-d} \log \log(1/r)$ with a suitably large constant $c$, we set $g(\bX_{t-1},X_{2t-1}) = 1$, where we write $\bX_t$ for $\bX_t(g)$ for short to the end of this section.
In this way, we start by embedding $t_2$ independent uniformly distributed points. Up to this stage, the distribution of the process is the same as the one of the 1-choice process considered in \Cref{section:1-choice}.

To describe the second stage, we recall the tessellation $\cT$ of $\mathbb T^d$ into cubes of side length $r/c$, the auxiliary graph $\cG$ with vertex set $\cT$, and the definition of an $M$-full and $M$-nonfull cube from \Cref{subsection-hamiltonicity}.
Here, we choose $c$ at least as large as in \Cref{subsection-hamiltonicity}, making sure that the Euclidean distance between any pair of points in cubes sharing a corner is at most $r$.
Also, we choose $M\ge k$ in the proof of \ref{item:A1}, and $M>2U+2+2(2c+1)^d$ for the purposes of \ref{item:A2}.
Denote by $\cN\subseteq \cT$ the set of $M$-nonfull cubes in the graph $G(\mathbf{X}_{t_2})$. In particular,
\begin{equation}\label{eq:cN}
\bbE |\cN| = \sum_{i=0}^{M-1} \bigg(\frac{c}{r}\bigg)^{d} \binom{t_2}{i} \bigg(\frac{r}{c}\bigg)^{id} \bigg(1-\bigg(\frac{r}{c}\bigg)^d\bigg)^{t_2-i} = \Theta\left(\frac{(\log\log (1/r))^{M-1}}{r^d(\log (1/r))^2} \right) = o\bigg(\frac{|\cT|}{\log(1/r)}\bigg).    
\end{equation}
Again, a straightforward second moment computation implies that whp $|\cN| \in [\bbE |\cN|/2, 2\bbE |\cN|]$.

Define
\[\Lambda = \bigcup_{q\in \cN} q \]
to be the union of the cubes in $\cN$, and denote by $\Lambda_{2c}$ the $2c$-blow-up of the corresponding set of cubes in $\cT$. In the rest of the proof, we condition on the point set $\mathbf{X}_{t_2}$ and on the event $|\cN| \in [\bbE |\cN|/2, 2\bbE |\cN|]$.

To finish the construction of the choice function $g$, for every $t > t_2$, we define $g(\bX_{t-1}, X_{2t-1}) = 1$ if $X_{2t-1}$ is in $\Lambda$, and $g(\bX_{t-1}, X_{2t-1}) = 2$ otherwise.
We show that $g$ satisfies each of \ref{item:A1}--\ref{item:A2} for suitably large $C = C(d)$.
Roughly speaking, the reason behind this phenomenon is that most cubes are already $M$-full. 
Similarly to the case of the 1-choice process, in order to reach a $k$-connected geometric graph (in the case of \ref{item:A1}) and a Hamiltonian geometric graph (in the case of \ref{item:A2}), it will be enough to ensure these properties in the region $\Lambda_{2c}$.
By using that $|\cN|$ is rather small compared to $|\cT|$, most partner pairs containing a vertex landing in $\Lambda_{2c}$ actually contain a single such vertex. 
In particular, upon restriction on $\Lambda_{2c}$, the 1-choice process and the online 2-choice process with choice function $g$ `run with approximately the same speed' inside $\Lambda_{2c}$.
We formalise the latter statement by constructing a coupling between the two processes, and ensuring \ref{item:A1}--\ref{item:A2} based on \Cref{thm:1}.

For every $t\ge 1$, we denote by $X_t^g$ the last point in $\bX_t$, that is, the $t$-th point selected by the choice function $g$. The promised coupling is constructed as follows. 
Define $\cI$ as the set of indices $t\ge t_2+1$ such that at least one of $X_{2t-1}$ and $X_{2t}$ lands in $\Lambda_{2c}$, and set $\cJ = \{t_2+1,t_2+2,\ldots\}\setminus \cI$. 
We construct a sequence of independent random points $\mathbf{Y} = (Y_i)_{i=1}^{\infty}$. First, for every $i\in [t_2]$, set $Y_i = X_{2i-1}$.
Then, for every $t\ge t_2+1$, we sample a Bernoulli random variable $B_t$ with success probability $p:=\text{Vol}_d(\Lambda_{2c})$ which is independent of the history of the process. 
If $B_t=1$, select the smallest index $i\in \cI$ such that $X_i^g$ is not yet coupled with the process $\mathbf{Y}$ and set $Y_t = X_i^g$.
If $B_t=0$, select the smallest index $j\in \cJ$ such that $X_j^g$ is not yet coupled with the process $\mathbf{Y}$ and set $Y_t = X_j^g$.

\begin{claim}\label{cl:coupleXY}
$\mathbf{Y}$ is a sequence of independent random variables uniform over $\mathbb T^d$.
\end{claim}
\begin{proof}
For every $t\in [t_2]$, the fact that $Y_t$ is uniform over $\mathbb T^d$ and independent from the history of the process is evident.
For every $t\ge t_2+1$, the statement follows from the fact that, with probability $p$, $Y_t$ is uniformly distributed over the region $\Lambda$ independently of the history of the process and, with probability $1-p$, $Y_t$ is uniformly distributed over $\mathbb T^d\setminus \Lambda$ independently of the history of the process.
\end{proof}

We recall the hitting times $\taukcone{k}$ defined in the introduction and adapt them to the process $\mathbf{Y}$.
Denote by $T$ the number of Bernoulli variables among $(B_t)_{t=t_2+1}^{\taukcone{k}}$ equal to 1: in particular, since $G((Y_t)_{t=1}^{\taukcone{2}})$ is typically Hamiltonian, when $k=2$, whp $T$ is the number of points which landed in the region $\Lambda_{2c}$ until the hitting time for Hamiltonicity of the 1-choice process.
Also, denote by $T_{\cI}$ the $T$-th smallest element in the set $\cI$.

\begin{lemma}\label{lem:hit}
For every $k\ge 1$, whp each of the graphs $G(\mathbf{X}_t)$ for $t\ge T_{\cI}$ is $k$-connected.
\end{lemma}
\begin{proof}
First, we show that it is enough to justify that whp $k$-connectivity holds at each time $t\ge T_{\cI}$ for the geometric graph on $\mathbf{X}_t$ restricted to every (topologically) path-connected component of the region $\Lambda_{2c}$.
Indeed, suppose that this is the case and denote by $\Lambda^1,\ldots,\Lambda^k$ the path-connected components of $\Lambda_{2c}$, and by $R^1,\ldots,R^\ell$ the path-connected components of $\mathbb T\setminus \Lambda_{2c}$.
Then, by our choice of $c$, $M\ge k$ and the $M$-fullness of the cubes outside $\Lambda$, each of the graphs induced by $R^1,\ldots,R^{\ell}$ is $k$-connected.
Moreover, by using that every vertex in $R^i_c$ has at least $k$ neighbours in $R^i$ for each $i\in [\ell]$ and applying \Cref{cl:Hav}, each of the graphs induced by $R^1_c,\ldots,R^{\ell}_c$ is also $k$-connected. 
Since the cubes at $\ell_{\infty}$-distance $c$ from the boundary of $\Lambda_{2c}$ are $M$-full, 
the entire graph $G(\mathbf{X}_t)$ can be constructed by consecutively `gluing' $\Lambda^1,\ldots,\Lambda^k,R^1_{c},\ldots,R^\ell_{c}$ in some order so that every new graph intersects the currently constructed union along (one or several) $k$-connected subgraphs.
The latter property is sufficient to deduce the $k$-connectivity of $G(\mathbf{X}_t)$ by Menger's theorem, see e.g.\ Chapter~3.3 in \cite{Die05}.

Now, fix a path-connected component $\Lambda^i$ of $N_{2c}$. Note that, thanks to the constructed coupling of the 1-choice and the online 2-choice processes following $g$, the points landing in $\Lambda^i$ induce a $k$-connected subgraph of $G(\mathbf{X}_t)$ for all $t\ge T_{\cI}$ if and only if they induce a $k$-connected subgraph of $G((Y_j)_{j=1}^t)$ for all $t\ge \taukcone{k}$. We show that the latter event holds conditionally on the whp statement of \Cref{thm:1}.

Fix any $t\ge \taukcone{k}$ and a subset $S$ of $k-1$ points among $(Y_j)_{j=1}^t$ landing in the region $\Lambda^i$. Moreover, fix any two points $Y',Y''$ among $(Y_j)_{j=1}^t\setminus S$ in the region $\Lambda^i$ and suppose that $Y',Y''$ belong to distinct connected components within $\Lambda^i$ after deleting the points in $S$.
By $k$-connectivity of $G((Y_j)_{j=1}^t)$, for any point $Z$ among $(Y_j)_{j=1}^t\setminus S$ outside $\Lambda^i$, each of $Y'$ and $Y''$ connects to $Z$ via a path in $G((Y_j)_{j=1}^t)\setminus S$.
In particular, each of $Y'$ and $Y''$ connects via a path in $G((Y_j)_{j=1}^t)\setminus S$ to a vertex in $\Lambda^i\setminus \Lambda$.
However, the region $\Lambda^i\setminus \Lambda$ is connected and contains only $M$-full cubes. 
In particular, since $M\ge k$, it induces a $k$-connected graph. Thus, the fact that $Y',Y''$ simultaneously connect to points in $\Lambda^i\setminus \Lambda$ means that they actually remain connected within $\Lambda^i$ without going through $S$. 
This implies that the graph induced by the points among $(Y_j)_{j=1}^t$ landing in the region $\Lambda^i$ is $k$-connected and concludes the proof.
\end{proof}

\begin{lemma}
For $k = 2$, whp each of the graphs $G(\mathbf{X}_t)$ for $t\ge T_{\cI}$ is Hamiltonian.
\end{lemma}
\begin{proof}
Recall properties \ref{prop:cG} and \ref{prop:G} from \Cref{subsection-hamiltonicity}. 
From the construction of the coupling of $\bX$ and $\bY$, it follows that the restrictions of $G(\bX_{T_{\cI}})$ and $G(\bY_{T_{\cI}})$ to $\Lambda_{2c}$ coincide. 
In particular, if the whp statement of \Cref{lem:3} holds for $G(\bY_{T_{\cI}})$, then it also holds for $G(\bX_{T_{\cI}})$. Consequently, \ref{prop:cG} holds whp for $G(\bX_{T_{\cI}})$, and for each of $(G(\bX_t))_{t > T_{\cI}}$ by \Cref{rem:lem:3}. Property \ref{prop:G} follows from \Cref{lem:hit} for $k = 2$.
\end{proof}

It remains to estimate the value of $T_{\cI}$. Fix any $\eps > 0$ and set
\[L^-_\eps = d\frac{r^{-d}\log(1/r)}{\theta_d} + (1-\eps)k \frac{r^{-d}\log\log(1/r)}{\theta_d}\quad \text{and}\quad L^+_\eps = d\frac{r^{-d}\log(1/r)}{\theta_d} + (1+\eps)k \frac{r^{-d}\log\log(1/r)}{\theta_d}.\]
On the one hand, by~\eqref{eq:expZ} and \Cref{lem:conc}, for every $\eps > 0$, we know that whp
\begin{equation}\label{eq:taukcone}
\taukcone{k}\in [L^-_{\eps}, L^+_{\eps}].
\end{equation}
Then, conditionally on $|\cN|$ and assuming~\eqref{eq:taukcone}, we have that $T$ is stochastically dominated by a binomial random variable with parameters $L^+_\eps-t_2$ and $p=\text{Vol}_d(\Lambda_{2c})$.
Moreover, since $\text{Vol}_d(\Lambda_{2c})$ is at most a constant factor away from $\text{Vol}_d(\Lambda) = |\cN|/|\cT|$ and thanks to~\eqref{eq:cN}, $p = \omega((\log(1/r))^{-2})$. 
Thus, Chernoff's bound implies that, for every $\eps > 0$, whp $T \le p(L^+_{2\eps}-t_2)$.
Conditionally on the latter event, by using that a partner pair contains a point in $\Lambda_{2c}$ with probability $2p-p^2$, Chernoff's inequality implies that whp
\[|\cI\cap [t_2+1,(L^+_{4\eps}+t_2)/2]|\ge (2p-p^2)\frac{L^+_{3\eps}-t_2}{2}\ge T\quad\Longrightarrow\quad T_{\cI}\le \frac{L^+_{4\eps}+t_2}{2}.\]
In particular, with high probability, for every $t$ such that $2t\ge L^+_{4\eps}+t_2$, $G(\mathbf{X}_t)$ is $k$-connected (and Hamiltonian in the case $k = 2$), which is enough to deduce \ref{item:A1} and \ref{item:A2} with $C = (2c^d+2)/\theta_d$.

\section{Concluding remarks}\label{section:conclusion}

Several interesting research directions stem from our work. 
Perhaps the most relevant questions concern the power of choice setting when $b\ge 3$ choices are proposed to the agent. 
While we believe that the tools we developed would be useful for the analysis of the model when $b\ge 3$, this further step would undoubtedly require an additional effort due to the growth of far components.

Following a line of research concerning the behaviour of the giant component in the Achlioptas process~\cite{ADS09,BF01,MS15,RW12,SW07}, one could wonder if linear-sized components may appear around the sharp threshold for existence of a giant component $t_g = t_g(r)$ in $(G_t)_{t=0}^{\infty}$ as $r\to 0$.
For $\delta>0$ and small $\eps=\eps(\delta)>0$, by privileging points arriving in the rectangle $[0,\eps]^d$ (seen as a subset of $\mathbb T^d$) over their partners, the agent may force the existence of components in $G_t$ with linear number of vertices as long as $t\ge (1+\delta)t_g$ (as also observed in~\cite{MS15}).
Analysing the optimal proportion of vertices in the largest component throughout the process in each of the online and offline 2-choice setting is an open problem of interest.

Finally, we remark that exhibiting analogous results to \Cref{thm:1,thm:2,thm:3} for perfect matchings (up to slightly annoying parity considerations) follows from a simplified version of the presented approach for Hamiltonicity. 
We believe that similar results could hold more generally for $H$-factors for graphs $H$ of constant order: while the minimum degree assumption is insufficient at this level of generality, the natural sufficient condition would be that every vertex in the graph is contained in at least one copy~of~$H$.

\paragraph{Acknowledgements.} 
The authors would like to thank Marcos Kiwi and Dieter Mitsche for several comments on the manuscript,
James Martin for directing our attention to the reference~\cite{VDK96}, and Matija Pasch for a useful related question.
This work was supported by the Austrian Science Fund (FWF) grant number 10.55776/ESP624. For open access purposes, the authors have applied a CC BY public copyright license to any author-accepted manuscript version arising from this
submission.

\bibliographystyle{abbrv}  
\bibliography{Bibliography}

\end{document}